\documentclass[12pt, a4paper]{article}

\usepackage{bm,mathtools}
  \usepackage{amsfonts}
  \usepackage{bm}
  \usepackage{stmaryrd}
\usepackage[deletedmarkup=xout,defaultcolor=purple,authormarkup=none,final]{changes}
\definechangesauthor[name=catB, color=blue]{catB}
\definechangesauthor[name=catR, color=red]{catR}
\definechangesauthor[name=vD, color=purple]{vD}

\usepackage[a4paper, margin=2cm]{geometry}
  \usepackage{enumitem}
  \newtheorem{theorem}{Theorem}[section]
  \newtheorem{lemma}[theorem]{Lemma}
 \newtheorem{proposition}[theorem]{Proposition}

  \newtheorem{remark}[theorem]{Remark}
  \newenvironment{proof}{{\it Proof}. }{\hfill\END\\[0.5ex]}

\DeclareMathOperator{\vcurl}{\overrightarrow{\mathrm{curl}}}
\DeclareMathOperator{\curl}{curl}
\DeclareMathOperator{\divv}{div}

\DeclareMathOperator{\SL}{SL}
\DeclareMathOperator{\DL}{DL}

\newcommand{\D}{\mathrm{D}}
\DeclareMathOperator{\V}{V}
\DeclareMathOperator{\W}{W}
\DeclareMathOperator{\K}{K}
\DeclareMathOperator{\HH}{H}
\DeclareMathOperator{\I}{I}
\DeclareMathOperator{\Hes}{Hess}
\DeclareMathOperator{\sign}{sign}

\usepackage{tikz}
\usetikzlibrary{arrows,decorations.markings,arrows.meta}

\usepackage{pgfplots}
\usepackage{makecell}

\usepackage{amsmath}
\usepackage{amssymb}

\usepackage[normalem]{ulem}

\newcommand{\be}{\begin{equation}}
\newcommand{\ee}{\end{equation}}

\newcommand{\x}{{\bm{x}}}
\newcommand{\y}{{\bm{y}}}

\newcommand{\nnn}{\bm{n}}
\newcommand{\ttt}{\bm{t}}

\newcommand{\END}{\hfill$\Box$}

\title{Robust boundary integral equations for the solution of elastic scattering problems via Helmholtz decompositions}
\date{\today}
\author{V\'{i}ctor Dom\'{i}nguez\thanks{Dep. Estadística, Matemática e Inform\'atica, Universidad P\'{u}blica de Navarra. Campus de Tudela 31500 - Tudela, Spain, e-mail: victor.dominguez@unavarra.es.}\and Catalin Turc\thanks{  Department of
Mathematical Sciences, New Jersey  Institute of Technology,
Univ. Heights. 323 Dr. M. L. King Jr. Blvd, Newark, NJ 07102, USA, e-mail: catalin.c.turc@njit.edu.}}
\allowdisplaybreaks

\pgfplotsset{compat=1.18}

\numberwithin{equation}{section}
 
\begin{document}
\maketitle
\begin{abstract}
 Helmholtz decompositions of the elastic fields open up new avenues for the solution of linear elastic scattering problems via boundary integral equations (BIE) \added[id=catB]{[Dong et al., Mathematics of Computation 90 (2021)]}. The main appeal of this approach is that the ensuing systems of BIE feature only integral operators associated with the Helmholtz equation. However, these BIE involve non standard boundary integral operators that do not result after the application of either the Dirichlet or the Neumann trace to Helmholtz single and double layer potentials. Rather, the Helmholtz decomposition approach leads to BIE formulations of elastic scattering problems with Neumann boundary conditions that involve boundary traces of the Hessians of Helmholtz layer potential. As a consequence, the classical combined field approach applied in the framework of the Helmholtz decompositions leads to BIE formulations which, although robust, are not of the second kind. Following the regularizing methodology introduced in \added[id=catB]{[Boubendir et al., SIAP 75 (2015)]} we design and analyze novel robust Helmholtz decomposition BIE for the solution of elastic scattering that are of the second kind in the case of smooth scatterers in two dimensions. We present a variety of numerical results based on Nystr\"om discretizations that illustrate the good performance of the second kind regularized formulations in connections to iterative solvers.

\end{abstract}

\section{Introduction}

The extension of Boundary Integral Equation (BIE) based discretizations for the numerical solution of acoustic scattering problems (i.e. Helmholtz equations) to their elastic scattering problems (i.e. Navier equations) counterpart is thought to be more or less straightforward. \added[id=catR]{A main challenge in the discretization of BIE formulations of various linear, constant coefficient PDEs is the use of quadratures that resolve the singularities of the associated Green's functions.} Typically, BIE for elastic scattering problems are based on the Navier Green's function, which albeit more complicated than the Helmholtz Green's function (not in the least because it involves two wave-numbers), exhibits the same integrable/weak singularity as the latter. \added[id=catR]{Therefore, the Boundary Integral Operators associated with both PDEs feature weakly singular, singular, and hypersingular integral operators, and quadratures which resolve in a black box manner the singularities of Green's functions and their derivatives are preferred for the numerical discretization of these operators.}  Amongst these types of quadratures we mention the $\delta$-BEM methods~\cite{dominguez2008dirac,dominguez2012fully,dominguez2014nystrom,DoSaSa:2015} and the Density Interpolation Methods~\cite{faria2021general}. On the other hand, other discretizations strategy such as the Kussmaul-Martensen singularity splitting Nystr\"om methods can be extended from the Helmholtz to the Navier setting but require cumbersome modifications~\cite{chapko2000numerical,dominguez2021boundary}.

Recently a new BIE approach to elastic scattering problems shortcuts the need to use Navier Green's functions and relies  entirely on Helmholtz layer potentials~\cite{dong2021highly,lai2019framework}. This approach is based on Helmholtz decompositions of  elastic waves into compressional and shear waves, a manner which reduces the elastic scattering problem, at least in two dimensions, to the solution of two Helmholtz equations coupled by their boundary values on the scatterers. In three dimensions this approach leads to coupling the solution of a Helmholtz scattering problem to that of a Maxwell scattering problem~\cite{dong2022spectral}. These recent contributions~\cite{dong2021highly,lai2019framework} consider only single layer potential representations of Helmholtz fields and elastodynamic fields with Dirichlet boundary conditions on the boundary of the scatterers, and, therefore, the ensuing BIE are not robust for all frequencies. We extended the Helmholtz decomposition approach in two dimensions to Neumann boundary conditions as well as for combined field representations in~\cite{dominguez2022nystrom}. The extension of the Helmholtz decomposition approach to the Neumann case is nontrivial since it gives rise to non standard Boundary Integral Operators (BIOs) which arise from applications of boundary traces to Hessians of single and double layer Helmholtz potentials. More importantly, the BIE Helmholtz decomposition route is  a viable approach to the solution of elastic scattering problems in as much as robust formulations are used. The goal of this paper is to derive and analyze such robust formulations, including some which are of the second kind in the case of smooth scatterers.

The most widely used strategy to deliver robust BIE formulations for the solution of time harmonic scattering problems is the combined field (CFIE) strategy~\cite{BrackhageWerner,BurtonMiller}. The analysis of such formulations relies on the Fredholm theory and the well posedness of Robin boundary value problems in bounded/interior domains. The analysis of the Fredholm property of the BIOs that arise in the combined field strategy in the Helmholtz decomposition framework for elastic scattering problems is a bit more delicate, as it requires the use of lower order terms in the asymptotic expansion in the pseudodifferential sense of the constitutive BIOs. The reason for this more involved analysis is the fact that the principal symbols of these operators (unlike the Helmholtz case) are defective, which was already remarked in~\cite{dong2021highly}. However, the ensuing CFIE feature pseudodifferential operators of order one (Dirichlet) and respectively two (Neumann) and \added[id=catR]{as} such are not ideal for iterative solutions \added[id=catR]{since the spectra of these integral operators spread out to infinity}.

In order to design integral formulations of the second kind for the solution of elastic scattering problems via the Helmholtz decomposition approach we employ the general methodology in~\cite{boubendir2015regularized}.  Using  coercive approximations of Dirichlet to Neumann ({\rm DtN}) operators, we construct certain regularizing operators that are approximations of the operators that map the boundary conditions in the Helmholtz decomposition approach to the Helmholtz Cauchy data on the boundary of the scatterer. The regularizing operators we consider, whose construction relies on the pseudodifferential calculus, are based on either square root Fourier multipliers or Helmholtz BIOs, and are straightforward to implement in existing discretizations methodologies for Helmholtz BIOs. More importantly, we prove that in the case of smooth scatterers the regularized formulations are robust and of the second kind, and we provide numerical evidence about the superior performance of these formulations over the combined field formulations with respect to iterative solvers. Thus, we show in this paper that it is possible to construct BIE with good spectral properties for the solution of the elastic scattering problems in two dimensions via Helmholtz decompositions, making this approach a viable alternative to the recently introduced BIE based on the Navier Green's function~\cite{chaillat2015approximate,chaillat2020analytical,dominguez2021boundary,osti_10185828}.  The extension of this approach to three dimensional elastic scattering problem is currently ongoing.

 The paper is organized as follows: in Section~\ref{setting} we introduce the Navier equations in two dimensions and we present the Helmholtz decomposition approach; in Section~\ref{GradHessCalculations} we present the trace for the gradient and Hessian  of the boundary layer operators for Helmholtz equation. We also review certain properties of Helmholtz BIOs with asymptotic expansions of in the pseudodifferential sense.
In Section~\ref{BIE} we analyze the nonstandard BIOs that arise in the Helmholtz decomposition approach for Dirichlet (subsection \ref{dirP}) and Neumann (subsection \ref{neuP}) boundary conditions. Finally, we present in Section~\ref{NR} a variety of numerical results illustrating the iterative behavior of solvers based on Nystr\"om discretizations of the various BIE formulations of the Helmholtz decomposition approach  in the high frequency regime.

\section{Navier equations}\label{setting}

For any vector function ${\bf u}=(u_1,u_2)^\top:\mathbb{R}^2\to   \mathbb{R}^2$ (vectors in this paper will be always regarded as {\em column} vectors)  the strain tensor in  a linear isotropic and homogeneous elastic medium with Lam\'e constants $\lambda$ and $\mu$  is defined as
\[
 \bm{\epsilon}({\bf u}):=\frac{1}2 (\nabla {\bf u}+(\nabla {\bf u})^\top)
 =\begin{bmatrix}
  \partial_{x_1} u_1& \tfrac{1}2\left(\partial_{x_1} u_2+\partial_{x_2} u_1\right)     \\
   \tfrac{1}2\left(\partial_{x_1} u_2+\partial_{x_2} u_1\right) &
   \partial_{x_2} u_2
     \end{bmatrix}.
\]
The stress  tensor is then given by
\[
 \bm{\sigma}({\bf u}):=2\mu \bm{\epsilon}({\bf u})+\lambda (\divv{\bf u}) I_2
\]
where $I_2$ is the identity matrix of order 2 and the Lam\'e coefficients  \added[id=catB]{$\lambda$  and $\mu$} are assumed to satisfy $\lambda,\ \lambda +2\mu>0$.  The time-harmonic elastic wave (Navier) equation is
\begin{equation}\label{eq:NavD}
 \divv\bm{\sigma}({\bf u})+\omega^2{\bf u}=\mu\Delta {\bf u}+(\lambda+\mu)\nabla(\divv {\bf u})+\omega^2 {\bf u}=0
\end{equation}
where the frequency $\omega\in\mathbb{R}^+$ and the {\rm divergence} operator ${\rm div}$ is applied row-wise. \added[id=vD]{Notice that the gradient operator is taken as a column vector operator for compatibility.}

Considering a bounded domain $\Omega$ in $\mathbb{R}^2$ whose boundary $\Gamma$ is a closed smooth curve, we are interested in solving the impenetrable elastic scattering problem in the exterior of $\Omega$, denoted from now on as $\Omega^+$. That is, we look for solutions of the time-harmonic Navier equation
that satisfy the Kupradze radiation condition at infinity \cite{AmKaLe:2009,kupradze2012three}: if
\begin{equation}\label{eq:Hdecomp1}
\mathbf{u}_p:=-\frac{1}{k_p^2} \nabla \divv \mathbf{u}, \quad \mathbf{u}_s:=\mathbf{u}-\mathbf{u}_p= \vcurl {\rm{curl} \ {\bf u}}
\end{equation}
($\vcurl \varphi :=(\partial_{x_2} \varphi,-\partial_{x_1}\varphi)$, \added[id=catR]{$\curl{\bf u} := \partial_{x_1}u_2-\partial_{x_2} u_1$} are
the vector and scalar curl, or rotational, operator)
with
\begin{equation}\label{eq:ks:kp}
k_p^2:=\frac{\omega^2}{\lambda+2 \mu}, \quad k_s^2:=\frac{\omega^2}{\mu}
\end{equation}
the \added[id=catB]{associated pressure and stress wave-numbers}, then
\[
\frac{\partial \mathbf{u}_p}{\partial \widehat{\bm{x}}}(\bm{x})-i k_p \mathbf{u}_p(\bm{x})=o\left(|\bm{x}|^{-1 / 2}\right),\quad \frac{\partial \mathbf{u}_s}{\partial \widehat{\bm{x}}}(\bm{x})-i k_s \mathbf{u}_s(\bm{x})=o\left(|\bm{x}|^{-1 / 2}\right), \quad
 \widehat{\bm{x}} := \frac{1}{|\bm{x}|}\bm{x}.
\]

On the boundary $\Gamma$ the solution ${\bf u}$ of~\eqref{eq:NavD} satisfies either the Dirichlet boundary condition
\[
  {\bf u} ={\bf f}\quad{\rm on}\ \Gamma
\]
or the Neumann boundary condition
\[
 T{\bf u}:=\bm{\sigma}({\bf u})\nnn=\added[id=catR]{
\lambda(\divv   {\bf u})\ \nnn
+2\mu (\nnn\cdot \nabla) {\bf u} -\mu({\rm curl}\ {\bf u}) \ttt}
={\bf g}\quad{\rm on}\ \Gamma.
\]
Here $\nnn$ is \added[id=vD]{the} unit normal vector pointing outward and
 \[
 \ttt  := -\mathrm{Q}\nnn,\quad \mathrm{Q} :=\begin{bmatrix}
                                                            & 1\\
                                                            -1&
                                                           \end{bmatrix}
\]
the unit tangent field positively (counterclockwise if $\Gamma$ is simply connected) oriented. \added[id=catR]{In the scattering applications we consider in this paper, the Dirichlet and Neumann data correspond the boundary values on $\Gamma$ of an incident field ${\bf u}^{\rm inc}$ which is a smooth solution of the Navier equations in the exterior domain; thus ${\bf f}:=-{\bf u}^{\rm inc}$ and ${\bf g}:=-T\ {\bf u}^{\rm inc}$ on $\Gamma$.}

 In view of \eqref{eq:Hdecomp1}  we can look for the fields ${\bf u}$ in the form
\begin{equation}\label{eq:Hdecomp2}
{\bf u}=\nabla u_p+\vcurl  {u_s}
\end{equation}
where     $u_p$ and $u_s$ are respectively solutions of the Helmholtz equations in $\Omega^+$  with wave-numbers $k_p$ and $k_s$  fulfilling the radioactive, or Sommerfeld, condition at infinity (as consequence of the Kupradze 
condition).

In the case of Dirichlet boundary conditions, by taking the scalar product of the decomposition~\eqref{eq:Hdecomp2} with the orthogonal frame $(\ttt, \nnn)$ on $\Gamma$, it is straightforward to see that $u_p$ and $u_s$ must satisfy the following coupled boundary conditions
\begin{equation}\label{eq:DH}
\left|
\begin{array}{rcll}
\partial_{\nnn}u_p+\partial_{\ttt}u_s&=&-{\bf u}^{\rm inc}\cdot \nnn,& {\rm on}\ \Gamma \\
\partial_{\ttt}u_p-\partial_{\nnn}u_s&=&-{\bf u}^{\rm inc}\cdot \ttt,& {\rm on}\ \Gamma.
\end{array}
\right.
\end{equation}
In the case of Neumann boundary conditions, the same approach described above combined with the identities
\begin{eqnarray*}
{\bm{\sigma} [\nabla u_p]}  &=& 2\mu \Hes u_p+\lambda \Delta u_p\  \added[id=catR]{I_2}=
  2\mu \Hes u_p-\lambda k_p^2u_p\ \added[id=catR]{I_2}\\
{\bm{\sigma} [{\protect \overrightarrow{\rm curl}}\ u_s]} &=& -(2\mu \Hes u_s-\mu  \Delta u_s I_2){{\rm Q}}=-(2\mu \Hes u_s+\mu k^2_su_s{I_2}){{\rm Q}}
\end{eqnarray*}
where $\Hes:=\nabla\nabla^\top$ is the Hessian  matrix operator, imply that  $u_p$ and $u_s$ must satisfy the following coupled boundary conditions
\begin{equation}\label{eq:NH}
\left|\!\!
\begin{array}{rclrcl}
 2\mu \nnn^\top \Hes u_p{\bm n}-\lambda k_p^2 u_p&\!\!\!\!+\!\!\!\!&2\mu
 \ttt^\top \Hes u_s{\bm n} &\!\!\!=\!\!\! & -T {\bf u}^{\rm inc}\cdot{\bm n},&\!\!{\rm on}\ \Gamma \\
 2\mu \ttt^\top \Hes u_p{\bm n} &\!\!\!\!-\!\!\!\!& 2\mu \nnn^\top \Hes u_s{\bm n}-\mu k_s^2 u_s &\!\!\!=\!\!\!
 & -T {\bf u}^{\rm inc}\cdot{\bm t},&\!\!{\rm on}\ \Gamma.\\
\end{array}
\right.
\end{equation}
The goal of this paper is to develop robust BIE \added[id=vD]{(Boundary Integral Equation)} formulations for the solution of elastic scattering problems based on the Helmholtz decomposition approach in connection with the systems of boundary conditions~\eqref{eq:DH} and respectively~\eqref{eq:NH}. The BIE we develop use Helmholtz potentials, for which reason we review certain properties of those next section.

\begin{remark}\label{remark:curves:2pi} Throughout this article
we will assume that $\Gamma$ is of length $2\pi$. This simplifies some of the following expressions. The modifications needed to cover the case of arbitrary length curves essentially consist of replacing the wave-number(s) $k$ (as well as $k_p$, $k_s$ and its complexifications $\widetilde{k}_p$ and $\widetilde{k}_s$ \added[id=vD]{that we will introduce later}) by $ Lk/(2\pi)$, its characteristic length. We can see this either by adapting the analysis or by a simple scaling argument.
\end{remark}

\section{Helmholtz BIOs, gradient and Hessian  calculations}\label{GradHessCalculations}

For a given \added[id=catB]{complex wave-number $k$ such that $\Im{k}\geq 0$} and a functional density $\varphi$ on the boundary $\Gamma$ we define the Helmholtz single and double layer potentials in the form
\[
 \SL_k[\varphi](\x):=\int_\Gamma \phi_k(\x-\y)\varphi(\y) {{\rm d}\y},\quad
 \DL_k[\varphi](\x):=\int_\Gamma \frac{\partial \phi_k(\x-\y)}{\partial \nnn(\y)}\varphi(\y) {{\rm d}\y},\ \x\in\mathbb{R}^2\setminus\Gamma.
\]
\added[id=catR]{
Here $\phi_k(\x) = \frac{i}4 H_0^{(1)}(k|\x|)$, with $H_0^{(1)}$ being the zeroth order Hankel function of the first kind  {\cite[Ch. 9]{AbraSteg72}},
is the fundamental solution of the Helmholtz equation.}

The four BIOs \added[id=vD]{(Boundary Integral Operators)} of the Calder\'on's calculus associated with the Helmholtz equation are defined by applying the exterior/interior Dirichlet and Neumann traces on $\Gamma$ (denoted in what follows by $\gamma^+/\gamma^-$ and $\partial_{\nnn}^+/\partial_{\nnn}^-$ respectively) to the Helmholtz single and double layer potentials defined above via the classical relations
\cite{hsiao2008boundary,mclean:2000,Saranen}
\begin{equation}\label{eq:traces}
\begin{aligned}
\gamma^\pm  \SL_k[\varphi] &=\V_{k}[\varphi], \quad & \partial_{\nnn}^\pm  \SL_k[\varphi] &=\mp\frac{1}{2}\varphi +\K_{k}^\top[\varphi],  \\
\gamma^\pm \DL_k[\varphi] &=\pm\frac{1}{2}\varphi+\K_{k}[\varphi], \quad & \partial_{\nnn}^\pm \DL_k[\varphi] &=\W_{k}[\varphi],
\end{aligned}
\end{equation}
where, for $\x\in\Gamma$, the four Helmholtz BIOs are defined as
\[
\begin{aligned}
\V_{k}[\varphi](\x) &:=\int_\Gamma \phi_k(\x-\y)\varphi(\y) {{\rm d}\y},& \quad
\K_k[\varphi](\x) &:= \int_\Gamma \frac{\partial \phi_k(\x-\y)}{\partial \nnn(\y)}\varphi(\y) {{\rm d}\y},\\
\K_k^\top[\varphi](\x)  &:= \int_\Gamma \frac{\partial \phi_k(\x-\y)}{\partial \nnn(\x)}\varphi(\y) {{\rm d}\y},&\quad
\W_k[\varphi](\x) &:=  \mathrm{f.p.}\int_\Gamma \frac{\partial^2 \phi_k( \x-\y )}{\partial \nnn(\x)\,\partial \nnn(\y) }\varphi(\y) {{\rm d}\y},\
\end{aligned}
\]
(``f.p.'' stands for {\em finite part} since the kernel of the operator is strongly singular)
are respectively the single layer, double layer, adjoint double and hypersingular operator. The well-know jump relations, which can be extended to merely Lipschitz curves, are  then simple byproducts:
\begin{equation}\label{eq:jumps}
\begin{aligned}
\added[id=vD]{\added[id=vD]{\llbracket}} \gamma \SL_k[\varphi]  \added[id=vD]{\added[id=vD]{\rrbracket}} &:=\gamma^+ \SL_k[\varphi] - \gamma^-\SL_k[\varphi] =0, \ & \added[id=vD]{\llbracket}\partial_{\nnn}   \SL_k[\varphi]  \added[id=vD]{\rrbracket}  &:=\partial_{\nnn}^+ \SL_k[\varphi] - \partial_{\nnn}^- \SL_k[\varphi] =- \varphi, \\
\added[id=vD]{\llbracket}  \gamma   \DL_k[\varphi]  \added[id=vD]{\rrbracket}  &:=\gamma ^+ \DL_k[\varphi] - \gamma ^- \DL_k[\varphi] = \varphi,\  &
\added[id=vD]{\llbracket}\partial_{\nnn}   \DL_k[\varphi]  \added[id=vD]{\rrbracket}  &:=\partial_{\nnn}^+ \DL_k[\varphi] - \partial_{\nnn}^- \DL_k[\varphi] = 0.
\end{aligned}
\end{equation}
\added[id=catB]{Here in what follows we will denote the action of Helmholtz layer potentials and their associated BIOs on functional densities defined on the boundary $\Gamma$ by bracketing the densities.} Let $H^s(\Gamma)$ be the Sobolev space on $\Gamma$ of order $s$. It is well known then that, for smooth closed boundaries $\Gamma$,
\[
\K_k,\  \K_k^\top :H^s(\Gamma)\to
H^{s+3}(\Gamma),\quad
\V_k:H^s(\Gamma)\to H^{s+1}(\Gamma),\quad
\W_k:H^{s+1}(\Gamma)\to H^{s}(\Gamma),
\]
that is, the Helmholtz BIOs $\K_k, \K_k^\top$ and $\V_k, \W_k$ are pseudodifferential  operators of orders $-3$, $-1$ and $1$ respectively~\added[id=vD]{\cite{hsiao2008boundary,Saranen,mclean:2000,dominguez2012fully}}. We will express these mapping properties of the Helmholtz BIOS in what follows in the succinct form
\[
 \K_k,\ \K_k^{\top}\in\mathrm{OPS}(-3),\quad
 \V_k  \in\mathrm{OPS}(-1),\quad
 \W_k\in\mathrm{OPS}(1).
 \]

Having reviewed properties of Helmholtz BIOs that will be useful in what follows, we turn our attention to connecting the systems of boundary conditions~\eqref{eq:DH} and~\eqref{eq:NH} to the four Helmholtz BIOs of the Calder\'on calculus. In particular, since the boundary conditions~\eqref{eq:NH} feature the Hessian operators, more involved calculations are necessary to express them via the classical four BIOs. We detail in what follows these derivations which are largely based on integration by parts techniques akin to those leading to Maue's identity~\eqref{eq:maue:helmholtz}.

We will follow from now on the convention that for a vector function $\bm{\varphi} =(\varphi_1,\varphi_2)^\top$ and a pseudodifferential operator $\mathrm{R}$,
 \begin{equation}\label{eq:notation}
  \mathrm{R}\bm{\varphi}:=(\mathrm{R} {\varphi}_1,\mathrm{R} {\varphi}_2)^\top.
 \end{equation}
The following identities will be used in the proofs of the next results
\begin{equation}\label{eq:ttnn}
 \ttt \ttt^\top+ \nnn \nnn^\top = I_2,\qquad
 \ttt \ttt^\top- \nnn \nnn^\top = I_2 - 2 \nnn \nnn^\top.
\end{equation}
(Notice in pass that the second one is just a reflection, a Householder matrix, about the $\ttt$ direction.) Finally, \deleted{the} we define the \added[id=vD]{signed} curvature $\kappa$  in the following manner
\begin{equation}\label{eq:def:kappa}
  \added[id=catR]{\partial_{\ttt} \ttt=-\kappa\ \nnn,\quad \partial_{\ttt}\nnn=\kappa\ \ttt.}
 \end{equation}


We will study now the exterior trace for the gradient of the single and double layer potentials. Although these results (Propositions \ref{proposition:grad:SL} and \ref{proposition:grad:DL}) can be certainly found in the literature, we cite for example \cite{Kress,qbx:2013,kolm2003quadruple}, we will prefer to present their proofs here for the sake of completeness and to prepare both the statements and the proofs themselves of the corresponding results for the Hessian  matrix of the layer boundary potentials for which the authors were not able to find detailed proofs.  Clearly very similar results can be derived for the {\em interior} traces with the same techniques but since our aim is the study of exterior problems, we omit the study of this case, leaving it as (a simple) exercise for the reader.

\begin{proposition} \label{proposition:grad:SL}For any smooth closed curve sufficiently smooth,
\begin{equation}\label{eq:01:proposition:grad:SL}
\gamma^+ \nabla {\rm SL}_k [\varphi] = -\frac{1}2 \varphi\, \nnn + \nabla \V_k[\varphi]
\end{equation}
with \added[id=vD]{for $\bm{x}\in\Gamma$}
\begin{subequations}\label{eq:02:proposition:grad:SL}
\begin{eqnarray}
\nabla \V_k [\varphi](\x) &:=&  {\rm p.v.}\int_\Gamma (\nabla\phi_k)(\x-\y)\varphi(\y) \,\mathrm{d}\y\nonumber\\
&=&\nnn(\x)\ \K_k^\top[\varphi](\x)
+\ttt(\x) \partial_{\ttt} \V_k[\varphi](\x)
\label{eq:02b:proposition:grad:SL}\\
&=&
 \V_k[\partial_{\ttt} (\varphi\,\ttt)](\x) -\K_k [\varphi\,\nnn](\x)  \label{eq:02a:proposition:grad:SL}\nonumber\\
 &=& -\K_k [\varphi\,\nnn](\x)
 + \V_k [\partial_{\ttt} \varphi\,\ttt ](\x)
- \V_k[\kappa \varphi\,\nnn](\x)
\label{eq:02c:proposition:grad:SL}
\end{eqnarray}
\end{subequations}
where {\rm p.v.} stands for ``principal value'' of the integral.
 \end{proposition}

 \begin{proof}
For $\x\in\Omega^+$ sufficiently close to $\Gamma$ extend
 \[
  \nnn(\x) = \nnn(\widehat{\x}), \quad
  \ttt(\x) = \ttt(\widehat{\x}), \qquad \x = \widehat{\x}+\varepsilon \nnn(\widehat{\x}), \quad \widehat{\x}\in \Gamma
 \]
 with $\varepsilon>0$ sufficiently small.

 By \eqref{eq:ttnn}
 \[
\nabla_{\x} \phi_k({\x}-{\y})=  \left(\nabla_{\x} \phi_k({\x}-{\y})\cdot \nnn(\x)\right)\nnn(\x) + \left(\nabla_{\x} \phi_k({\x}-{\y})\cdot\ttt(\x)\right)\ttt(\x)
 \]
and therefore
 \begin{equation}\label{eq:02bb:proposition:grad:SL}
 \begin{aligned}
\nabla  {\rm SL}_k& [\varphi](\x)
= & \partial_{\nnn} \SL_k [\varphi ](\x)\nnn(\x)+(\nabla_\x  \SL_k  [\varphi ](\x)\cdot  {\ttt}(\x)) {\ttt}(\x).
\end{aligned}
\end{equation}
Alternatively,
\[
 \nabla_{\x} \phi_k({\x}-{\y})=   - \left(\nabla_{\y} \phi_k({\x}-{\y})\cdot\ttt(\y)\right)\ttt(\y)-\left(\nabla_{\y} \phi_k({\x}-{\y})\cdot \nnn(\y)\right)\nnn(\y)
\]
which implies, after integration by parts,
\begin{align}
\nabla & {\rm SL}_k[\varphi](\x) \nonumber\\
&= \int_{\Gamma }  \phi_k({\x}-{\y})  \partial_{\ttt(\y)}  \left(\ttt (\y)  \varphi(\y)\right){\rm d}\y
-\int_{\Gamma }\left(\nabla_{\y} \phi_k({\x}-{\y})\cdot \nnn(\y)\right)\nnn(\y)\varphi(\y){\rm d}\y
\nonumber \\
&=
{\rm SL}_k [\partial_{\ttt}(\varphi\,\ttt)](\x)
-{\rm DL}_k[\varphi\,\nnn](\x)\label{eq:02cc1:proposition:grad:SL}
\\
&=
{\rm SL}_k[\partial_{\ttt}\varphi\, \ttt](\x)-
{\rm SL}_k[\kappa\varphi\,\nnn](\x)
-{\rm DL}_k[\varphi\,\nnn](\x).
\label{eq:02cc2:proposition:grad:SL}
  \end{align}
  Clearly,  \eqref{eq:02b:proposition:grad:SL} and
\eqref{eq:02c:proposition:grad:SL} follows now from the jump relations \eqref{eq:traces} applied to \eqref{eq:02bb:proposition:grad:SL} and \eqref{eq:02cc1:proposition:grad:SL}-\eqref{eq:02cc2:proposition:grad:SL} respectively.

 \end{proof}

 Recall the relations
\begin{equation}\label{eq:2.19}
\begin{aligned}
  \overrightarrow{{\rm curl}}\,{u}\cdot\nnn &= \nabla u \cdot \ttt=\partial_{\ttt} u, &
  \overrightarrow{{\rm curl}}\,{u}\cdot\ttt = -\nabla u \cdot \bm{n}=-\partial_{\nnn} u,\\
  \  \nabla (\divv {\bf u}) &= \Delta{\bf u} +\overrightarrow{{\rm curl}}\,\curl {\bf u}.
\end{aligned}
\end{equation}
($\Delta{\bf u}$ is just, according to the convention \eqref{eq:notation}, the vector laplacian) and
\begin{equation}\label{eq:2.20}
\curl({u {\bf v}})= \nabla u \cdot \mathrm{Q} {\bf v}+ u \curl  {\bf v}.
\end{equation}

\begin{lemma}\label{lemma:2.1}
It holds
\begin{equation}\label{eq:lemma:2.1:03}
 \nabla {\rm DL}_k [\varphi] = k^2{\rm SL}_k[\varphi\,\nnn]  + \overrightarrow{{\rm curl}}\,{\rm SL}_k[\partial_{\ttt}\varphi].
\end{equation}
\end{lemma}
\begin{proof}
Notice first
\[
 \nabla_{\x}((\nabla_{\y}\phi_k(\x-\y))\cdot \nnn(\y)) = -\nabla_{\x}\left(\divv_{\x}(\phi_k(\x-\y)  \nnn(\y))\right).
\]
The third identity in \eqref{eq:2.19} implies  then
\begin{eqnarray*}
 \nabla {\rm DL}_k [\varphi](\x)&=& -\int_{\Gamma} \Delta_{\x} \phi_k(\x-\y)  \nnn(\y)\varphi(\y)\,\mathrm{d}\y
 \\
 &&- \int_{\Gamma} \overrightarrow{{\rm curl}}\,_{\x}(\curl_{\x} \phi_k(\x-\y)  \nnn(\y))\varphi(\y)\,\mathrm{d}\y
 \\
 %
  &=&   k^2 \SL_{k} [\varphi\nnn](\x)  - \int_{\Gamma} \overrightarrow{{\rm curl}}\,_{\x}  \partial_{\ttt(\y)}\phi_k(\x-\y)  \varphi(\y)\,\mathrm{d}\y
\end{eqnarray*}
where we have made use that $\phi_k$ is the fundamental solution of the Helmholtz equation along with the first identity in \eqref{eq:2.19}. The result follows now by integration by parts in the last integral.
\end{proof}



\begin{proposition}\label{proposition:grad:DL}
 It holds
\[
\gamma^+ \nabla {\rm DL}_k [\varphi] =  \frac{1}2 \partial_{\ttt} \varphi\, \ttt + \nabla \K_k[\varphi]
\]
where \added[id=vD]{for $\bm{x}\in\Gamma$}
\begin{subequations}\label{eq:02:proposition:grad:DL}
\begin{eqnarray}
 \nabla \K_k[\varphi](\x )&:=& {\rm p.v.}\int_\Gamma \nabla_{\x}\partial_{\nnn(\y)}\phi_k(\added[id=catB]{\x}-\y)\, \varphi(\y)\,\mathrm{d}\y\nonumber \\
 &=&  k^2\V_k[\varphi \nnn](\x)+ \partial_{\ttt} \V_k [\partial_{\ttt}\varphi]\!(\x)\,\nnn(\x) - \K_k^\top[ \partial_{\ttt}\varphi] (\x)\, {\ttt}(\x)  \label{eq:02a:proposition:grad:DL}
 \\
 &=& \W_k[\varphi] (\x)\, \nnn(\x)+  \big(
  k^2\V_k[\varphi \nnn](\x) \cdot\ttt(\x)-
 \K_k^\top [\partial_{\ttt}\varphi](\x)\big)\, \ttt (\x) .
 \label{eq:02b:proposition:grad:DL}
\end{eqnarray}
\end{subequations}
\end{proposition}
\begin{proof} From Lemma \ref{lemma:2.1}       we derive
\begin{eqnarray*}
\gamma^+ (\ttt\cdot \nabla {\rm DL}_k [\varphi]) &=&\gamma^+( k^2 \ttt\cdot {\rm SL}_k[\varphi\,\nnn]  + \ttt\cdot \overrightarrow{{\rm curl}}\,{\rm SL}_k[\partial_{\ttt}\varphi])
\\
&=&\gamma^+(k^2 \ttt\cdot {\rm SL}_k[\varphi\,\nnn]  - \nnn\cdot \nabla {\rm SL}_k[\partial_{\ttt}\varphi]) \\
 &=& k^2\ttt\cdot \V_k[\varphi\,\nnn]  -  \K_k^\top [\partial_{\ttt}\varphi] +\frac{1}2 \added[id=catR]{\partial_{\ttt} \varphi} \\
\gamma^+ ( \nnn\cdot \nabla {\rm DL}_k [\varphi])&=&\gamma^+ (k^2 \nnn\cdot {\rm SL}_k[\varphi\,\nnn] + \added[id=catR]{\ttt}\cdot \nabla {\rm SL}_k[\partial_{\ttt}\varphi])
= k^2 \nnn\cdot {\rm V}_k[\varphi\,\nnn]  + \partial_{\ttt} \V_k[\partial_{\ttt} \varphi].
\end{eqnarray*}
Alternatively,
\[
 \gamma^+ ( \nnn\cdot \nabla {\rm DL}_k [\varphi]) = \partial_{\nnn} {\rm DL}_k[\varphi] = \W_k[\varphi].
\]
Using \eqref{eq:ttnn} the proof is finished.
\end{proof}
\added[id=catB]{In order to further clarify certain notations we used in the proof of Proposition~\ref{proposition:grad:DL} whereby normals/tangents appear to act both on the left and on the right of integral operators, we specify for instance that the double action integral operator $\nnn\cdot \V_k[\varphi\,\nnn]$ is understood as
\[
\left(\nnn\cdot \V_k[\varphi\,\nnn]\right)(\x)=\nnn(\x)\cdot\int_\Gamma \phi_k(\x-\y)\, \varphi(\y)\, \nnn(\y)\,\mathrm{d}\y=\int_\Gamma \phi_k(\x-\y)\, \varphi(\y)\, \nnn(\x)\cdot\nnn(\y)\,\mathrm{d}\y.
\]
The other integral operators that feature such double action patterns are to be understood in the similar manner. In other words, we have adopted  the following convention that will be used hereafter:  for $\mathrm{L}$ a {\em scalar} operator, $\bm{d}=(d_1,d_2)^\top$ a vector function,
\[
\mathrm{L}[\varphi\bm{d}] = \begin{bmatrix}
\mathrm{L}[d_1 \varphi]  & \mathrm{L}_k[d_2 \varphi]
                              \end{bmatrix}^\top.
\]
}

We mention that as byproduct of Proposition \ref{proposition:grad:DL} the well-known Maue identity (cf. \cite[Ch. 9]{mclean:2000}, \cite[Ch. 7]{Kress} or \cite[Ch. 2]{Saranen}) has been derived for the hypersingular operator:
\begin{equation}\label{eq:maue:helmholtz}
 \W_{k}[\varphi]  =\partial_{\ttt} \V_{k}[\partial_{\ttt}\varphi] +\added[id=catR]{k^2} \nnn \cdot \V_{k}[\nnn \varphi].
\end{equation}

We consider next Dirichlet and Neumann traces of the Hessian for the Helmholtz single and double layer operators which play a major role in the Helmholtz decomposition approach for the two dimensional Navier equations with Neumann boundary conditions.
\added[id=vD]{Before stating the result, let us clarify the notation that will be used:
\[
 \nabla {\rm L}_k[\varphi {\bm d}^\top]:= \begin{bmatrix}
                                   \nabla \mathrm{L}_k[ \varphi d_1] &
                                   \nabla \mathrm{L}_k[\varphi d_2]
                                  \end{bmatrix}, \quad \mathrm{L}_k\in\{\V_k, \K_k\},\ \bm{d} = (d_1,d_2)^\top \in\{\nnn,\ttt\}.
\]
(Recall that the gradient operator $\nabla$ is always understood in this paper as a column vector). Consequently, $ \nabla {\rm L}_k[\varphi {\bm d}^\top]$ is a $2\times 2$ block matrix operator.
}

\begin{theorem}\label{theo:Hessians}
It holds
\begin{eqnarray}
\gamma^+(\Hes {\rm SL}_k[\varphi])
 &=&- \frac12\partial_{\ttt}\varphi \,(\nnn \ttt^\top+\ttt \nnn^\top)-\frac12\kappa \varphi \,(I_2-2\nnn\nnn^\top)\nonumber
 \\ & &  +  \nabla \V_k[\partial_{\ttt} \varphi\,\ttt^\top]  -
           \nabla \V_k[\kappa \varphi\,\nnn^\top]   -
           \nabla \K_k [\varphi\,\nnn^\top]\label{eq:01:theo:Hessians}\\
\gamma^+(\Hes {\rm DL}_k[\varphi])
 &=& \frac12\partial_{\ttt}^2\varphi \,(I_2-2\nnn\nnn^\top)-\frac12\kappa \partial_{\ttt}\varphi \,(\nnn \ttt^\top+\ttt \nnn^\top) -\frac12k^2\varphi \,\nnn\nnn^\top \nonumber
 \\ &&+  k^2 \nabla \V_k[\varphi\,\nnn^\top]+
      \nabla \V_k[\partial^2_{\ttt} \varphi\,\nnn^\top]  +
           \nabla \V_k[\kappa \partial_{\ttt}\varphi\,\ttt^\top]   +
           \nabla \K_k [\partial_{\ttt}\varphi\,\ttt^\top].
           \label{eq:02:theo:Hessians}
\end{eqnarray}
\end{theorem}
\begin{proof}
\added[id=vD]{Let us point out first that with the notation $\nabla^\top u = (\nabla u)^\top$, we can write the Hessian as $\Hes u =\nabla \nabla^\top u$). This identity will be used profusely in this proof. Hence, we start from the identity}
\begin{eqnarray*}
\nabla^\top{\rm SL}_k[\varphi](\x) &=&
{\rm SL}_k[\partial_{\ttt} \varphi\,\ttt^\top](\x)-{\rm SL}_k[\kappa \varphi\,\nnn^\top](\x)-{\rm DL}_k[\varphi\,\nnn^\top](\x).
\end{eqnarray*}
\added[id=vD]{(from \eqref{eq:02cc2:proposition:grad:SL}    in Proposition
 \ref{proposition:grad:SL})}.
Therefore we can apply Proposition \ref{proposition:grad:SL} (\added[id=vD]{again}) for the single layer terms  and  Proposition \ref{proposition:grad:DL} for the double layer to obtain
\begin{eqnarray*}
\gamma^+ \Hes  {\rm SL}_k[\varphi]  &=&
%
\nabla \V_k[\partial_{\ttt} \varphi\,\ttt^\top]  -
 \nabla \V_k[\kappa \varphi\,\nnn^\top]  -
 \nabla \K_k [\varphi\,\nnn^\top]
 \\
 &&
 \added[id=catR]{- \frac12\partial_{\ttt}\varphi \,\nnn \ttt^\top +\frac12\kappa \varphi \,\nnn\nnn^\top-\frac12\ttt\, (\partial_{\ttt}\varphi\, \nnn^\top +\kappa \varphi\, \ttt^\top)}\\
 &=&
 \nabla \V_k[\partial_{\ttt} \varphi\,\ttt^\top]  -
 \nabla \V_k[\kappa \varphi\,\nnn^\top]  -
 \nabla \K_k [\varphi\,\nnn^\top]
 \\
 &&
 - \frac12\partial_{\ttt}\varphi \,(\nnn \ttt^\top+\ttt \nnn^\top) -\frac12\kappa \varphi \,(\ttt\ttt^\top-\nnn\nnn^\top)
\end{eqnarray*}
Then \eqref{eq:01:theo:Hessians} follows now from \eqref{eq:ttnn}.

To prove \eqref{eq:02:theo:Hessians} we start from Lemma \ref{lemma:2.1}
which implies, using the vector calculus identity $\overrightarrow{{\rm curl}}\, = \mathrm{Q} \nabla$,
\begin{eqnarray*}
 \Hes {\rm DL}_k[\varphi] &=&
  k^2\nabla \mathrm{SL}_k[\varphi \nnn^\top]
 +\nabla\nabla^\top \mathrm{SL}_k[\partial_{\ttt}\varphi]\ \mathrm{Q}^\top.
\end{eqnarray*}
Proposition \ref{proposition:grad:SL} for the first term in the Hessian of the double layer identity above and formula~\eqref{eq:01:theo:Hessians} for the second term in the same equation imply
\begin{eqnarray*}
 \added[id=vD]{\gamma^+}\Hes {\rm DL}_k[\varphi]
 &=&  -\frac12 k^2 \varphi \nnn \nnn^\top + k^2 \nabla \V_k[\varphi \nnn^\top]\\
 &&\hspace{-32pt}+\Big(- \frac12\partial_{\ttt}^2\varphi \,(\nnn \ttt^\top+\ttt \nnn^\top) -\frac12\kappa  \partial_{\ttt} \varphi \,(I_2-2\nnn\nnn^\top)
 \\
&&
  +  \nabla \V_k[\partial^2_{\ttt} \varphi\,\ttt^\top]  -
           \nabla \V_k[\kappa \partial_{\ttt}\varphi\,\nnn^\top]   -
           \nabla \K_k [\partial_{\ttt}\varphi\,\nnn^\top]
 \Big)\mathrm{Q}^\top.
\end{eqnarray*}
The result follows from the relations $ \mathrm{Q}\nnn= -\ttt, \
 \mathrm{Q}\ttt=  \nnn.$

\end{proof}
%

We will summarize the results proven in this section in a way that will be used in both the analysis and implementation of the numerical algorithms. \added[id=vD]{Let us point out that these expressions are one of the possible forms, stemming from the different expressions for the trace of the Hessians for the single and double layer operators, derived from Theorem \ref{theo:Hessians} and Propositions \ref{proposition:grad:SL} and \ref{proposition:grad:DL}.}

\begin{theorem}\label{th:3.5}
 It holds
 \begin{eqnarray*}
\partial_{\nnn}  {\rm SL}_k[\varphi]=\gamma^+ (\nabla {\rm SL}_k[\varphi])\cdot \nnn &=& -\frac{1}2 \varphi + \K_k^\top [\varphi]
\\
\partial_{\ttt}  {\rm SL}_k[\varphi]= \gamma^+ (\nabla {\rm SL}_k[\varphi])\cdot \ttt &=& \partial_{\ttt} \V_k[\varphi]
 \end{eqnarray*}
and
 \begin{eqnarray*}
\partial_{\nnn}  {\rm DL}_k[\varphi]= \gamma^+ (\nabla {\rm DL}_k[\varphi])\cdot \nnn &=& \W_k[\varphi]=\partial_{\ttt} \V_k[\partial_{\ttt} \varphi]+ k^2 \nnn\cdot {\rm V}_k[\varphi\,\nnn]
\\
\partial_{\ttt}  {\rm DL}_k[\varphi] =\gamma^+ (\nabla {\rm DL}_k[\varphi])\cdot \ttt &=&\frac{1}2\partial_{\ttt}\varphi +
  k^2\V_k[\varphi \nnn] \cdot\ttt-
 \K_k^\top[ \partial_{\ttt}\varphi].
 \end{eqnarray*}

\end{theorem}
\added[id=vD]{For the next theorem, the following clarification for the notation  might be needed: if $\mathrm{L}$ is a scalar operator,  and  $A = (a_{ij})$ a $2\times 2$ matrix function, $\mathrm{L}[\varphi A]$ is the matrix operator given by
\[
 \mathrm{L}[\varphi A ] := \begin{bmatrix}
                             \mathrm{L}[a_{11}\varphi] &\mathrm{L}[a_{12}\varphi]  \\
                             \mathrm{L}[a_{21}\varphi] &\mathrm{L}[a_{22}\varphi]
                              \end{bmatrix}.
\]}

\begin{theorem}\label{th:3.6}
For the trace of the Hessian  matrix of the single layer boundary potential it holds
\[
 \begin{aligned}
\gamma^+\, \nnn^\top & (\Hes {\rm SL}_k[\varphi])\nnn =
- \gamma^+\, \ttt^\top (\Hes {\rm SL}_k[\varphi])\ttt -k^2 \V_k[\varphi]\\
&=
  \frac12 \kappa \varphi  + \nnn\cdot \Big( \K_k^\top [\partial_{\ttt} \varphi\,\ttt]-\K_k^\top [\kappa\varphi\,\nnn]
-\W_k[\varphi\,\nnn]\Big) \\
\gamma^+\, \ttt^\top & (\Hes {\rm SL}_k[\varphi])\ttt =
- \gamma^+\, \nnn^\top (\Hes {\rm SL}_k[\varphi])\nnn -k^2 \V_k[\varphi]\\
&=
  -\frac12 \kappa \varphi  + \ttt\cdot \Big(  \partial_{\ttt} \V_k [\partial_{\ttt}\varphi\,\ttt]-\partial_{\ttt} \V_k [\kappa\varphi \nnn]
\added[id=vD]{-k^2 \V_k[\varphi\,\nnn\nnn^\top]\ttt +\K_k^\top [\partial_{\ttt} \varphi\,\nnn]+\K_k^\top [\kappa \varphi\,\ttt]}\Big)
\\
\gamma^+\, \nnn^\top & (\Hes {\rm SL}_k[\varphi])\ttt =
\gamma^+\, \ttt^\top (\Hes {\rm SL}_k[\varphi])\nnn\\
&=  -\frac12 \partial_{\ttt} \varphi +{\bm n}\cdot \Big( \partial_{\ttt} \V_k [\varphi\,\ttt]-\partial_{\ttt} \V_k [\kappa\varphi\, \nnn]
\added[id=vD]{-k^2 \V_k[\varphi\,\nnn\nnn^\top]\ttt +\K_k^\top [\partial_{\ttt} \varphi\,\nnn]+\K_k^\top [\kappa \varphi\,\ttt]}\Big)
\\
&=  -\frac12 \partial_{\ttt} \varphi +{\bm t}\cdot  \Big( \K_k^\top [\partial_{\ttt} \varphi\,\ttt]-\K_k^\top [\kappa\varphi\,\nnn]
-\W_k[\varphi\,\nnn]\Big).
\end{aligned}
\]
Similarly, for the trace of the Hessian  matrix of the double layer potential  we have
\[
 \begin{aligned}
\gamma^+\, \nnn^\top (\Hes {\rm DL}_k[\varphi])\nnn =\ &
- \gamma^+\, \ttt^\top (\Hes {\rm DL}_k[\varphi])\ttt -{ k^2 \left(\frac{1}2 \varphi+\K_k[\varphi]\right)} \\
=\ &
 -\frac12 \partial_{\ttt}^2 \varphi  -\frac12 k^2 \varphi\,
+ \nnn\cdot \Big(  k^2\K_k^\top[\varphi\,\nnn]+  \K_k^\top [\partial_{\ttt}^2 \varphi\,\nnn]
\\
& +\K_k^\top [\kappa\partial_{\ttt}\varphi \ttt]+
\W_k[\partial_{\ttt} \varphi\,\ttt]\Big) \\
\gamma^+\, \ttt^\top (\Hes {\rm DL}_k[\varphi])\ttt =\ &
- \gamma^+\, \nnn^\top (\Hes {\rm DL}_k[\varphi])\nnn -{ k^2 \left(\frac{1}2\varphi+\K_k[\varphi]\right)}\\
=\ &
 \frac12 \partial_{\ttt}^2 \varphi
+ \ttt\cdot \Big(  k^2 \partial_{\ttt} \V_k[\varphi\,\nnn]
+
\partial_{\ttt} \V_k [\partial_{\ttt}^2\varphi  \nnn]+\partial_{\ttt} \V_k [\kappa\varphi \ttt]\\
&
+k^2 \V_k[\added[id=vD]{\partial_t}\varphi\,\ttt\nnn^\top]\ttt
\added[id=vD]{
 -\K_k^\top [\partial_{\ttt}^2\varphi\,\ttt]
 +\K_k^\top [\kappa \partial_{\ttt}\varphi\,\nnn]}\Big) \\
\gamma^+\, \nnn^\top (\Hes {\rm DL}_k[\varphi])\ttt\ =\ &
\gamma^+\, \ttt^\top (\Hes {\rm DL}_k[\varphi])\nnn\\
=\ &   -\frac12 \kappa \partial_{\ttt}\varphi    + {\bm n}\cdot\Big(  k^2 \partial_{\ttt} \V_k[\varphi\,\nnn]
+
\partial_{\ttt} \V_k [\partial_{\ttt}^2\varphi  \nnn]+\partial_{\ttt} \V_k [\kappa\varphi \ttt]\\
&
+k^2 \V_k[\added[id=vD]{\partial_t}\varphi\,\ttt\nnn^\top]\ttt
\added[id=vD]{
 -\K_k^\top [\partial_{\ttt}^2\varphi\,\ttt]
 +\K_k^\top [\kappa \partial_{\ttt}\varphi\,\nnn]}\Big)
\\
=\ & -\frac12 \kappa \partial_{\ttt}\varphi  +{\bm t}\cdot  \Big(k^2\K_k^\top[\varphi\,\nnn] \\
& +  \K_k^\top [\partial_{\ttt}^2 \varphi\,\nnn]+\K_k^\top [\kappa\partial_{\ttt}\varphi \ttt]+
\W_k[\partial_{\ttt} \varphi\,\ttt]\Big).
\end{aligned}
\]

\end{theorem}
\begin{proof}
\added[id=catR]{
It is consequence of the expressions we have derived for the trace of the Hessians of the single and double layer operator in Theorem \ref{theo:Hessians} combined with \eqref{eq:02b:proposition:grad:SL} in Proposition \ref{proposition:grad:SL} and
\eqref{eq:02b:proposition:grad:DL} in Proposition \ref{proposition:grad:DL}.}

\added[id=catR]{Notice also that we have used that if $u$ solves the Helmholtz equation with wave-number $k$ then
\[
 \nnn^\top  \Hes u\, \nnn+
 \ttt^\top  \Hes u\,\ttt = {\rm Tr} (\Hes u) =\Delta u = -k^2 u
\]
where ${\rm Tr}(A)$ denotes the trace of the matrix $A$.
}
\end{proof}

 To conclude this section we will write the principal part of the normal and tangential derivatives for the layer potentials in terms of two basic pseudodifferential operator: powers of the tangential derivative and the Hilbert transform. Hence, let us denote
\[
 \D   = \partial_{\ttt} : H^s(\Gamma)\to H^{s-1}(\Gamma), \quad
 \D_{-1} :=\D^\dagger: H^s(\Gamma)\to H^{s+1}(\Gamma)
\]
the tangential derivative and its pseudoinverse. We will keep this double notation, ${\D}$ and $\partial_{\ttt}$, for the tangential derivative, favouring the former when dealing with expansions of the boundary operators.

We will define the integer powers of $\D$ in the natural way: with $r\in\mathbb{N}$, set
\[
 \D_{r} :=  \D  ^r, \quad
 \D_{-r} :=  (\added[id=catR]{\D_{-1}})^{r}, \quad \D_0 :=\D_r \D_{-r}.
\]
Clearly,
\[
 \D_{r} \D_{s} = \D_{r+s},\qquad
 \D_0 \varphi =\varphi -\added[id=vD]{\mathrm{J}\varphi}, \text{ with } \added[id=vD]{\mathrm{J}\varphi} := \frac{1}{2\pi}\int_\Gamma \varphi .
\]
On the other hand, let $\HH:H^s(\Gamma)\to H^s(\Gamma)$ be  the Hilbert transform or Hilbert singular operator (cf. \cite[Ch. 5]{Saranen}, \cite[Ch. 7]{Kress})   given by
\[
\begin{aligned}
( \HH\varphi)({\bf x}(t)) & =
-\frac{1}{2\pi}\mathrm{p.v.}\int_0^{2\pi} \cot\left(\frac{\added[id=catB]{t}-\tau}2\right)\varphi({\bf x}(\tau))\,{\rm d}\tau+  i\added[id=vD]{\mathrm{J}\varphi}\\
& =
i \bigg[ \widehat{\varphi}(0) + \sum_{n\ne 0}\mathop{\rm \rm sign}(n)\widehat{\varphi}(n)e_{n}(t)\bigg],
\end{aligned}
\]
where ${\bf x}:\mathbb{R}\to\Gamma$ is (a) $2\pi-$periodic arc-length parameterization, cf. Remark \ref{remark:curves:2pi},  positively oriented of $\Gamma$   and
\begin{equation}\label{eq:3.18}
\begin{aligned}
 \widehat\varphi(n) :=\frac{1}{\added[id=catR]{2\pi}}\int_0^{2\pi}  \varphi({\bf x}(\tau))e_{-n}(\tau)\,{\rm d}\tau,\quad e_n(\tau) :=\exp(i  \,n \tau),\quad n\in\mathbb{Z}
 \end{aligned}
\end{equation}
the Fourier coefficients of $\varphi\circ{\bf x}$. Clearly $
 \HH^2 = -\I.$
The operators $\D_r$ can be also written in the Fourier basis:
\[
 ( \D_r\varphi)({\bf x}(t)) = \sum_{n\ne 0} \left( i n  \right)^r \widehat{\varphi}(n)e_{n}(t)
\]
which as byproduct shows that
\[
 \HH\D_r =\D_r\HH.
\]

Furthermore, since for any smooth function $a:\Gamma\to\mathbb{C}$,  $\HH a-a\HH$   can be written as an integral operator with smooth kernel we can conclude that $a\HH$ and $\HH a$ differ  in a smoothing operator. The notation, which will be extensively used from now on,
\[
 {\mathrm A}_1 = {\mathrm A}_2 +\mathrm{OPS}(r) \ \Leftrightarrow \
 {\mathrm A}_1 - {\mathrm A}_2 \in \mathrm{OPS}(r)
\]
allows us to simply write
\[
 a\HH=\HH a+\mathrm{OPS}(-\infty), \qquad
 a\D_r=\D_r a+\mathrm{OPS}(r-1).
\]

The last two results of this section detail how the normal and tangential derivatives of the boundary layer operators can be expanded in terms of these basic operators.  Hence,

The following result is then consequence of this Theorem and Theorem \ref{th:3.5}-\ref{th:3.6} and Theorem \ref{theo:ntDnt:02} in the appendix.
\begin{proposition} \label{prop:principalPart:Grad} It holds
\begin{equation}\label{eq:01:prop:principalPart:Grad}
\begin{aligned}
 \left(\gamma^+\,   \nabla {\rm SL}_k\right)\cdot \nnn & \ = \
-\frac12 \I  -\frac{1}4\kappa k^2 \D_{-3}\HH    +\mathrm{OPS}(-4),\\
 \left(\gamma^+\,   \nabla {\rm SL}_k\right)\cdot \ttt &\ =
 \frac12 \HH  -\frac{1}4  k^2 \D_{-2}\HH    +\mathrm{OPS}(-4)
 \end{aligned}
\end{equation}
and
\begin{equation}\label{eq:02:prop:principalPart:Grad}
\begin{aligned}
 \left(\gamma^+\,   \nabla {\rm DL}_k\right)\cdot \nnn  & \ = \
\frac12 \HH\D  +\frac{1}4  k^2 \HH\D_{-1}  +\mathrm{OPS}(-3),\\
 \left(\gamma^+\,   \nabla {\rm DL}_k\right)\cdot \ttt  & \ = \
 \frac12 \D  -  \frac{k^2\kappa}{4}\HH\D_{-2}    +\mathrm{OPS}(-3).
 \end{aligned}
\end{equation}
\end{proposition}

\begin{proposition} \label{prop:principalPart:Hess} It holds
\begin{equation}\label{eq:01:prop:principalPart:Hess}
\begin{aligned}
 \gamma^+\, \nnn^\top  \Hes {\rm SL}_k\ \nnn &\ = \
-\frac12 \D\HH  +\frac{1}2\kappa \I -\frac{k^2}4 \HH\D_{-1}    +\mathrm{OPS}(-3),\\
  \gamma^+\, \ttt^\top  \Hes {\rm SL}_k\ \ttt &\ = \
 \frac12 \D\HH  -\frac{1}2\kappa \I -\frac{k^2}4 \HH\D_{-1}    +\mathrm{OPS}(-3),\\
 \gamma^+\, \ttt^\top  \Hes {\rm SL}_k\ \nnn& \ =\
 \gamma^+\, \nnn^\top  \Hes {\rm SL}_k\ \ttt
 \\
&\ = \
-\frac12\D \varphi-\frac{1}2\kappa\HH      -\frac{k^2\kappa}4  \deleted{\kappa}\D_{-2}\HH+ \mathrm{OPS}(-3),
\end{aligned}
\end{equation}
and
\begin{equation}\label{eq:02:prop:principalPart:Hess}
\begin{aligned}
 \gamma^+\, \nnn^\top  \Hes {\rm DL}_k\ \nnn &\ = \
 -\frac12  \D_2   -\frac12\kappa\HH\D  -\frac12k^2\I   -\frac{k^2}4\kappa  \HH\D_{-1} +\mathrm{OPS}(-2),
 \\
  \gamma^+\, \ttt^\top  \Hes {\rm DL}_k\ \ttt&\ = \
  \frac12  \D_2   +\frac12\kappa\HH\D +\frac{k^2}4\kappa  \HH\D_{-1}    +\mathrm{OPS}(-2),\\
 \gamma^+\, \ttt^\top  \Hes {\rm DL}_k\ \nnn &\ = \
 \gamma^+\, \nnn^\top  \Hes {\rm DL}_k\ \ttt \\
&\ = \frac{1}2 \HH\D_2-\frac{1}2 \kappa \D +\frac{k^2}4 \HH +\mathrm{OPS}(-2).
\end{aligned}
\end{equation}
\end{proposition}

\section{Boundary integral formulations} \label{BIE}

From now on, and to alleviate the notation, we will use ``$p$'' and ``$s$'' in subscripts to refer to $k_p$ and $k_s$, the pressure and stress wave-numbers  so that
\[
 \SL_{p}= \SL_{k_p}, \
 \SL_{s}= \SL_{k_s}, \
 \DL_{p}= \DL_{k_p}, \
 \DL_{s}= \DL_{k_s},
\]
and similarly for the associated BIOs.
In what follows the operator matrix
\[
 {\cal H}_0:=\begin{bmatrix}\I & -\HH \\- \HH & -\I\end{bmatrix}\in\mathrm{OPS}(0)
 \]
 will play an essential role. Notice that ${\cal H}_0$ is nilpotent of index 2, i.e., $
  {\cal H}_0^2 = 0.$

We note also that the matrix operator
 \begin{subequations}\label{eq:inverH2}\begin{equation}
  \begin{bmatrix}
   a_{11} \I & a_{12}\HH\\
   a_{12}\HH &a_{22} \I
   \end{bmatrix}\in\mathrm{OPS}(0)
   \end{equation}
   with $a_{ij}\in\mathbb{C}$
 is invertible if and only if
 \begin{equation}
  \det   \begin{bmatrix}
   a_{11}   & a_{21} i\\
   a_{12} i  &a_{22}
   \end{bmatrix} = a_{11}a_{22}+a_{12}a_{22}\ne 0
 \end{equation}
 and that the inverse is given  by
 \begin{equation}
 \frac{1}{a_{11}a_{22}+a_{12}a_{21}} \begin{bmatrix}
    a_{22} \I & -a_{12}\HH\\
   -a_{21}\HH &a_{11} \I
   \end{bmatrix}.
 \end{equation}
 \end{subequations}
 In particular, this also implies that $\mathcal{H}_0$ is not invertible.

For any matrix ${\cal A}$ and scalar operator $\mathrm{B}$ with
 \[
{\cal A} =\begin{bmatrix}
           \mathrm{A}_{11} &           \mathrm{A}_{12} \\
           \mathrm{A}_{21} &           \mathrm{A}_{22} \\
          \end{bmatrix},
          \quad
         \deleted{\mathrm{B}}
 \]
we will denote
 \[
  {\cal A}\,\mathrm{B} =
  {\cal A}\otimes\mathrm{B} =\begin{bmatrix}
           \mathrm{A}_{11} \mathrm{B}&           \mathrm{A}_{12}\mathrm{B} \\
           \mathrm{A}_{21}\mathrm{B} &           \mathrm{A}_{22}\mathrm{B}
          \end{bmatrix},\qquad
            \mathrm{B}\,{\cal A}=
  \mathrm{B} \otimes {\cal A}=\begin{bmatrix}
           \mathrm{B}\mathrm{A}_{11} &           \mathrm{B}\mathrm{A}_{12} \\
           \mathrm{B}\mathrm{A}_{21} &           \mathrm{B}\mathrm{A}_{22}
          \end{bmatrix}.
 \]
 Clearly, if $a_{ij}\in \mathbb{C}$,
 \[
  \begin{bmatrix}
   a_{11} \I & a_{12}\HH\\
   a_{12}\HH &a_{22} \I
   \end{bmatrix} \D_r\HH =
   \D_r\HH \begin{bmatrix}
   a_{11} \I & a_{12}\HH\\
   a_{12}\HH &a_{22}\I
   \end{bmatrix}.
 \]

We will work in this section with Fourier multiplier, or diagonal,  scalar
and  matrix operator,
\[
( \added[id=vD]{\mathrm{A}_0}\varphi)\circ{\bf x} = \sum_{n\in\mathbb{Z}}
\widehat{\added[id=vD]{\mathrm{A}_0}}(n)   \widehat{\varphi} (n)  e_n, \qquad
 (\added[id=vD]{\mathcal{A}_0}\bm{\varphi})\circ{\bf x} := \sum_{n\in\mathbb{Z}}
 \widehat{\added[id=vD]{\mathcal{A}_0}}(n) \begin{bmatrix}
                \widehat{\varphi}_1(n)\\
                \widehat{\varphi}_2(n)
                \end{bmatrix} e_n
\]
where  $\widehat{\added[id=vD]{\mathrm{A}_0}}(n)\in\mathbb{C}$ and the $2\times 2$ complex matrix
$\widehat{ \added[id=vD]{\mathcal{A}_0}}(n) $ are the so-called \deleted{principal} symbol  of $\added[id=vD]{\mathrm{A}_0}$ and $\added[id=vD]{\mathcal{A}_0}$. Besides,
$e_n$ and ${\bf x}$ are the complex exponential and the (or an) arc-length parameterization  of the curve such as it was introduced in \eqref{eq:3.18}.

Therefore we can write
\[
 ({\cal H}_0\bm{\varphi})({\bf x}(\cdot)) =\begin{bmatrix}
                          1&-i\\
                          -i&-1
                         \end{bmatrix}
\begin{bmatrix}
                         \widehat{\varphi}_1(0)\\
                         \widehat{\varphi}_2(0)
                         \end{bmatrix}
 +
\sum_{n\ne 0}
 \begin{bmatrix}
  1 & -i\sign(n)\\
  -i\sign(n)&-1\\
 \end{bmatrix}
\begin{bmatrix}
                \widehat{\varphi}_1(n)\\
                \widehat{\varphi}_2(n)
                \end{bmatrix} e_n.
\]

Clearly \added[id=vD]{$A,{\cal A}\in\mathrm{OPS}(m)$, that is}
\[
\added[id=vD]{\mathrm{A}_0}:H^{s}(\Gamma)\to H^{s-m}(\Gamma), \quad \added[id=vD]{\mathcal{A}_0}:H^s(\Gamma)\times H^{s}(\Gamma)\to  H^{s-m}(\Gamma)\times H^{s-m}(\Gamma)
\]
are continuous for any $s$ provided that, respectively,
\[
 |\widehat{\added[id=vD]{\mathrm{A}_0}}(n)|\le C (1+|n|)^{m},\quad
 \| \widehat{\added[id=vD]{\mathcal{A}_0}}(n)\|_2\le C (1+|n|)^{m}
\]
with $C$ independent of $n$.
Injectivity for matrix, respectively scalar, operator is equivalent to $ \widehat{\added[id=vD]{\mathcal{A}_0}}(n)-$invertibility, respectively  $|\widehat{A}(n)|\ne 0$ for all $n$. Finally, $\added[id=vD]{\widehat{\mathrm{A}}_0^{-1}(n)}$ and $\added[id=vD]{\widehat{\mathcal{A}}_0^{-1}(n)}$ are the principal symbols of the operators $\added[id=vD]{\mathrm{A}_0^{-1}}$ and $\added[id=vD]{\mathcal{A}_0^{-1}}$.

\added[id=vD]{Finally, for any continuous operator $\mathrm{A}:H^s(\Gamma)\to H^{s-m}(\Gamma)$,   $s\in\mathbb{R}$ and for integer value $m$, we will say that a Fourier multiplier, or diagonal, operator $\mathrm{A}_0$ is the (or a) principal symbol of $\mathrm{A}$ if $\mathrm{A}-\mathrm{A}_0\in \mathrm{OPS}(m-1)$. We will then write
\[
 \mathrm{PS}({\mathrm{A}}) = \mathrm{A}_0.
\]
Similarly, a matrix Fourier multiplier $\mathcal{A}_0$ is a  principal symmbol of $\mathcal{A}:H^s(\Gamma)\times H^s(\Gamma)\to H^{s-m}(\Gamma) \times H^{s-m}(\Gamma)$, and we will write
\[
 \mathrm{PS}(\mathcal{A}) = \mathcal{A}_0.
 \]
if   $\mathcal{A}- \mathcal{A}_0\in\mathrm{OPS}(m-1)$. It is a well known fact in the Theory of pseudodifferential operators that a principal symbol for $\mathrm{A}^{-1}$, respectively $\mathcal{A}^{-1}$, is given by $\mathrm{PS}(\mathrm{A})^{-1}$, respectively $\mathrm{PS}(\mathcal{A})^{-1}$.}

\subsection{Dirichlet problem}\label{dirP}
\subsubsection{Combined formulations}

If we look for combined field representations of the Helmholtz fields $u_p$ and $u_s$ defined in equation~\eqref{eq:Hdecomp2} satisfying the system of boundary conditions~\eqref{eq:DH} in the form
\begin{equation}\label{eq:comb_field}
u_p =\DL_p [\varphi_p] -ik\SL_p [\varphi_p] \qquad u_s=\DL_s [\varphi_s] -ik\SL_s [\varphi_s],\ \added[id=catB]{k>0}
\end{equation}
the results in Theorem \ref{th:3.5} led to the \added[id=catB]{Combined Field Integral Equation (CFIE)} formulation
\[
\left({\cal A}_{\rm DL}-ik{\cal A}_{\rm SL}\right)\begin{bmatrix}\varphi_p\\ \varphi_s\end{bmatrix}=-\begin{bmatrix}{\bf u}^{\rm inc}\cdot \nnn\\ {\bf u}^{\rm inc}\cdot \ttt\end{bmatrix}
\]
in terms of the matrix BIO
 \begin{eqnarray}
{\cal A}_{\rm DL} &=&
\begin{bmatrix}
\W_p&\frac{1}2 \partial_{\ttt} + k_s^2\ttt \cdot \V_s[\nnn\,]-\K_s^\top [\partial_{\ttt}\,]       \\
\frac{1}2 \partial_{\ttt} + k_p^2\ttt \cdot \V_p[\nnn\,]-\K_p^\top [\partial_{\ttt}\,]  &-\W_s
\end{bmatrix}\label{eq:ADL}
\\
&=&
\frac{1}2\mathcal{H}_0\D+\frac{1}4\begin{bmatrix}
                                       k_p^2&\\
                                            &-k_s^2
                                      \end{bmatrix}\HH\D+\mathrm{OPS}(-2)\nonumber\\
{\cal A}_{\rm SL} &=&
\begin{bmatrix}
-\frac{1}2 \I + \K_p^\top &  \partial_{\ttt} \V_s  \\
\partial_{\ttt} \V_p  &\frac{1}2 \I - \K_s^\top
\end{bmatrix}\label{eq:ASL}
=
-\frac{1}2\mathcal{H}_0-\frac{1}4\begin{bmatrix}
                                           &k_s^2 \\
                                      k_p^2
                                      \end{bmatrix}\HH\D_{-2}+\mathrm{OPS}(-3).
 \end{eqnarray}
\added[id=vD]{(The expansions above for ${\cal A}_{\rm DL}$  and
 ${\cal A}_{\rm SL}$ are direct consequences of   Proposition \ref{prop:principalPart:Grad}.) Let us clarify that in the expression above, $ \ttt \cdot \V_s[\nnn\,]$ (similarly for $\ttt \cdot \V_p[\nnn\,]$) is
\[
 (\ttt \cdot \V_s[\nnn\,])[\varphi](\bm{x})  =\ttt({\bm x}) \cdot \V_s[\nnn\,\varphi](\bm{x})
\]
with ``$\cdot$'' denoting the inner product  (and so $\V_s$ is applied to two-components vector).
}

 We then have
 \begin{equation}\label{eq:cfieD}
 \begin{aligned}
 {\cal A}_{\rm CFIE}&:= {\cal A}_{\rm DL}-ik{\cal A}_{\rm SL}
 =
 \frac{1}2 \left[\HH\D +ik\right] \mathcal{H}_0 +\frac{1}4
 \begin{bmatrix}
  k_p^2 & \\
   & -k_s^2
 \end{bmatrix}
\HH\D_{-1} +\mathrm{OPS}(-2).
 \end{aligned}
 \end{equation}

 \begin{theorem}\label{CFIED}
 The  ${\cal A}_{\rm CFIE}$ defined in equation~\eqref{eq:cfieD} enjoys the property ${\cal A}_{\rm CFIE}^2\in \mathrm{OPS}(0)$ and is invertible.
 \end{theorem}
 \begin{proof}
Clearly
 \begin{eqnarray*}
  {\cal A}_{\rm CFIE}^2 &=& \frac{1}4\underbrace{\left[ \left(\HH\D +ik\right)  \mathcal{H}_0\right]^2}_{=0}
  -\frac{1}8  \begin{bmatrix}
  k_p^2 & \\
   & -k_s^2
 \end{bmatrix}\mathcal{H}_0 -\frac{1}8 \mathcal{H}_0  \begin{bmatrix}
  k_p^2 & \\
   & -k_s^2
 \end{bmatrix} +\mathrm{OPS}(-1)\\
 &=&
 -\frac{1}{8}\begin{bmatrix}2k_p^2 \I & -(k_p^2+k_s^2)\HH \\ -(k_p^2+k_s^2)\HH & -2k_s^2 \I \end{bmatrix}+\mathrm{OPS}(-1).
 \end{eqnarray*}
 The principal part is clearly invertible  with \added[id=vD]{cf. \eqref{eq:inverH2}}
 \[
 \added[id=vD]{\mathrm{PS}(  {\cal A}_{\rm CFIE}^2 )=}
 \begin{bmatrix}2k_p^2 \I & -(k_p^2+k_s^2)\HH \\ -(k_p^2+k_s^2)\HH & -2k_s^2 \I \end{bmatrix}^{-1} =
  \frac{1}{(k_p-k_s)^2}
 \begin{bmatrix}-2k_s^2 \I & (k_p^2+k_s^2)\HH \\ (k_p^2+k_s^2)\HH & 2k_p^2 \I \end{bmatrix}
 \]

 Therefore, the invertibility of the operator ${\cal A}_{\rm CFIE}^2$ is equivalent to the injectivity of the operator ${\cal A}_{\rm CFIE}$. The latter, in turn, can be established in a straightforward manner via classical arguments related to the well posedness of Helmholtz CFIE formulations. Indeed, assuming $(\varphi_p,\varphi_s)^\top \in \mathop{\rm Ker}({\cal A}_{\rm CFIE})$, we define the associated elastic field via the Helmholtz combined field representations \eqref{eq:comb_field} with these densities
\[
{\bf u}=\nabla u_p+\vcurl {u_s}
\]
and we see that ${\bf u}$ is a radiative solution of the Navier equations in $\Omega^+$ satisfying ${\bf u}\cdot \ttt=0$ and ${\bf u}\cdot \nnn=0$ on $\Gamma$ and therefore ${\bf u}=0$ on $\Gamma$. We can invoke the uniqueness result for solutions of exterior Navier problems with Dirichlet boundary conditions to get that ${\bf u}=0$ in $\Omega^+$. Therefore we obtain $\divv{\bf u}=0$ in $\Omega^+$ which amounts to $\Delta u_p=0$ in $\Omega^+$.  Given that $u_p$ is a solution of the Helmholtz equation with wave-number $k_p$ in $\Omega^+$ we derive $u_p=0$ in $\Omega^+$. Then, we have $\vcurl {u_s}=0$ and hence $\Delta u_s=0$ in $\Omega^+$ which, since $u_s$ is a solution of the Helmholtz equation with wave-number $k_s$, implies that $u_s=0$ in $\Omega^+$. Finally, using the jump conditions of Helmholtz layer potentials \eqref{eq:jumps}, we see that $u_p$ and $u_s$ are solutions of Helmholtz equations in $\Omega$ with wave-numbers $k_p$ and $k_s$ satisfying the Robin conditions
\[
\partial_{\nnn}^{-} u_p-ik\gamma^{-}u_p=0 \qquad \partial_{\nnn}^{-} u_s-ik\gamma^{-}u_s=0\qquad {\rm on}\ \Gamma.
\]
Therefore, $u_p=0$ and $u_s=0$ in $\Omega$ as well, and hence $\varphi_p=0$ and $\varphi_s=0$ on $\Gamma$.
\end{proof}

Following the Helmholtz paradigm, it would appear more natural to look for combined field representations of the fields $u_p$ and $u_s$ in the form
\begin{equation}\label{remarkCFIED}
u_p:=\DL_p [\varphi_p] -ik_p\SL_p [\varphi_p] \qquad u_s:=\DL_s[ \varphi_s] -ik_s\SL_s [\varphi_s]
\end{equation}
leading to the CFIE formulation
\[
\widetilde{\cal A}_{\rm CFIE} \begin{bmatrix}\varphi_p\\ \varphi_s\end{bmatrix}:= \left({\cal A}_{\rm DL}-\begin{bmatrix}ik_p & \\ & ik_s\end{bmatrix}{\cal A}_{\rm SL}\right)\begin{bmatrix}\varphi_p\\ \varphi_s\end{bmatrix}=-\begin{bmatrix}{\bf u}^{\rm inc}\cdot \nnn\\ {\bf u}^{\rm inc}\cdot \ttt\end{bmatrix}.
\]
In this case  it can be easily proven that
 \begin{eqnarray*}
  \widetilde{\cal A}_{\rm CFIE}^2 &=& \frac{1}4 \left[ \left(\HH\D +i\begin{bmatrix}
                                                                            k_p  &\\
                                                                            & k_s
                                                                           \end{bmatrix}
\right) \mathcal{H}_0\right]^2
\\
&&
  -\frac{1}8  \begin{bmatrix}
  k_p^2 & \\
   & -k_s^2
 \end{bmatrix}\mathcal{H}_0 -\frac{1}8 \mathcal{H}_0  \begin{bmatrix}
  k_p^2 & \\
   & -k_s^2
 \end{bmatrix} +\mathrm{OPS}(-1)
 \\
 &=& \frac{1}4   i\mathcal{H}_0  \begin{bmatrix}
                                                                            k_p  &\\
                                                                            & k_s
                                                                           \end{bmatrix}\mathcal{H}_0 \D
                                                                           +\mathrm{OPS}(0)
=\frac{i}4 (k_p-k_s) \mathcal{H}_0\HH\D +\mathrm{OPS}(0)
 \end{eqnarray*}
Then   $\widetilde{\cal A}_{\rm CFIE}^2\in\mathrm{OPS}(1)$ and its principal part is {\em not} invertible.

\subsubsection{Regularized formulations}

The design of these new regularized formulations  starts  with Green's identities in $\Omega^+$
\added[id=catB]{for the radiative solution $u_k$ of a Helmholtz equation with wave-number $k$}
\begin{equation} \label{eq:representation}\added[id=catB]{
u_k=\DL_k[\gamma^+ u_k]-\SL_k[\partial_{\nnn}^+u_k], \quad \text{in }\Omega^+.
}
\end{equation}
The main thrust in regularized formulations~\cite{boubendir2015regularized} is the construction of a certain regularizing operator ${\cal R}$ that is an approximation of the operator that maps the left hand side of equation~\eqref{eq:DH} to the Dirichlet Cauchy data $(\gamma^+  u_p, \gamma^+ u_s)$ on $\Gamma$ corresponding to the Helmholtz solutions $u_p$ and $u_s$ in $\Omega^+$. At the core of the construction of regularizing operators lies the use of {\em coercive} approximations of {\rm DtN} \added[id=vD]{(Dirichlet-to-Neumann)} operators which are typically represented in the form of Fourier square root multipliers. Once such an operator is constructed, approximations of {\rm DtN} operators are used to access the Neumann Cauchy data $(\partial_{\nnn}^+u_p, \partial_{\nnn}^+u_s)$ on $\Gamma$, which in turn leads via Green's identities to representations of the fields $(u_p,u_s)$ in $\Omega^+$.
Specifically, from  \eqref{eq:representation}, the jump relations for the layer potentials and Theorem \ref{th:psido:exp}
\begin{equation}\label{eq:DtN_W}
 {\rm DtN}_k = 2\W_k-2\K_k^\top =2\W_k +\mathrm{OPS}(-3)=   \HH\D_{ 1}+\frac{k^2}2\HH\D_{-1}+\mathrm{OPS}(-2)
\end{equation}
we can easily deduce
\begin{equation}\label{eq:DTNptilde}
 \begin{aligned}
{\rm DtN}_k& =2\W_{\widetilde{k}}+\mathrm{OPS}(-1) =  \HH\D+\mathrm{OPS}(-1), \quad k\in \{k_p,k_s\}
 \end{aligned}
\end{equation}
where
\begin{equation}\label{eq:comp_waven}
\widetilde{k}_p:=k_p+i\varepsilon_p\qquad \widetilde{k}_s:=k_s+i\varepsilon_s,\qquad \varepsilon_p>0,\ \varepsilon_s>0.
\end{equation}
\added[id=catB]{Using BIOs with complex wave-numbers for the approximation of DtN operators is common practice in the construction of regularized BIE formulations~\cite{Antoine,AntoineX,boubendir2015regularized}. At the theoretical level the complexification procedure ensures certain coercivity properties that are needed to establish the well-posedness of the regularized formulations. On the practical side, if the imaginary parts of the complex wave-numbers is selected to be proportional to the cubic root of the real wave-number in the underlying scattering application, with constants of proportionality in turn depending on the (positive) curvature of the boundary, the aforementioned DtN approximations are quasi-optimal as far as the iterative performance of regularized formulations is concerned.}

Using the complexified DtN approximations above, we express then the fields $(u_p,u_s)$ in the combined field form
\begin{eqnarray}\label{eq:systemD}
 u_r&=& {\rm DL}_{r}[ \varphi_r] -2{\rm SL}_{r} [\W_{\widetilde{k}_r}\varphi_r],\qquad r\in\{p,s\}
 \end{eqnarray}
through a regularizing operator ${\cal R}$ in terms of boundary functional densities $(\varphi_p,\varphi_s)$
\[
 \begin{bmatrix}
  \varphi_p\\
  \varphi_s
 \end{bmatrix}=
  {\cal R}\begin{bmatrix}
                f_p\\
                f_s
               \end{bmatrix}
\]
for appropriate $(f_p,f_s)$. The enforcement of boundary conditions~\eqref{eq:DH} on the combined field representation~\eqref{eq:systemD} leads to the \added[id=catB]{Combined Field Regularized Integral Equation (CFIER)} formulation to be solved for the unknown densities $(f_p,f_s)$
\begin{equation}\label{eq:RegFormulation}
{\cal A}^{\rm comb}{\cal R}\begin{bmatrix}
                f_p\\
                f_s
               \end{bmatrix}
 =\begin{bmatrix}
   -{\bf u}^{\rm inc}\cdot \nnn\\
 -{\bf u}^{\rm inc}\cdot \ttt
  \end{bmatrix}, \quad
   {\cal A}^{\rm comb}:= {\cal A}_{\rm DL} -2{\cal A}_{\rm SL} \begin{bmatrix}
                                        \W_{\widetilde{k}_p} & \\
                                            &\W_{\widetilde{k}_s}
                                       \end{bmatrix}.
\end{equation}
(See \eqref{eq:ADL} and \eqref{eq:ASL}).
Provided the regularizing operator ${\cal R}$ is constructed per the prescriptions above, the CFIER operator in the left hand side of equation~\eqref{eq:RegFormulation} \added[id=catB]{should be} an approximation of the identity operator. Clearly, the main challenge in this approach is the construction of the regularizing operator ${\cal R}$. Using the {\rm DtN} maps ${\rm DtN}_p$ and  ${\rm DtN}_s$, the \emph{exact} regularizing operator $\mathcal{R}^{\rm ex}_{\rm D}$ has the property
\[
\begin{bmatrix} {\rm DtN}_p & \partial_{\ttt}\\ \partial_{\ttt} & -{\rm DtN}_s\end{bmatrix}\mathcal{R}^{\rm ex}_{\rm D}=\mathcal{I}.
\]
(see \eqref{eq:DH})
and therefore, {\em formally},
\begin{eqnarray*}
 \mathcal{R}^{\rm ex}_{\rm D} &=& \begin{bmatrix} {\rm DtN}_p & \partial_{\ttt}\\ \partial_{\ttt} & -{\rm DtN}_s\end{bmatrix}^{-1}
 =
  ({\rm DtN}_p{\rm DtN}_s+\partial_{\bm t}^2)^{-1}
   \begin{bmatrix}
   {\rm DtN}_s& \partial_{\ttt}\\
    \partial_{\ttt} &  -{\rm DtN}_p
   \end{bmatrix}.
\end{eqnarray*}
Notice that by \eqref{eq:DtN_W}
\[
\begin{aligned}
 {\rm DtN}_p{\rm DtN}_s+\partial_{\bm t}^2 &=
 \left(\HH\D +\frac{k_p^2}{2}\HH\D_{-1}\right)
 \left(\HH\D +\frac{k_s^2}{2}\HH\D_{-1}\right)+\D_2 +\mathrm{OPS}(-2)
 \\
 &=
  -\frac{k_p^2+k_s^2}2\I +\mathrm{OPS}(-2)
  \end{aligned}
\]
which lead us to propose
\begin{equation}
 \label{def:eq:RD}
  \mathcal{R}_{\rm D} =-\frac{2}{\widetilde{k}_p^2+\widetilde{k}_s^2}\mathcal{H}_0\HH\D-\frac{1}{\widetilde{k}_p^2+\widetilde{k}_s^2}\begin{bmatrix}
                       \widetilde{k}_s^2\\
                         &- \widetilde{k}_p^2
                      \end{bmatrix} \HH\D_{-1}+\begin{bmatrix}
                      \mathrm{J}&\\
                      &\mathrm{J}
                      \end{bmatrix}
\end{equation}
as an approximation of $\mathcal{R}_{\rm D}^{\rm ex}$. It is easy to see that
\[
 \mathcal{R}_{\rm D} \begin{bmatrix} {\rm DtN}_p & \partial_{\ttt}\\ \partial_{\ttt} & -{\rm DtN}_s\end{bmatrix} = \mathcal{I}+
 \mathrm{OPS}(-1).
\]
Moreover ${\cal R}_{\rm D}:H^s(\Gamma)\times H^s(\Gamma)\to H^{s-1}(\Gamma)\times H^{s-1}(\Gamma)$, i.e.   ${\cal R}_{\rm D}\in \mathrm{OPS}(1)$, and it is injective since a Fourier matrix principal symbol is
\[
 \widehat{\mathcal{R}_{\rm D}}(0) := \begin{bmatrix}
                     1&\\
                      &1
                    \end{bmatrix},\qquad
\widehat{\mathcal{R}_{\rm D}}(n) =  \frac{ 2}{\widetilde{k}_p^2+\widetilde{k}_s^2}
 \begin{bmatrix}
                     |n|-\frac{\widetilde{k}_s^2}{2|n|}& -   n i \\
                     -   n i   &-|n|+\frac{ \widetilde{k}^2_p}{2 |n|}
 \end{bmatrix}
\]
which is clearly invertible, but not uniformly invertible, for all $n$.

\added[id=catB]{Alternatively, we can construct an approximation for $\mathcal{R}^{\rm ex}_{\rm D}$  via the 
Fourier multipliers $2\mathrm{PS}(\W_{\widetilde{k}_p})$ and $2\mathrm{PS}(\W_{\widetilde{k}_s})$  (cf. \eqref{eq:DTNptilde} and \eqref{eq:02:th:psido:exp} in Theorem \ref{th:psido:exp}) which can be easily defined as
\begin{equation}\label{eq:PSY}
 2\widehat{\mathrm{PS}(\W_{\widetilde{k}_p})}(n) =
-  (n^2- \widetilde{k}_p^2)^{1/2}
\qquad 2\widehat{\mathrm{PS}(\W_{\widetilde{k}_s})}(n) =
- (n^2- \widetilde{k}_s^2)^{1/2}  .
\end{equation}
\added[id=vD]{Note that with this definition we take advantage of the improved (in the sense of regularity of the remainder) approximation $\W_{\widetilde{k}_p}-\mathrm{PS}(\W_{\widetilde{k}_p}) \in \mathrm{OPS}(-3)$}.
}

Using formulas~\eqref{eq:PSY}, and taking into account \eqref{eq:DtN_W}--\eqref{eq:comp_waven}, we obtain  $\mathrm{OPS}(1)\ni \widetilde{\mathcal{R}}_{\rm D} \approx \mathcal{R}^{\rm ex}_{\rm D}$,  the Fourier multiplier of matrix principal symbol
\begin{equation}\label{eq:sqR}
\widehat{\widetilde{\mathcal{R}}_{\rm D}}(0) = \begin{bmatrix}
                                      1&\\
                                       & 1
                                     \end{bmatrix}, \quad
\widehat{\widetilde{\mathcal{R}}_{\rm D}}(n):=
\added[id=catB]{r_{\rm D}(n,\widetilde{k}_p,\widetilde{k}_s)} \begin{bmatrix} (n^2-  \widetilde{k}_s ^2)^{1/2} & -in \\ -in & -(n^2- \widetilde{k}_p ^2)^{1/2}\end{bmatrix},\quad n\ne 0,
\end{equation}
where
\begin{equation}\label{eq:Delta}
\added[id=catB]{r_{\rm D}(n,\widetilde{k}_p,\widetilde{k}_s)}:=\frac{1}{n^2-(n^2-\widetilde{k}_p^2)^{1/2}( n^2-\widetilde{k}_s^2)^{1/2}}=\frac{n^2+(n^2-\widetilde{k}_p^2)^{1/2}( n^2-\widetilde{k}_s^2)^{1/2}}{(\widetilde{k}_p^2+\widetilde{k}_s^2)n^2-\widetilde{k}_p^2\widetilde{k}_s^2}.
\end{equation}
Note that $\added[id=catB]{r_{\rm D}(n,\widetilde{k}_p,\widetilde{k}_s)}$ is well defined since the denominator has positive  imaginary part (cf. \eqref{eq:comp_waven}).

Furthermore, since
\[
\added[id=catB]{r_{\rm D}(n,\widetilde{k}_p,\widetilde{k}_s)}=
\frac{2}{\widetilde{k}_p^2+\widetilde{k}_s^2} -\frac{\widetilde{k}_p^2-\widetilde{k}_s^2}{2(\widetilde{k}_p^2+\widetilde{k}_s^2)n^2}+{\cal O}(n^{-4})
\]
we easily deduce
\begin{equation}\label{eq:4.14}
 \widetilde{\mathcal{R}}_{\rm D} = \mathcal{R}_{\rm D} -\frac{\widetilde{k}_p^2-\widetilde{k}_s^2}{2(\widetilde{k}_p^2+\widetilde{k}_s^2)}\mathcal{H}_0\HH\D_{-1}+\mathrm{OPS}(-2).
\end{equation}

Before entering into the analysis, let us note that
\begin{eqnarray*}
\mathcal{A} ^{\rm comb} & =&
\mathcal{A}_{\rm DL}-  2{\cal A}_{\rm SL} \begin{bmatrix}
                                        \W_{\widetilde{k}_p} & \\
                                            & \W_{\widetilde{k}_s}
                                       \end{bmatrix}\\
                                       &=&
  \frac{1}2 \mathcal{H}_0 \HH \D +\frac{1}4\begin{bmatrix}
                                            k_p^2&\\
                                                 & -k_s^2
                                           \end{bmatrix}\HH\D_{-1}
\\
&&+\left( \frac12 \mathcal{H}_0+\frac{1}4 \begin{bmatrix}
                                 &k_s^2\\
                                 k_p^2&
                                \end{bmatrix}\HH\D_{-2}
\right)  \left( \HH\D+ \frac{1}2\begin{bmatrix}
                                            \widetilde{k}_p^2&\\
                                                 &  \widetilde{k}_s^2
                                           \end{bmatrix}\HH\D_{-1}
                     \right)+\mathrm{OPS}(-2)
\\
&=&
\mathcal{A}^{\rm comb}_0+\mathrm{OPS}(-2),\\
&& \qquad\quad\mathcal{A}^{\rm comb}_0:=\mathcal{H}_0 \HH\D +\frac14 \begin{bmatrix}
                                 (k_p^2+\widetilde{k}_p^2)\I&  (k_s^2-\widetilde{k}_s^2)\HH\\
                                   (k_p^2-\widetilde{k}_p^2)\HH& -(k_s^2+ \widetilde{k}_s^2)\I
                                \end{bmatrix}\HH \D_{-1}.
\end{eqnarray*}

\begin{lemma}\label{lemma:4.2}
The CFIER operators
\begin{equation}
\begin{aligned}
{\cal A}_{\rm CFIER}&:=
{\cal A}^{\rm comb}{\cal R}_{\rm D},\qquad
\widetilde{\cal A}_{\rm CFIER}&:=
{\cal A}^{\rm comb}\widetilde{\cal R}_{\rm D}.
\end{aligned}
\end{equation}
are compact perturbations of an invertible
\added[id=vD]{Fourier multiplier operator}
of order 0.
\end{lemma}
\begin{proof}Clearly (recall that $\mathcal{H}_0$ commutes with $\HH$ and $\D$ and that $\mathcal{H}_0^2=0$),
\begin{eqnarray*}
   {\cal A}_{\rm CFIER} &=&
   \frac{1}{\widetilde{k}_p^2+\widetilde{k}_s^2} {\cal A}_0^{\rm comb}
                                \left(
                              2\mathcal{H}_0\D+ \begin{bmatrix}
                         \widetilde{k}_s^2\\
                         & -\widetilde{k}_p^2
                      \end{bmatrix} \D_{-1}\right)
                     +\mathrm{OPS}(-1)\\
                      &=&
  \frac{1}{2(\widetilde{k}_p^2+\widetilde{k}_s^2)}
  \begin{bmatrix}
  (\alpha+\widetilde{\alpha})\I &-(\alpha-\widetilde{\alpha})  \HH\\
(\alpha-\widetilde{\alpha})\HH        &
(\alpha+\widetilde{\alpha})
 \I
 \end{bmatrix}
%
+\mathrm{OPS}(-1)
\end{eqnarray*}
where
\[
 \alpha =k_p^2+k_s^2,\quad
 \widetilde{\alpha}=\widetilde{k}_p^2+\widetilde{k}_s^2.
\]
Besides, by \eqref{eq:4.14},
\begin{eqnarray*}
\widetilde{\cal A}_{\rm CFIER}&=&
{\cal A}^{\rm comb}\widetilde{\cal R}_{\rm D}=\mathcal{A}^{\rm comb}\mathcal{R}_{\rm D}-\frac{\widetilde{k}_p^2-\widetilde{k}_s^2}{2(\widetilde{k}_p^2+\widetilde{k}_s^2)}
\mathcal{A}^{\rm comb} \mathcal{H}_0\HH\D_{-1} +
                                \mathrm{OPS}(-1)\\
&=&
 {\cal A}_{\rm CFIER}\\
 &&+ \frac{\widetilde{k}_p^2-\widetilde{k}_s^2}{2(\widetilde{k}_p^2+\widetilde{k}_s^2)}\underbrace{\left(\mathcal{H}_0+\frac14 \begin{bmatrix}
                                 (k_p^2+\widetilde{k}_p^2)\I&  (k_s^2-\widetilde{k}_s^2)\HH\\
                                   (k_p^2-\widetilde{k}_p^2)\HH& -(k_s^2+ \widetilde{k}_s^2)\I
                                \end{bmatrix}  \D_{-2}\right)
                               \mathcal{H}_0}_{\in\mathrm{OPS}(-2)} +
                                \mathrm{OPS}(-1).
\end{eqnarray*}
In short,
\[
\begin{aligned}
 \widetilde{\cal A}_{\rm CFIER}\ = \ &{\cal A}_{\rm CFIER}+\mathrm{OPS}(-1) \\
 =\
 &\frac{1}{2(\widetilde{k}_p^2+\widetilde{k}_s^2)}
  \begin{bmatrix}
  (\alpha+\widetilde{\alpha})\I &-(\alpha-\widetilde{\alpha})  \HH\\
(\alpha-\widetilde{\alpha})\HH        &
(\alpha+\widetilde{\alpha})
 \I\end{bmatrix}+\mathrm{OPS}(-1).
\end{aligned}
\]

Since
\[
 (\alpha+ \widetilde{\alpha})^2-  (\alpha- \widetilde{\alpha})^2 =4\alpha\widetilde{\alpha} =4(k_p^2+k_s^2)(\widetilde{k}_p^2+\widetilde{k}_s^2)\ne 0
\]
we conclude, cf. \eqref{eq:inverH2}, that the principal part is an invertible operator of order 0 that finishes the proof.
\end{proof}

\begin{theorem}\label{thm_inv} The operators ${\cal A}_{\rm CFIER}$ and $\widetilde{\cal A}_{\rm CFIER}$  are invertible pseudodifferential  operators of order zero.
\end{theorem}
\begin{proof}
Since  $ {\mathcal{R}}_{\rm D}$ and $\widetilde{\mathcal{R}}_{\rm D}$ are injective, by Fredholm alternative, and  in view of Lemma \ref{lemma:4.2}, it suffices to show that $\mathcal{A} ^{\rm comb} $ cf. \eqref{eq:RegFormulation} is injective too.  Let $(\varphi_p,\varphi_s)\in \mathop{\rm Ker}({\cal A}_{\rm comb})$ and define $u_p,\ u_s:\mathbb{R}^2\setminus\Gamma\to\mathbb{C}$ given by
\begin{eqnarray*}
 u_r&=& {\rm DL}_{r} [\varphi_r] -2{\rm SL}_{r}[ \W_{\widetilde{k}_r}\varphi_r],\qquad r\in\{p,s\}
 \end{eqnarray*}
From the jump relations cf. \eqref{eq:jumps} it follows
\added[id=catR]{
\[
 \added[id=vD]{\llbracket}\partial_{\nnn} u_p\added[id=vD]{\rrbracket} -2\W_{\widetilde{k}_p} \added[id=vD]{\llbracket}\gamma u_p\added[id=vD]{\rrbracket} =0,\quad
 \added[id=vD]{\llbracket}\partial_{\nnn} u_s\added[id=vD]{\rrbracket} -2\W_{\widetilde{k}_s} \added[id=vD]{\llbracket}\gamma u_s\added[id=vD]{\rrbracket} =0.
\]
}
But on account of the uniqueness of solutions of the elastic scattering problem with Dirichlet boundary conditions in $\Omega^+$, we obtain via the same arguments as in Theorem~\ref{CFIED} that $u_p=0$ and $u_s=0$ in $\Omega^+$. Hence, \added[id=catR]{since $\gamma^+u_p=\gamma^+ u_s=0$ and $\partial_{\nnn}^+ u_p=\partial_{\nnn}^+ u_s=0$, we derive}
\[
\partial_{\nnn}^- u_p-2\W_{\widetilde{k}_p}[\gamma^- u_p]=0 \qquad \partial_{\nnn}^- u_s-2\W_{\widetilde{k}_s}[\gamma^- u_s]=0\quad {\rm on}\ \Gamma.
\]
\added[id=catB]{
Applying Green's identities in the domain $\Omega$ we obtain the following identities by taking into account the fact the wave-numbers $k_p$ and $k_s$ are real
\[
 \Im \int_\Gamma \partial_{\nnn}^- u_p\ \overline{\gamma^{-}u}_p=
 \Im \int_\Gamma \partial_{\nnn}^- u_s\ \overline{\gamma^{-}u}_s=0.
\]
Trading the interior Neumann traces in these two equations with the hypersingular operators we derive
\[
 \Im \int_\Gamma \W_{\widetilde{k}_p}[\gamma^- u_p]\  \overline{\gamma^{-}u}_p=
 \Im \int_\Gamma \W_{\widetilde{k}_s}[\gamma^- u_s]\ \overline{\gamma^{-}u}_s=0.
\]
Given the coercivity property~\cite{turc2}
\[
\Im \int_\Gamma \W_{\widetilde{k}}[f]\ \overline{f} >0,\qquad \Im\widetilde{k}>0,\qquad f\neq 0
\]
}
we conclude that $\gamma^{-}u_p=\partial_{\nnn}^{-}u_p=0$ and $\gamma^{-}u_s=\partial_{\nnn}^{-}u_s=0$, from which, in turn, it follows that $u_p=0$ and $u_s=0$ in $\Omega$. Consequently, it follows that $\varphi_p=0$ and $\varphi_s=0$.
\end{proof}

In an effort to simplify the regularizing operator, we can consider an alternative regularizer ${\cal R}_{{\rm D},1}\in \mathrm{OPS}(1)$ which is defined as the Fourier multiplier matrix operator whose symbol is given by the formula
\begin{equation}\label{eq:qaltR}
\widehat{\mathcal{R}_{{\rm D},1}}(n):=\frac{2}{\widetilde{k}_p^2+\widetilde{k}_s^2}
\begin{bmatrix}
(n^2- \widetilde{k}_s ^2)^{1/2} & -in \\
-in & -(n^2- \widetilde{k}_p ^2)^{1/2}\end{bmatrix}  .
\end{equation}
\added[id=vD]{Notice that this operator}
defines also  a well posed CFIER formulation
since
\[
 \mathcal{R}_{{\rm D},1}-\mathcal{R}_{{\rm D}} \added[id=vD]{\in}\mathrm{OPS}(-2)
\]
(see \eqref{def:eq:RD}) and therefore (see Lemma \ref{lemma:4.2})
\[
{\cal A}^{\rm comb}\mathcal{R}_{{\rm D},1}={\cal A}_{\rm CFIER} + \mathrm{OPS}(-1).
\]

Interestingly, it is possible to propose a different regularizing operator that involves Helmholtz BIOs rather than Fourier multipliers. Indeed, since the operator
\begin{equation}\label{eq:qaltR1}
\mathcal{R}_{{\rm D},2}:= \frac{1}{\widetilde{k}_p^2+\widetilde{k}_s^2} \begin{bmatrix}2\W_{\widetilde{k}_p} & \partial_{\ttt}\\ \partial_{\ttt} & -2 \W_{\widetilde{k}_s}\end{bmatrix}.
\end{equation}
has the property (cf. Theorem \ref{th:psido:exp}) $
\mathcal{R}_{{\rm D},2}-\mathcal{R}_{{\rm D},1}\in \mathrm{OPS}(-3)$,
we can see that the ensuing CFIER formulations based on the operator $\mathcal{R}_{{\rm D},2}$ leads again to Fredholm operators of index 0. We establish the invertibility of the CFIER operator in  the following result
\begin{lemma}  $\mathcal{R}_{{\rm D},2}$ is an injective pseudodifferential operator of order~$1$.
\end{lemma}
\begin{proof}
Assume $(\varphi_p, \varphi_s)\in \mathop{\rm Ker}(\mathcal{R}_{{\rm D},2})$, that is
\begin{eqnarray*}
2\W_{\widetilde{k}_p}[\varphi_p]+\partial_{\ttt} \varphi_s&=&0\\
\partial_{\ttt}\varphi_p-2\W_{\widetilde{k}_s}[\varphi_s]&=&0.
\end{eqnarray*}
Then
\begin{eqnarray*}
0
&=& 2\int_\Gamma  \overline{\varphi_p}\,  {\W_{\widetilde{k}_p}[\varphi_p]} + 2\int_\Gamma \overline{\varphi}_s \,\W_{\widetilde{k}_s}[\varphi_s] +\int_\Gamma \overline{\varphi_p} \partial_{\ttt}\varphi_s  + \int_\Gamma \overline{\partial_{\ttt} \varphi}_s  \varphi_p.
\end{eqnarray*}
(Notice that  integration by parts  has been applied in  the last step).  Then,
\[
\Im \left(\int_\Gamma  \overline{\varphi}_p\,  {\W_{\widetilde{k}_p}[\varphi_p]} +  \int_\Gamma \overline{\varphi}_s \,\W_{\widetilde{k}_s}[{\varphi}_s]\right)=0\added[id=vD]{.}
\]
Given that for    $\Im\widetilde{k}>0$
\[
\int_\Gamma  \W_{\widetilde{k}}[\varphi]\  \overline{\varphi} >0,\qquad \text{for }\varphi\ne 0
\]
we conclude $\varphi_p=\varphi_s=0$, which completes the proof.
\end{proof}
\begin{remark} It is possible to replace the {\rm DtN} approximations ${\rm DtN}_p\approx 2\W_{\widetilde{k}_p}$  and ${\rm DtN}_s\approx 2\W_{\widetilde{k}_s}$ by the Fourier multipliers $2 \mathrm{PS}[\W_{\widetilde{k}_p}]$ and respectively $2 \mathrm{PS}[\W_{\widetilde{k}_s}]$ defined in equations~\eqref{eq:PSY} in the construction of the CFIER operators considered above. In that case, the CFIER operators are of the form
\begin{equation}\label{eq:CFIERY}
{\cal A}_{\rm CFIER}^{\rm PS}=\left({\cal A}_{\rm DL} -2{\cal A}_{\rm SL} \begin{bmatrix}
                                         \mathrm{PS}[\W_{\widetilde{k}_p}]& \\
                                            & \mathrm{PS}[\W_{\widetilde{k}_s}]
                                       \end{bmatrix}\right){\cal R}_{\rm D}.
\end{equation}
Since $\mathrm{PS}[\W_{\widetilde{k}_p}]-\W_{\widetilde{k}_p}\in \mathrm{OPS}(-3)$ and $\mathrm{PS}[\W_{\widetilde{k}_s}]-\W_{\widetilde{k}_s}\in \mathrm{OPS}(-3)$, and $\Im \mathrm{PS}[\W_{\widetilde{k}_p}]>0,\ \Im \mathrm{PS}[\W_{\widetilde{k}_s}]>0$
the CFIER formulations define in equations~\eqref{eq:CFIERY} are well posed.
\end{remark}

\subsection{Neumann Problem}\label{neuP}

\subsubsection{Double and single layer formulations}
We turn our attention to the Neumann problem where the Helmholtz fields $u_p$ and $u_s$ in the Helmholtz decomposition~\eqref{eq:Hdecomp2} satisfy the system of boundary conditions~\eqref{eq:NH}. Looking for Helmholtz double layer potential representations of  $u_p$ and $u_s$ in $\Omega^+$ of the form
\[
 u_p ={\rm DL}_p\,[\varphi_p],\quad
 u_s ={\rm DL}_s\,[\varphi_s]
\]
the system of boundary conditions~\eqref{eq:NH} can be recast as a system of BIE involving the matrix operator ${\cal B}_{\rm DL}$ defined as:
{\small
\begin{equation}\label{eq:BDL}
 \begin{aligned}
   {\cal B}_{\rm DL} & \begin{bmatrix}
                   \varphi_p\\
                   \varphi_s
                  \end{bmatrix}
                  \\
\ :=\  &  \mu
  \begin{bmatrix}
  - \partial_{\ttt}^2 \varphi_p   &    2{\bm t}\cdot    \W_s[\partial_{\ttt} \varphi_s\,\ttt]  \\
2{\bm t}\cdot
\W_p[\partial_{\ttt} \varphi_p\,\ttt]   &    \partial_{\ttt}^2 \varphi_s
\end{bmatrix}+\mu
  \begin{bmatrix}   2{\bm n}\cdot
\W_p[\partial_{\ttt} \varphi_p\,\ttt]   &  -\kappa \partial_{\ttt}  \varphi_s \\
 -\kappa \partial_{\ttt}  \varphi_p &  -2{\bm n}\cdot
\W_s[\partial_{\ttt} \varphi_s\,\ttt]
\end{bmatrix} \\
&  -\mu
  \begin{bmatrix}
   k_p^2 \varphi_p  & \\
   & -k_s^2\varphi_s
\end{bmatrix}  +2\mu
  \begin{bmatrix}
  \nnn\cdot   \K_p^\top [\partial_{\ttt}^2\varphi_p\,\nnn]   &
  \\
  & -  \nnn\cdot   \K_s^\top [\partial_{\ttt}^2 \varphi_s\,\nnn]
   \end{bmatrix}\\
   &
   +2\mu
  \begin{bmatrix}   &
  {\bm t}\cdot    \left(\K_p^\top [\partial_{\ttt}^2 \varphi_p\,\nnn]+\K_p^\top[\partial_{\ttt}\varphi\,{\bm t}]\right)  \\
  {\bm t}\cdot    \left(\K_s^\top [\partial_{\ttt}^2 \varphi_s\,\nnn]+\K_s^\top[\partial_{\ttt}\varphi\,{\bm t}]\right)  &
   \end{bmatrix}\\
   &
    +2\mu  \begin{bmatrix}
  \nnn\cdot    \left(k_p^2\K_p^\top[\varphi_p\,\nnn] +\K_p^\top [\kappa\partial_{\ttt}\varphi_p\ttt] \right) &  \\
  &  - \nnn\cdot  \left( k_s^2\K_s^\top[\varphi_s\,\nnn] +\K_s^\top [\kappa\partial_{\ttt}\varphi_s \ttt] \right)
\end{bmatrix}   \\
 &  +  2\mu \begin{bmatrix}
 &   k_s^2 {\bm t}\cdot  \K_s^\top[\varphi_p\,\nnn] \\
k_p^2 {\bm t}\cdot  \K_p^\top[\varphi_s\,\nnn]   &
\end{bmatrix}
-
\begin{bmatrix}
  \lambda k_p^2(\frac{1}2\varphi_p+\K_p [\varphi_p])  \\
  &
  \mu k_s^2( \frac{1}2\varphi_s +\K_s [\varphi_s])
              \end{bmatrix}.
 \end{aligned}
\end{equation}
}
Applying the results in \eqref{eq:02:prop:principalPart:Hess} in  Proposition \ref{prop:principalPart:Hess}  to \eqref{eq:NH} we can show that
\begin{equation}\label{eq:BDL2}
{\cal B}_{\rm DL} :=
 -\mu {\cal H}_0 \D_2 -\mu {\cal H}_0 \HH\kappa\D-\frac{1}2 \mu
 \begin{bmatrix}
  2k_p^2 \I & -k_s^2 \mathrm{H}\\
  -k_p^2 \mathrm{H} &  -k_s^2 \I
 \end{bmatrix}
-\frac{\lambda}2\begin{bmatrix}
                 k_p^2  \I &0\\
                 0&0
                \end{bmatrix}+\mathrm{OPS}(-1)
\end{equation}
\added[id=catB]{where we recall that $\kappa$ denotes the signed curvature on $\Gamma$ defined in equations~\eqref{eq:def:kappa}.} On the other hand, if  we seek fields $u_p$ and $u_s$ in the form of Helmholtz single layer potentials with wave-numbers $k_p$ and $k_s$, that is
\[
 u_p ={\rm SL}_p\,[\varphi_p],\quad
 u_s ={\rm SL}_s\,[\varphi_s],
\]
the system of boundary conditions~\eqref{eq:NH} gives rise to the system of BIE involving the matrix operator ${\cal B}_{\rm SL}$ defined as (see again Theorem \ref{th:3.6})
\begin{equation}\label{eq:BSL}
 \begin{aligned}
 {\cal B}_{\rm SL}&\begin{bmatrix}
                   \varphi_p\\
                   \varphi_s
                  \end{bmatrix}\\
\ =\ & \mu
                               \begin{bmatrix}
           -2 \nnn\cdot
 \W_k[\varphi_p\,\nnn] & - \partial_{\ttt} \varphi_s \\
- \partial_{\ttt} \varphi_p
  &    2 \nnn\cdot
 \W_s [\varphi_s\,\nnn]
             \end{bmatrix}
+\mu
             \begin{bmatrix}
             \kappa \varphi_p   & -2{\bm t}\cdot
\W_k[\varphi_s\,\nnn] \\
-2{\bm t}\cdot
\W_k[\varphi_p\,\nnn]
  & - \kappa \varphi_s
             \end{bmatrix}\\
             &
             +2\mu         \begin{bmatrix}
               - \nnn\cdot   \K_p^\top [\kappa \varphi_p\,\nnn] &   {\bm t}\cdot \K_s^\top [\varphi_s\,\ttt]  \\
                \ttt\cdot   \K_p^\top [\varphi_p\,\ttt] &   {\bm n}\cdot    \K_s^\top [\varphi_s\,\nnn]
               \end{bmatrix}
             +2\mu\begin{bmatrix}
                 \nnn\cdot  \K_p^\top [\varphi_p\,\ttt] & - {\bm t}\cdot    \K_s^\top [\kappa \varphi_s\,\nnn]\\
            - {\bm t}\cdot    \K_p^\top [\kappa \varphi_p\,\nnn]&
            -\nnn\cdot  \K_s^\top [\varphi_s\,\ttt]
                \end{bmatrix}  \\
                &
                -\begin{bmatrix}
                  \lambda  k_p^2\V_p[\varphi_p]\\
                  &\mu  k_s^2\V_s[\varphi_s]
                 \end{bmatrix}.\end{aligned}
 \end{equation}
Proposition \ref{prop:principalPart:Hess}, identities  \eqref{eq:01:prop:principalPart:Hess},  implies now
 \begin{equation}\label{eq:BSL2}
 {\cal B}_{\rm SL}=-\mu{\cal H}_0\HH\D+\mu{\cal H}_0\kappa-\frac{1}{2}(\lambda+\mu)k_p^2\begin{bmatrix}\I & 0\\ 0& 0\end{bmatrix}\HH\D_{-1}+\mathrm{OPS}(-2).
 \end{equation}
Looking for fields $u_p$ and $u_s$ in the combined field form~\eqref{eq:comb_field} we are led to a CFIE formulation for the Neumann problem with the underlying operator
\[
{\cal B}_{\rm CFIE} =  {\cal B}_{\rm DL}  -i k{\cal B}_{\rm SL},\quad \added[id=catB]{k>0.}
\]
In this case, following the same approach as in the Dirichlet case, we obtain by combining formulas~\eqref{eq:BDL2} and~\eqref{eq:BSL2} 
\begin{equation}\label{eq:BCFIE}
\begin{aligned}
{\cal B}_{\rm CFIE}^2&\ =\  \frac{1}{2} \mu
\begin{bmatrix}
  \added[id=catB]{(}k_p^2 \left(2 \lambda +3 \mu
   )-k_s^2 \mu \right)\I & -
   \left(\omega^2-k_s^2 \mu \right)\HH \\
 -     \left(k_p^2 (\lambda +2 \mu
   )-k_s^2 \mu \right)\HH &  \mu
   \left(k_s^2-k_p^2\right)\I
\end{bmatrix}\mathrm{D}_2+\mathrm{OPS}(1)
\\
&\ =\
 \frac{\mu(\mu+\lambda)}{2(\lambda+2\mu)}\omega^2 \begin{bmatrix}
 \I& \\
 &\I
 \end{bmatrix}
 (\D ^2-\I)+\mathrm{OPS}(1).\quad
 \text{(by \eqref{eq:ks:kp})}
\end{aligned}
\end{equation}
(We have used the definition of the \added[id=catB]{wave-numbers} $k_p$ and $k_s$ given in \eqref{eq:ks:kp}).
Since the leading order operator in formulas~\eqref{eq:BCFIE} is invertible, because so is $\D ^2-\I$, the invertibility of the CFIE operator can be established using identical arguments as in the Dirichlet case (the only notable difference is that the uniqueness of solutions of Navier equations with Neumann boundary conditions in $\Omega^+$ has to be invoked). Note that both operators ${\cal B}_{\rm CFIE}\in \mathrm{OPS}(2)$ and ${\cal B}_{\rm CFIE}^2\in \mathrm{OPS}(2)$, and as such the CFIE formulations are not of the second kind. We describe in what follows a method that delivers robust BIE formulations of the second kind in the Neumann case.

\begin{remark} As in the Dirichlet case we can be tempted to use the combined potential \eqref{remarkCFIED} but the principal part of the square resulting operator is not invertible since it is given by
$
-i (k_p -k_s )\mathcal{H}_0 \HH\D_3.
$
\end{remark}

\subsubsection{Regularized formulations}

The main idea in the derivations of regularized BIE formulations for the boundary value system~\eqref{eq:NH} is the construction of suitable approximations to the operator ${\cal R}_{\rm N}^{\rm ex}$ that maps the left hand side of the system~\eqref{eq:NH} to the Cauchy data of $(u_p,u_s)$ on $\Gamma$. Starting with the formula$
\nabla=\nnn \partial_{\nnn}  +\ttt \partial_{\ttt}$
we get
\begin{equation}\label{eq:twonabla}
\begin{aligned}
\added[id=vD]{\Hes}&=\nabla\nabla^\top=(\nnn \partial_{\nnn}  +\ttt \partial_{\ttt} )(\nnn^\top  \partial_{\nnn} +\ttt^\top \partial_{\ttt}  )
\nonumber\\
&=
\nnn \nnn^\top \partial_{\nnn}^2  +\ttt \ttt^\top \partial_{\ttt}^2+\nnn \ttt^\top \partial_{\nnn}\partial_{\ttt} +\ttt \nnn^\top \partial_{\ttt}\partial_{\nnn} -\kappa\nnn  \nnn^\top \partial_{\nnn} +\kappa \ttt\ttt^\top \partial_{\nnn}.
\end{aligned}
\end{equation}
\added[id=catB]{
Again, using {\rm DtN} operators in the formula  {for the Hessian} above whenever the normal derivative $\partial_{\nnn}$ appears and neglecting the lower order contributions that contain the curvature,  we have, {\em formally}, that from \eqref{eq:NH}}
\begin{equation}\label{eq:Rn1}
\begin{bmatrix}
 \gamma^+u_p\\
  \gamma^+u_s
\end{bmatrix} =
 {\cal R}_N^{\rm ex}
 \begin{bmatrix}
-T{\bf u}^{\rm inc} \cdot \nnn\\
-T{\bf u}^{\rm inc} \cdot \ttt
 \end{bmatrix}
+\mathrm{OPS}(-1),\quad  {\cal R}_{\rm N}^{\rm ex} :=
\begin{bmatrix}2\mu {\rm DtN}_p^2 -\lambda k_p^2 \I & 2\mu \partial_{\ttt} {\rm DtN}_s \\ 2\mu \partial_{\ttt} {\rm DtN}_p & -2\mu {\rm DtN}_s^2-\mu k_s^2 \I
\end{bmatrix}^{-1}.
\end{equation}

Straightforward calculations, with the usual approximation for  operators ${\rm DtN}_k \approx 2\W_k$ and \eqref{eq:DtN_W} , shows that \deleted{, at least formally,}
\[
\begin{aligned}
({\cal R}_{\rm N}^{\rm ex})^{-1} -\mathrm{PS}(\widehat{({\cal R}_{\rm N}^{\rm ex})^{-1}})\ \in\ &\mathrm{OPS}(-1)
\\
\widehat{\mathrm{PS}}\left(\widehat{({\cal R}_{\rm N}^{\rm ex})^{-1}}\right)(n) =\ &
2\mu \begin{bmatrix}
n^2-\frac12\added[id=vD]{k_p^2} & -in(n^2- {k}_s^2)^{1/2} \\[1ex]
-in(n^2- {k}_p^2)^{1/2} & \frac12k_s^2-n^2\end{bmatrix}.
\end{aligned}
\]
This suggests us the following regularizer
\begin{equation}\label{eq:Rn2}
\begin{aligned}
 \widehat{{\cal R}_{\rm N}}(n)\ :=\ &
 \frac{1}{2\mu}
 \begin{bmatrix}
n^2-\frac1{2}\added[id=vD]{k_p^2} & -in(n^2- \widetilde{k}_s^2)^{1/2} \\[1.1ex]
-in(n^2- \widetilde{k}_p^2)^{1/2} & \frac1{2}k_s^2-n^2
\end{bmatrix} ^{-1}\\
=\ & \frac{1}{2\mu }  \added[id=vD]{r_{\rm N}} (n,k_p,k_s,\widetilde{k}_p,\widetilde{k}_s)
\begin{bmatrix}
-(n^2-\frac1{2}k_s^2) &  in(n^2- \widetilde{k}_s^2)^{1/2} \\[1.ex]
in(n^2- \widetilde{k}_p^2)^{1/2} &   n^2-\frac1{2}\added[id=vD]{k_p^2}
\end{bmatrix}
\end{aligned}
\end{equation}
We have {\em complexified} $\widetilde{k}_p$ and $\widetilde{k}_s$ in the off-diagonal terms to ensure the well definiteness, i.e. the invertibility of the matrix since in this case
\[
 \added[id=vD]{r_{\rm N}} (n,k_p,k_s,\widetilde{k}_p,\widetilde{k}_s)=-\left(n^2-\frac{1}{2 }k_p^2\right)\added[id=vD]{\left(n^2-\frac{1}{2 }k_s^2\right)} +n^2(n^2-\widetilde{k}_p^2)^{1/2} (n^2-\widetilde{k}_s^2)^{1/2}\ne 0.
\]
(Notice that   $\Im \added[id=vD]{r_{\rm N} (n,k_p,k_s,\widetilde{k}_p,\widetilde{k}_s)} <0$).
%

Then it can be proved after some direct but tedious calculations that
\begin{equation}\label{eq:formRN}
\begin{aligned}
 {\cal R}_{\rm N} \  =\  &
 \frac{1}{2\mu}\left(\frac{2}{ \widetilde{k}_p^2 -2k_s^2+\widetilde{k}_s^2 } \mathcal{I} +
 \frac{2k_s^4+(\widetilde{k}_p^2-\widetilde{k}_s^2)^2}{(\widetilde{k}_p^2 -2k_s^2+\widetilde{k}_s^2 )^2}
 \D_{-2}\right)\mathcal{H}_0
 \\
 \quad &-
  \frac{1}{2\mu(\widetilde{k}_p^2 -2k_s^2+\widetilde{k}_s^2)}
  \begin{bmatrix}
  k_s^2\I & -\widetilde{k}_s^2\HH\\
  -\widetilde{k}_p^2\HH & -k_s^2\I
 \end{bmatrix}
 \D_{-2}
 +\mathrm{OPS}(-4)
\end{aligned}
\end{equation}

\begin{remark}
It is possible to used complexified wave-numbers for the definition of the Fourier multipliers in equations~\eqref{eq:PSY} when we consider approximations of {\rm DtN} operators in equation~\eqref{eq:Rn1} and to consider only the pseudodifferential operators of order 2 that constitute the leading order contribution to the expression on the left hand side of equation~\eqref{eq:NH}. In that case, we are led to solve the following matrix equation
\begin{equation}\label{eq:Rn3}
2\mu\begin{bmatrix}n^2-\widetilde{k}_p^2 &- in(n^2-\widetilde{k}_s^2)^{1/2} \\
-in(n^2-\added[id=vD]{\widetilde{k}_p^2})^{1/2} & \widetilde{k}_s^2-n^2
\end{bmatrix}\widehat{{\cal R}_{\rm N}}(n)=\mathcal{I}.
\end{equation}
Although such a choice works equally well in practice, it is more difficult to establish the unique solvability of equations~\eqref{eq:Rn3}. Using the actual wave-numbers $k_p$ and $k_s$ when approximating the operators ${\rm DtN}_p^2$ and respectively ${\rm DtN}_s^2$ is also more intuitive if we write the Laplacian operator $\Delta$ in the $(\nnn,\ttt)$ frame and taking into account the fact that $u_p$ and $u_s$ are solutions of the Helmholtz equation with those wave-numbers.
 \end{remark}

Using the regularizing operator defined in equation~\eqref{eq:Rn2} we employ the same strategy as in the Dirichlet case and we arrive at the CFIER formulations
\begin{equation}\label{eq:precond:Neumann}
 {\cal B} ^{\rm comb}{\cal R}_{\rm N}\begin{bmatrix}
                        g^R_p\\
                        g^R_s
                       \end{bmatrix}
= -\begin{bmatrix}
T  {\bf u}^{\rm inc}\cdot{\bm n}\\
T  {\bf u}^{\rm inc}\cdot{\bm t}
   \end{bmatrix}
\end{equation}
with, see \eqref{eq:BDL2},  \eqref{eq:BSL} and \eqref{eq:DTNptilde},
\begin{equation}\label{eq:Bcomb}
\begin{aligned}
  {\cal B} ^{\rm comb}&\ :=\ {\cal B}_{\rm DL} -2{\cal B}_{\rm SL}
 \begin{bmatrix}
    \W_{\widetilde{k}_p} & \\
                         & \W_{\widetilde{k}_s}
                                       \end{bmatrix}
\\
=\ &-2\mu\mathcal{H}_0(\D_2+\kappa \HH\D)-\frac12\mu \mathcal{H}_0
\begin{bmatrix}
 \widetilde{k}_p^2&\\
      &\widetilde{k}_s^2
\end{bmatrix}
+\frac14 \begin{bmatrix}
          -3 \omega^2\I&2 k_s^2 \mu \HH  \\
 2 k_p^2 \mu  \HH&  3 k_s^2 \mu  \I\\
         \end{bmatrix} +\mathrm{OPS}(-1).
         \end{aligned}
\end{equation}

\begin{lemma}
 It holds
 \begin{equation}\label{eq:01:lemma:4.8}
  \begin{aligned}
{\cal B} ^{\rm comb}{\cal R}_{\rm N}
 = &   \begin{bmatrix}
 \left(2 \mu  \left(-3 \widetilde{k}_p^2+3 k_s^2+\widetilde{k}_s^2\right)-3\omega^2\right)\I &   \left( 3\omega^2 +2 \mu
   \left(\widetilde{k}_p^2+k_s^2-3 \widetilde{k}_s^2\right)\right)\HH \\
   \mu  \left(2 k_p^2+6 \widetilde{k}_p^2-7 k_s^2-2 \widetilde{k}_s^2\right)\HH &
   \mu  \left(2 k_p^2+2 \widetilde{k}_p^2+k_s^2-6 \widetilde{k}_s^2\right) \I
         \end{bmatrix}
         \\
    &         +\mathrm{OPS}(-1).
\end{aligned}
 \end{equation}
 Furthermore, the principal part is an \added[id=vD]{invertible Fourier multiplier operator of order zero.}
\end{lemma}
\begin{proof}
Using \eqref{eq:formRN} (recall again that ${\cal H}_0^2=0$) we conclude
\[
\begin{aligned}
   \bigg({\cal B}_{\rm DL}& -2{\cal B}_{\rm SL} \begin{bmatrix}
 \W_{\widetilde{k}_p} & \\
                         & \W_{\widetilde{k}_s}
 \end{bmatrix}
\bigg){\cal R}_{\rm N}\\
& \ =\
\frac{1}{  \widetilde{k}_p^2 -2k_s^2+\widetilde{k}_s^2}
\bigg(
  \mathcal{H}_0
  \begin{bmatrix}
  k_s^2\I & -\widetilde{k}_s^2\HH\\
  -\widetilde{k}_p^2\HH & -k_s^2\I
 \end{bmatrix}
-\frac12 {\cal H}_0\begin{bmatrix}
 \widetilde{k}_p^2&\\
      &\widetilde{k}_s^2
\end{bmatrix}{\cal H}_0
\\
& \qquad
 -
\frac{1}{4\mu}\begin{bmatrix}
          -3 k_p^2 (\lambda +2 \mu )\I&2 k_s^2 \mu \HH  \\
 2 k_p^2 \mu  \HH&  3 k_s^2 \mu  \I\\
         \end{bmatrix}\mathcal{H}_0
         \bigg) +\mathrm{OPS}(-1)
\end{aligned}
\]
from where \added[id=vD]{\eqref{eq:01:lemma:4.8} follows after some direct manipulations}.
The principal part is invertible since the {\em determinant} of the principal part \deleted{(cf. \eqref{eq:inverH2})} is given (see Lemma below) by
\[
\frac{(3 \lambda +4 \mu )k_p^2 +k_s^2 \mu }{4 \mu
   \left(\widetilde{k}_p^2-2 k_s^2+\widetilde{k}_s^2\right)}=\frac{(2 \lambda +3 \mu )k_p^2}{2  (\mu
    (\widetilde{k}_p^2+\widetilde{k}_s^2 )-2 \omega^2 )}\ne 0.
\]
\added[id=vD]{(Notice that in these last calculations we have used that $\omega^2 = (\lambda+2\mu)k_p^2 = \mu k_s^2$.)}
\end{proof}

\begin{lemma}
The matrix
{\small
 \[
  A = \frac{1}{4\mu (\widetilde{k}_p^2 -2k_s^2+\widetilde{k}_s^2)}\begin{bmatrix}
 2 \mu  \left(-3 \widetilde{k}_p^2+3 k_s^2+\widetilde{k}_s^2\right)-3\omega^2 &   \left( 3\omega^2 +2 \mu
   \left(\widetilde{k}_p^2+k_s^2-3 \widetilde{k}_s^2\right)\right) i \\
   \mu  \left(2 k_p^2+6 \widetilde{k}_p^2-7 k_s^2-2 \widetilde{k}_s^2\right)i &
   \mu  \left(2 k_p^2+2 \widetilde{k}_p^2+k_s^2-6 \widetilde{k}_s^2\right)
         \end{bmatrix}
 \]
 }
 is invertible with
 \[
  \det A = \frac{k_p^2 (3 \lambda +4 \mu )+k_s^2 \mu }{4 \mu
   \left(\widetilde{k}_p^2-2 k_s^2+\widetilde{k}_s^2\right)}\ne 0.
 \]
\end{lemma}
\begin{proof}
 Notice that
 \[
  A \begin{bmatrix}
     i\\
     1
    \end{bmatrix}
=- \begin{bmatrix}
     i\\
     1
    \end{bmatrix}.
 \]
 Hence $
  \det A = -(1+\mathop{\rm Tr}(A))$
 from where the result follows.
\end{proof}

 Just like in the Dirichlet case (see \eqref{eq:qaltR}), we can construct simpler regularizing operators starting from equation~\eqref{eq:Rn2} and using asymptotic arguments. For instance, let us consider the alternative regularizing operator $ {\cal R}_{{\rm N},1}\in \mathrm{OPS}(0)$ defined as the Fourier multiplier whose symbol is given by
 \begin{equation}\label{eq:RN1}
  \widehat{{\cal R}_{{\rm N},1}}(n)=\begin{bmatrix}1 & -in(n^2-\widetilde{k}_s^2)^{-1/2}\\ -in(n^2-\widetilde{k}_p^2)^{-1/2} & -1\end{bmatrix}.
 \end{equation}
 It is straightforward to see that
 \begin{equation}\label{eq:RN1v2}
    {\cal R}_{{\rm N},1} = \mathcal{H}_0 - \frac{1}2
   \begin{bmatrix}0 &   \widetilde{k}_s^2\I \\    \widetilde{k}_p^2 \I & 0\end{bmatrix}\HH\D_{-2} +\mathrm{OPS}(-4)
   \end{equation}
which is considerably simpler than $ {\cal R}_{\rm N}$  (cf. \eqref{eq:Rn2})  and satisfies
   \[
      {\cal R}_{{\rm N},1} = \frac{1}{\mu( \widetilde{k}_p^2 -2k_s^2+\widetilde{k}_s^2)}
     {\cal R}_{\rm N} +\mathrm{OPS}(-2).
   \]
  Next, we establish
 \begin{lemma}
 It holds
\begin{equation}\label{mainR}
\begin{aligned}
{\cal B}^{\rm comb}
  {\cal R}_{{\rm N},1} &\ =\  \begin{bmatrix}
  \left(2 \mu  \left(\widetilde{k}_p^2+k_s^2+\widetilde{k}_s^2\right)-3
  \omega^2\right) \I&   \left(3 \omega^2+2 \mu
   \left(\widetilde{k}_p^2-k_s^2+\widetilde{k}_s^2\right)\right)\HH \\
    \mu  \left(2 k_p^2-2 \widetilde{k}_p^2-3 k_s^2-2
   \widetilde{k}_s^2\right) \HH& \mu  \left(2 k_p^2+2 \widetilde{k}_p^2-3
   k_s^2+2 \widetilde{k}_s^2\right) \I
\end{bmatrix} + \mathrm{OPS}(-1).
\end{aligned}
 \end{equation}
 Moreover, the principal part of this operator given above is
 an \added[id=vD]{invertible Fourier multiplier operator of order zero.}
\end{lemma}
\begin{proof} Similarly as before,  using \eqref{eq:RN1v2}   we obtain
\[
         \begin{aligned}
& \left({\cal B}_{\rm DL} -2{\cal B}_{\rm SL} \begin{bmatrix}
                                        \W_{\widetilde{k}_p} & \\
                                            & \W_{\widetilde{k}_s}
                                       \end{bmatrix}
\right) {\cal R}_{{\rm N},1} \\
& \quad  =
\frac{1}{4}
\begin{bmatrix}
  \left(2 \mu  \left(\widetilde{k}_p^2+k_s^2+\widetilde{k}_s^2\right)-3
   \omega^2\right) \I&   \left(3 k_p^2(\lambda +2 \mu)+2 \mu
   \left(\widetilde{k}_p^2-k_s^2+\widetilde{k}_s^2\right)\right)\HH \\
    \mu  \left(2 k_p^2-2 \widetilde{k}_p^2-3 k_s^2-2
   \widetilde{k}_s^2\right) \HH& \mu  \left(2 k_p^2+2 \widetilde{k}_p^2-3
   k_s^2+2 \widetilde{k}_s^2\right) \I
\end{bmatrix}
+\mathrm{OPS}(-1)
         \end{aligned}
\]
The principal symbol is invertible since its {\em determinant} is different from zero (see Lemma below).
 \end{proof}

 \begin{lemma}
Let
\[
 A_2 =\frac14\begin{bmatrix}
   2 \mu  \left(\widetilde{k}_p^2+k_s^2+\widetilde{k}_s^2\right)-3
   \omega^2  &   \left(3 \omega^2+2 \mu
   \left(\widetilde{k}_p^2-k_s^2+\widetilde{k}_s^2\right)\right)i \\
    \mu  \left(2 k_p^2-2 \widetilde{k}_p^2-3 k_s^2-2
   \widetilde{k}_s^2\right)i& \mu  \left(2 k_p^2+2 \widetilde{k}_p^2-3
   k_s^2+2 \widetilde{k}_s^2\right)
\end{bmatrix}.
\]
Then
\[
 \det A_2 = -\frac{1}{4} \mu  \left(\widetilde{k}_p^2+\widetilde{k}_s^2\right) \left(k_p^2 (3
   \lambda +4 \mu )+k_s^2 \mu \right).
\]
 \end{lemma}
 \begin{proof}
 Clearly
  \[
   A_2\begin{bmatrix}
       i\\
       1
      \end{bmatrix}
      = \mu \left(\widetilde{k}_p^2+\widetilde{k}_s^2\right)\begin{bmatrix}
       i\\
       1
      \end{bmatrix}.
  \]
The relationship between the eigenvalues, determinant and trace of the matrix  implies then
\[
 \det A_2 = \mu \left(\widetilde{k}_p^2+\widetilde{k}_s^2\right)\left(\mathop{\rm Tr}(A_2)-\mu \left(\widetilde{k}_p^2+\widetilde{k}_s^2\right)\right).
 \]
 \end{proof}

 The arguments in the proof of Theorem~\ref{thm_inv} carry over in the Neumann case and we have
 \begin{theorem} The operators $
  {\cal B}^{\rm comb}
{\cal R}_{\rm N}$ and ${\cal B}^{\rm comb}
 {\cal R}_{{\rm N},1}$
 associated to CFIER formulations considered above are invertible pseudodifferential   operators of order zero.  Therefore the CFIER formulations are well posed.
 \end{theorem}
 \begin{proof} Indeed, the uniqueness of the solution of the elastic scattering equation in $\Omega^+$ with Neumann boundary conditions on $\Gamma$, together with the coercivity of the operators $\W_{\widetilde{k}_p},\ \W_{\widetilde{k}_s}$ and the invertibility of the operators ${\cal R}_{\rm N}$ and $ {\cal R}_{{\rm N},1}$ are used in the same manner as in the proof of Theorem~\ref{thm_inv} to conclude the injectivity of the Neumann CFIER operators. Given that we have already established the Fredholmness of these operators, the proof is complete.
 \end{proof}

 Finally, it is also possible to construct regularizing operators that involve Helmholtz BIOs only. Indeed, it is straightforward to see that the operator
 \begin{equation}\label{eq:RN2}
 {\cal R}_{{\rm N},2}:=
 \begin{bmatrix}
 \I & 2\partial_{\ttt} \V_{\widetilde{k}_s}\\
 2\partial_{\ttt} \V_{\widetilde{k}_p} & -\I
 \end{bmatrix}
 \end{equation}
 has the property
 \[
 {\cal R}_{{\rm N},2}- {\cal R}_{{\rm N},1}\in \mathrm{OPS}(-4).
 \]
 Therefore, the CFIER formulations based on the newly defined operator ${\cal R}_{{\rm N},2}$ are Fredholm. Their well posedness is a consequence of the following
 \begin{lemma} The operator ${\cal R}_{{\rm N},2}$ is injective.
 \end{lemma}
 \begin{proof}
 Assume $(f, g)\in \mathop{\rm Ker}(\mathcal{R}_{N,2})$, that is
\begin{eqnarray*}
f+2\partial_{\ttt} \V_{\widetilde{k}_s} [g]&=&0\\
2\partial_{\ttt} \V_{\widetilde{k}_p}[f] -g&=&0.
\end{eqnarray*}
First, we remark that by integrating each of the equations above on $\Gamma$ we obtain $\int_\Gamma f =\int_\Gamma g =0$. Therefore there exist two functional densities $F$ and $G$ on $\Gamma$ such that
\[
\partial_{\ttt} F = f\qquad \partial_{\ttt} G=g.
\]
We have then, after integration by parts,
\[
-2\int_\Gamma \overline{f}\V_{\widetilde{k}_p}[f] -
2\int_\Gamma \overline{g}\V_{\widetilde{k}_s}[g] -\left(
\int_\Gamma  F\overline{g}+\int_\Gamma  g\overline{F}\right)   =0.
\]

Taking the imaginary part of the last equation we obtain
\[
\Im   \int_\Gamma \overline{f}\V_{\widetilde{k}_p}[f]  + \Im \int_\Gamma \overline{g}\V_{\widetilde{k}_s}[g]=0
\]
but both terms are positive for $f,g\ne 0$. Hence we conclude $f=0$ and $g=0$, which completes the proof.
 \end{proof}

 \begin{theorem} The operator $
  {\cal B}^{\rm comb}
{\cal R}_{{\rm N},2}$
 is an invertible pseudodifferential operators of order zero. Therefore the CFIER \added[id=catB]{formulation} is well posed.
 \end{theorem}

\section{Numerical results}\label{NR}

In this brief section we check some of the robust formulations proposed in this paper.
It is rather straightforward to extend the Nystr\"om discretization based on global trigonometric interpolation and Kussmaul-Martensen singularity splittings for the four Helmholtz BIOs~\cite{KressH} to Nystr\"om discretizations of the BIOs introduced in Sections~\ref{dirP} and~\ref{neuP}. These extensions were described already in our previous contribution~\cite{dominguez2022nystrom}. In addition, the Fourier multipliers required by the various regularizing operators considered in this paper are easy to discretize using global trigonometric interpolation. We present in this section numerical results concerning the iterative behavior of solvers based on the Nystr\"om discretization of CFIE and CFIER  Helmholtz decomposition formulations using GMRES~\cite{SaadSchultz} iterative solvers. Specifically, we report numbers of \added[id=catR]{unrestarted} GMRES iterations required by various formulations to reach GMRES residuals of $10^{-5}$ for discretizations sizes that give rise to results accurate to at least four digits in the far field (which was estimated using reference solutions produced with  the high-order Nystr\"om solvers based on Navier Green functions~\cite{dominguez2021boundary}).

In all the numerical experiments in this section we assumed  plane wave incident fields of the form
\begin{equation}\label{eq:plwave}
  {\bf u}^{\rm inc}(\x)=\frac{1}{\mu}e^{ik_s\x\cdot{\bm d}}({\bm d}\cdot\mathrm{Q}{\bm p}){\bm d}+\frac{1}{\lambda+2\mu}e^{ik_p\x\cdot{\bm d}}({\bm d}\cdot{\bm p}){\bm d}
\end{equation}
where the direction ${\bm d}$ has unit length $|{\bm d}|=1$.  Specifically, we considered plane waves of direction ${\bm d}=\begin{bmatrix}0 & -1\end{bmatrix}^\top$ and ${\bm p}=\begin{bmatrix}1 &0\end{bmatrix}^\top$ in all of our numerical experiments,  and thus the incident wave is a shear S-wave. We observed that other choices of the direction ${\bm d}$ and of the vector ${\bm p}$ (including cases when ${\bm p}=\pm{\bm d}$---the incident plane is then a pressure wave or P-wave) lead to virtually identical results.  In all the numerical experiments we considered Lam\'e constants $\lambda=2$ and $\mu=1$ so that $k_s=2k_p$.

We present numerical results for three smooth scatterer shapes. Namely, a unit circle, the {\em kite}  parametrized by
\[
{\bf x}_2(t)=\frac{1}{L_1}(\cos{t}+0.65\cos{2t}-0.65, 1.5 \sin{t}),\
\]
and a smooth cavity whose parametrization is given by
 \[
{\bf x}_3(t) = \frac{1}{L_2}( 12\cos t+24 \cos 2 t,28 \sin t+17 \sin 2 t+18 \sin 3
   t-2 \sin 4 t)                                                                                                                                                                                                                                                       \]

The constants $L_1\approx 1.17145$ and $L_2\approx 56.2295$ are taken so that the length of the three curves is $2\pi$.
 In order to be consistent with the derivation of regularizing operators in Sections~\ref{dirP} and~\ref{neuP}, we used in our numerical experiments arc length parametrizations of the three shapes. For the CFIE formulations based on the representations~\eqref{eq:comb_field} we selected the coupling parameter $k=k_p$. for both the Dirichlet and Neumann case; we have found in practice that other choices of the coupling parameter such as $k=k_s$, or even the classical representations in equation~\eqref{remarkCFIED} lead to similar iterative behaviors for CFIE formulations. The CFIE formulations are of the first kind for both Dirichlet and Neumann cases since the operators ${\cal A}_{\rm CFIE}\in \mathrm{OPS}(1)$ and ${\cal B}_{\rm CFIE}\in \mathrm{OPS}(2)$, and thus the numbers of GMRES iterations required by CFIE formulations increase with the discretization size. The increase of the numbers of GMRES iterations with respect to discretization size is more dramatic for Neumann boundary conditions (i.e. for CFIE based on the operator ${\cal B}_{\rm CFIE}$---see Figure~\ref{fig:mysecond2} where we can see almost a linear growth with respect to the discretization size) but quite mild for Dirichlet boundary conditions (i.e. for CFIE based on the operator ${\cal A}_{\rm CFIE}$). In the Dirichlet case, and in the light of the result in Theorem~\ref{CFIED},  an obvious preconditioned CFIE formulation takes advantage of the fact that ${\cal A}_{\rm CFIE}^2\in \mathrm{OPS}(0)$. Unfortunately, this preconditioning strategies leads only to modest gains in numbers of iterations (i.e. about 20\%) despite the resulting formulation being of the second kind cf. Theorem~\ref{CFIED}.

 The CFIER formulations, on the other hand, are second kind integral equations for both types of boundary conditions. We present in Figure~\ref{fig:mysecond} and~\ref{fig:mysecond2} high frequency results based on CFIER formulations with regularizing operators ${\cal R}_{\rm D}$ defined in equations~\eqref{eq:sqR} in the Dirichlet case and respectively operators ${\cal R}_{\rm N}$ defined in equations~\eqref{eq:Rn2} in the Neumann case with the choice of complex wave-numbers
 \[
 \widetilde{k}_{p,s}=k_{p,s}+0.4\ i\ K^{2/3} k_{p,s}^{1/3},\qquad K=\max_\Gamma|\kappa|.
 \]

   \begin{figure}
\begin{center}
\resizebox{0.1825\textwidth}{!}{
\begin{tikzpicture}
[declare function = {f1(\x)=1; }]
\begin{axis}[
  xmin=-1.5 ,xmax=1.5 ,
  ymin=-1.5,ymax=1.5,
  grid=both,  axis equal image,
xtick=    {-2, -1,0, 1,2 },
ytick=    {-2, -1,0, 1,2 },
]
  \addplot[domain=0:360,samples=200,variable=\x,very thick, blue!70!black,data cs=polar](x,{f1(x)});%
\end{axis}
\end{tikzpicture}
}
\quad
 \resizebox{0.1825\textwidth}{!}{
\begin{tikzpicture}
\begin{axis}[
  xmin=-2 ,xmax=1.5 ,
  ymin=-1.75,ymax=1.75,
  grid=both,  axis equal image,
xtick=    {-2, -1,0, 1,2 },
ytick=    {-2, -1,0, 1,2 },
]
  \addplot[domain=0:360,samples=100,variable=\t,very thick, blue!70!black](%
    {1/1.17145*(cos( t)+0.65*cos(2* t)-0.65)},
    {1/1.17145* 1.5* sin(t)}%
  );
\end{axis}
\end{tikzpicture}
}
 \resizebox{0.1885\textwidth}{!}{
\begin{tikzpicture}
\begin{axis}[
  xmin=-0.75  ,xmax=0.75 ,
  ymin=-0.75  ,ymax=0.75,
  grid=both,  axis equal image,
xtick=    {  -.5,0,0.5 },
ytick=    {  -.5, 0,.5  },
]
  \addplot[domain=0:360,samples=100,variable=\t,very thick, blue!70!black](%
    {1/1.48396*(cos( t)+2*cos(2* t))/4},
    {1/1.48396* (sin( t)+ sin(2*t)+1/2* sin(3*t))/2-(-4* sin(t)+7*sin(2*t)-6*sin(3*t)+2*sin(4*t))/48}%
  );
\end{axis}
\end{tikzpicture}
}
\resizebox{0.187\textwidth}{!}{
\begin{tikzpicture}[
    /pgf/declare function={
        L = 6.2831/8;
    }
]
\begin{axis}[
  xmin=-1 ,xmax=1 ,
  ymin=-1,ymax=1,
  grid=both,  axis equal image,
xtick=    {-0.667, 0,0.667},
ytick=    {-0.667, 0,0.667 },
xticklabels=    {$-2/3$, 0,$2/3$},
yticklabels=    {$-2/3$, 0,$2/3$},
]
 \addplot[
    very thick, blue!70!black
    ]
    coordinates {
    (-L-0.0075,-L)(L,-L)(L,L)(-L-0.0075,L)(-L,-L)
    };
\end{axis}
\end{tikzpicture}
}
\resizebox{0.1827\textwidth}{!}{
\begin{tikzpicture}[
    /pgf/declare function={
        L = 6.2831/8;
    }
]
\begin{axis}[
  xmin=-1 ,xmax=1 ,
  ymin=-1,ymax=1,
  grid=both,  axis equal image,
xtick=    {-0.667, 0,0.667},
ytick=    {-0.667, 0,0.667 },
xticklabels=    {$-2/3$, 0,$2/3$},
yticklabels=    {$-2/3$, 0,$2/3$},
]
 \addplot[
    very thick, blue!70!black
    ]
    coordinates {
    (-L-0.0075,-L)(L,-L)(L,0)(0,0)(0,L)(-L,L)(-L,-L)
    };
\end{axis}
\end{tikzpicture}
}
\end{center}
\caption{\label{fig:geom} Geometries for the experiments considered in this section. Smooth curves: the unit circle, the kite   and   the cavity curve; Lipschitz domains: a square and the $L-$shaped  domain; Notice that all the curves are of length $2\pi$.}
\end{figure}
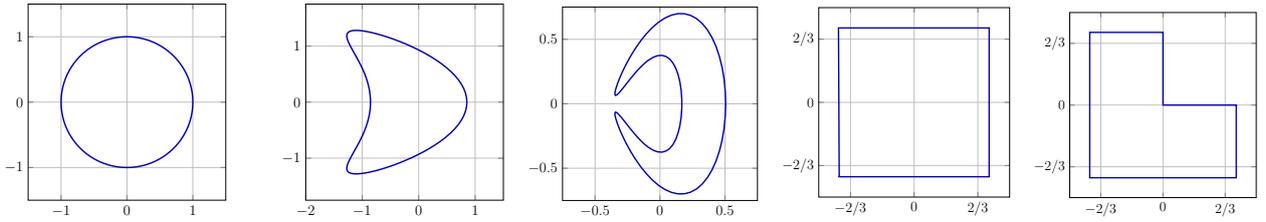

 The use of regularizing operators ${\cal R}_{{\rm D},1}$ defined in equations~\eqref{eq:qaltR} and ${\cal R}_{{\rm D},2}$ defined in equations~\eqref{eq:qaltR1} in the Dirichlet CFIER formulations leads to almost identical results and similarly for the use of regularizing operators $ {\cal R}_{{\rm N},1}$ defined in equations~\eqref{eq:RN1} and ${\cal R}_{{\rm N},2}$ defined in equations~\eqref{eq:RN2} in the Neumann CFIER formulations. Since the additional computational cost incurred by the implementation of the regularizing operators is negligible with respect to that of the CFIE operators,  we conclude that the use of CFIER formulation gives rise to significant gains in the high frequency regime.  We illustrate in Figure~\ref{fig:eig} the eigenvalue distribution corresponding to the CFIER formulations in the case of the kite geometry for the frequency $\omega=40$ and both Dirichlet and Neumann boundary conditions. We observe a very strong clustering of the eigenvalues in the Dirichlet case,  justifying the very small numbers of GMRES iterations needed for convergence (cf. Figure~\ref{fig:mysecond}). In the Neumann case the spectrum of the CFIER operator is more widespread, yet the clustering of eigenvalues around the value $1$ is still observed.

\begin{figure}
 \centering
 \includegraphics[width=0.34\textwidth]{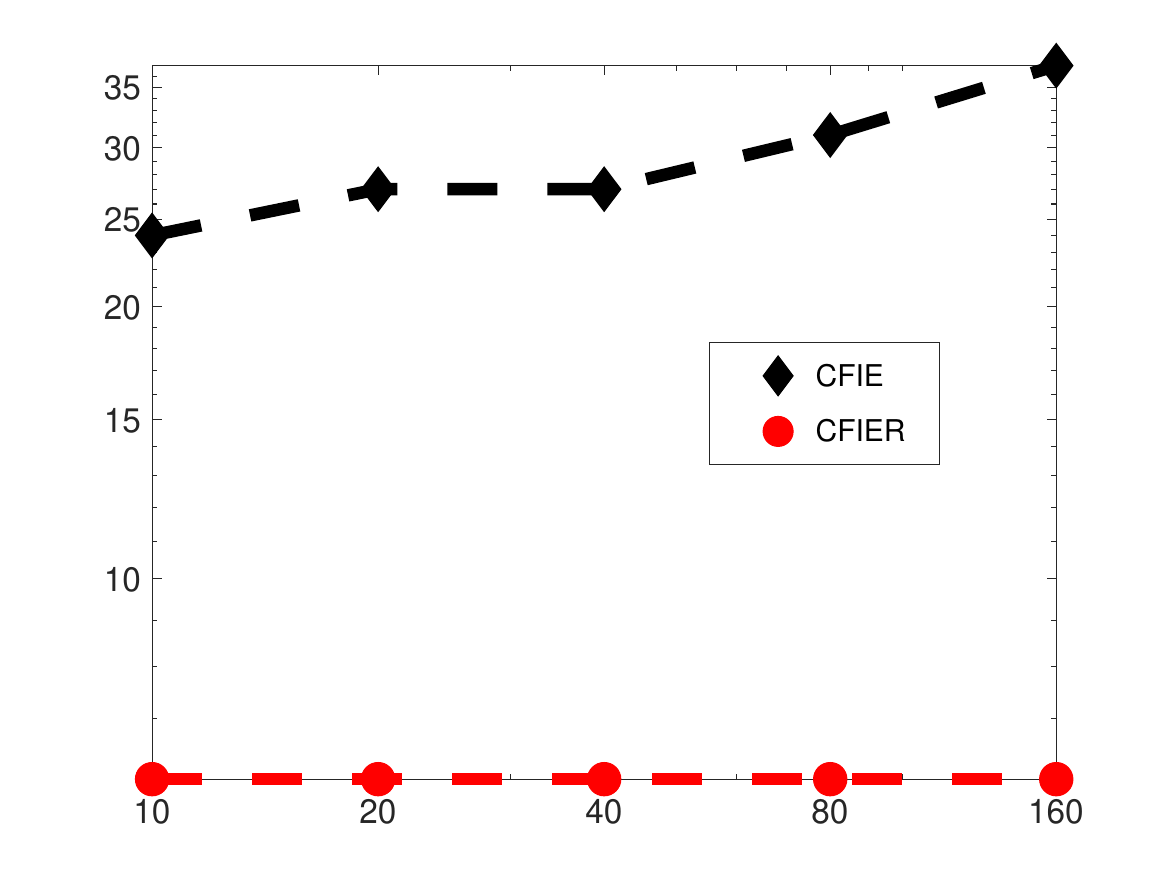}\includegraphics[width=0.34\textwidth]{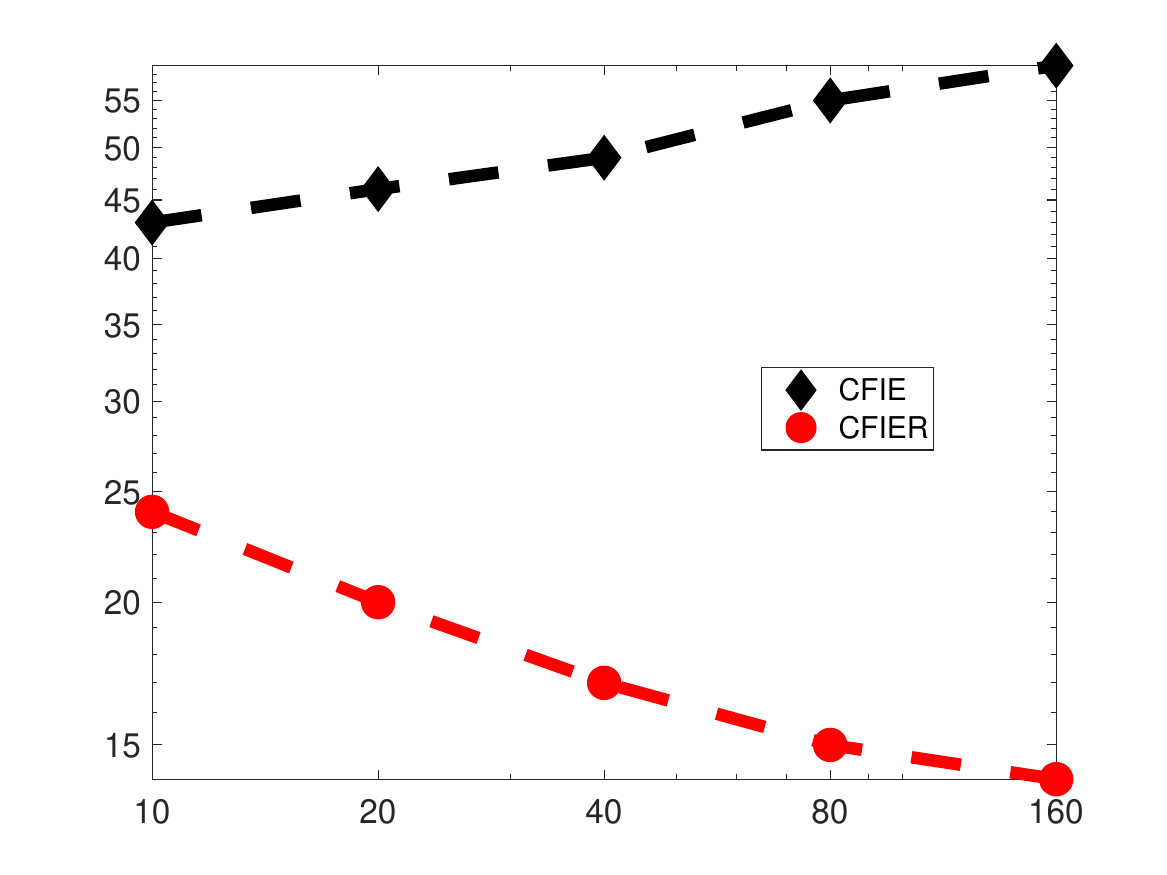}\includegraphics[width=0.34\textwidth]{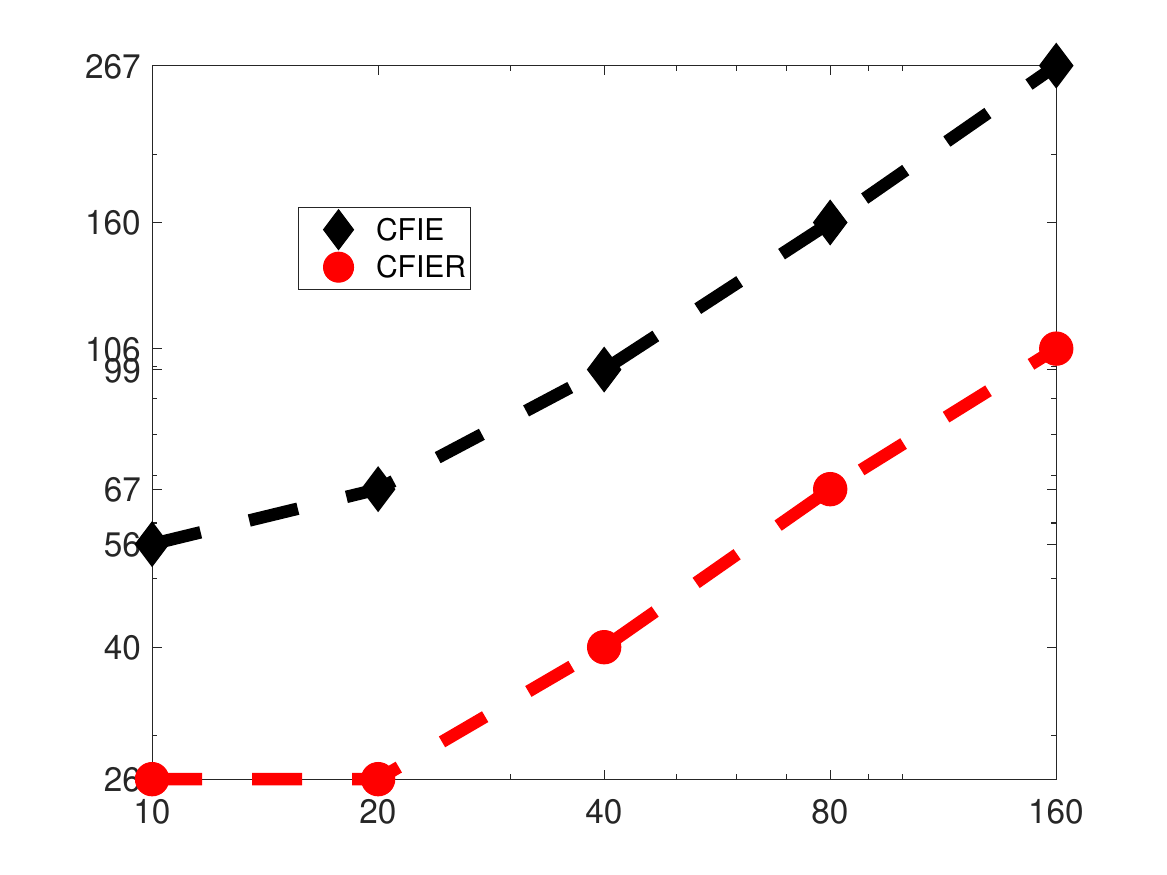}
\caption{Numbers of GMRES iterations required to reach residuals of $10^{-5}$ for the CFIE and CFIER formulations for the circle (left), kite (middle) and the smooth cavity (right) in the case of Dirichlet boundary conditions and frequencies $\omega=10,20,40,80,160$ with Lam\'e parameters $\lambda=2$ and $\mu=1$ under plane wave incidence. We  used Nystr\"om discretizations corresponding to 8 points per the shorter wavelength. The numbers of iterations are independent of the direction and polarization of the plane wave.}
 \label{fig:mysecond}
\end{figure}

\begin{figure}
 \centering
 \includegraphics[width=0.33\textwidth]{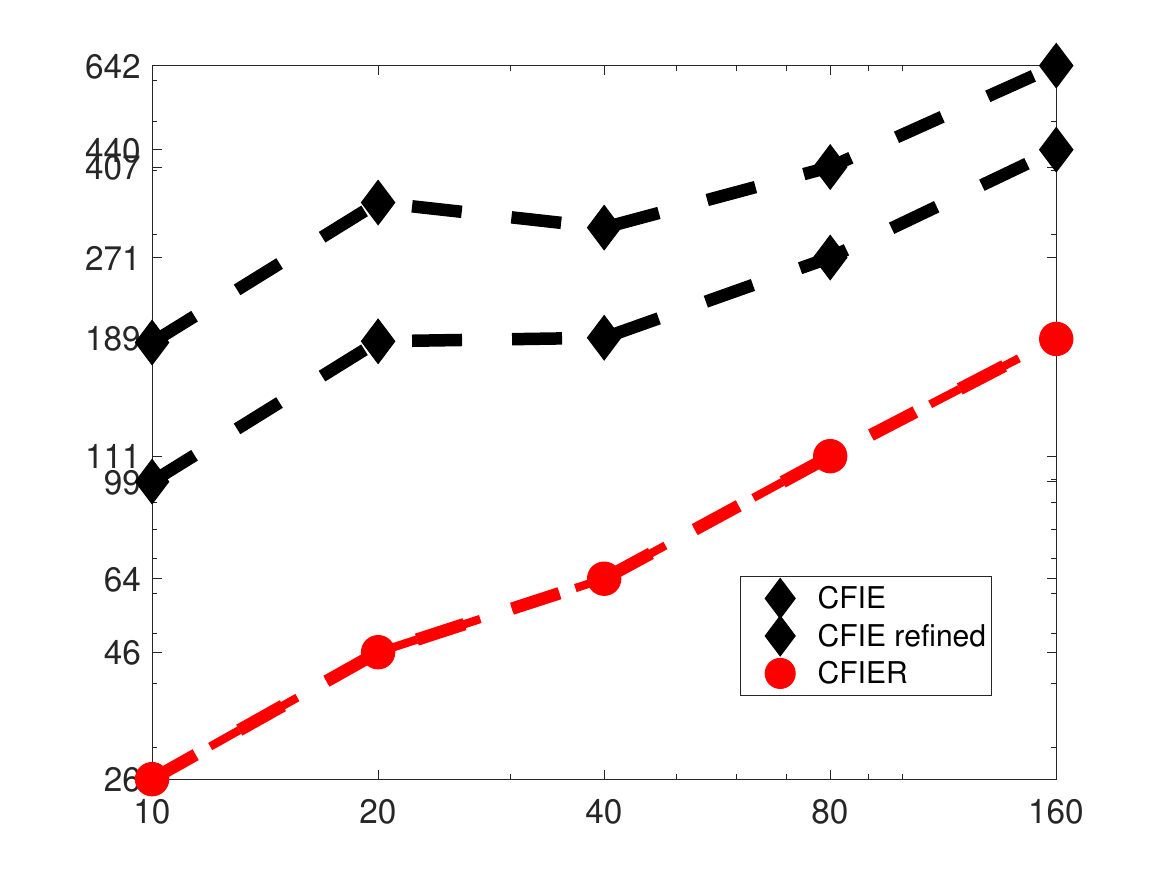}\includegraphics[width=0.33\textwidth]{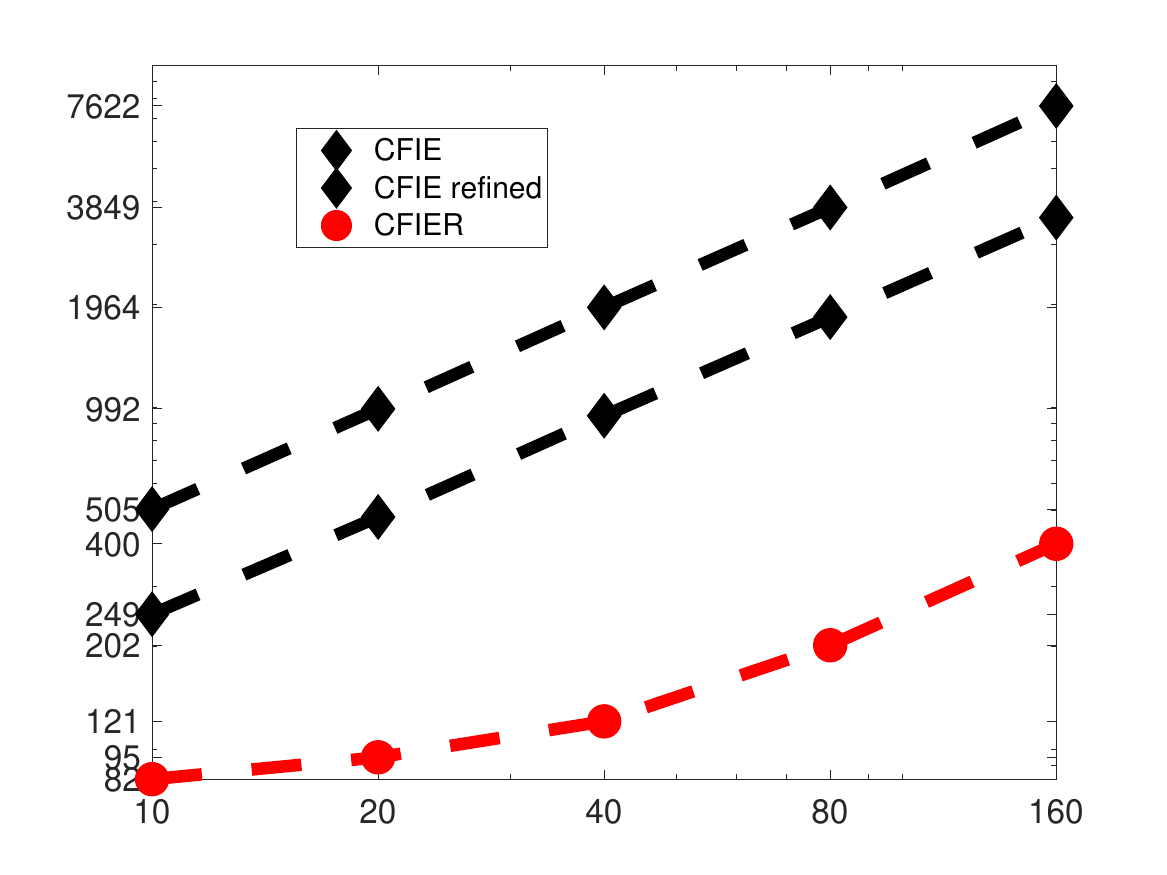}\includegraphics[width=0.33\textwidth]{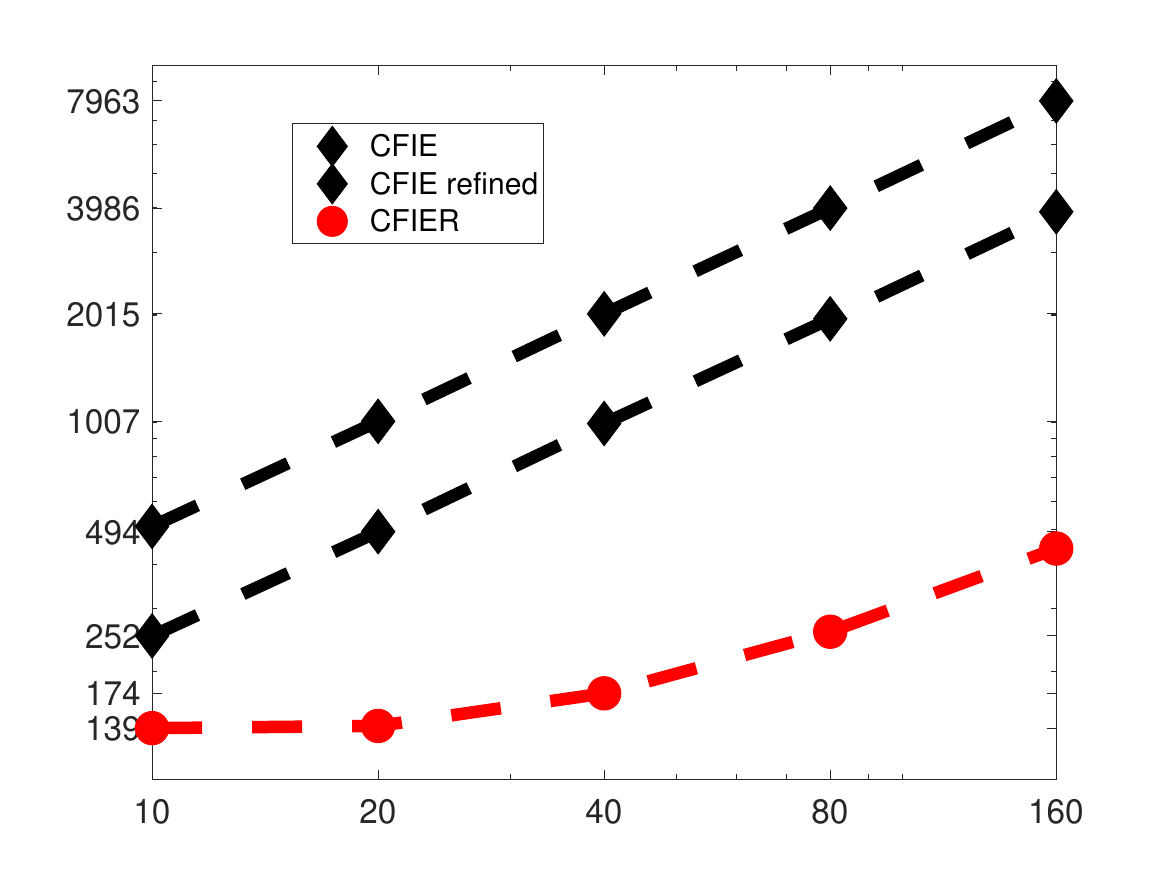}
\caption{Numbers of GMRES iterations required to reach residuals of $10^{-5}$ for the CFIE and CFIER formulations for the circle (left), kite (middle) and the smooth cavity (right) in the case of Neumann boundary conditions and frequencies $\omega=10,20,40,80,160$ with Lam\'e parameters $\lambda=2$ and $\mu=1$ under plane wave incidence. We  used Nystr\"om discretizations corresponding to 8 points per the shorter wavelength. In order to illustrate the effect of discretization size on the iterative behavior of the CFIE formulations, we report iteration counts "CFIE refined" corresponding to discretizations refined by a factor of two. The numbers of iterations are independent of the direction and polarization of the plane wave.}
 \label{fig:mysecond2}
\end{figure}

\begin{figure}
 \centering
 \includegraphics[width=0.48\textwidth]{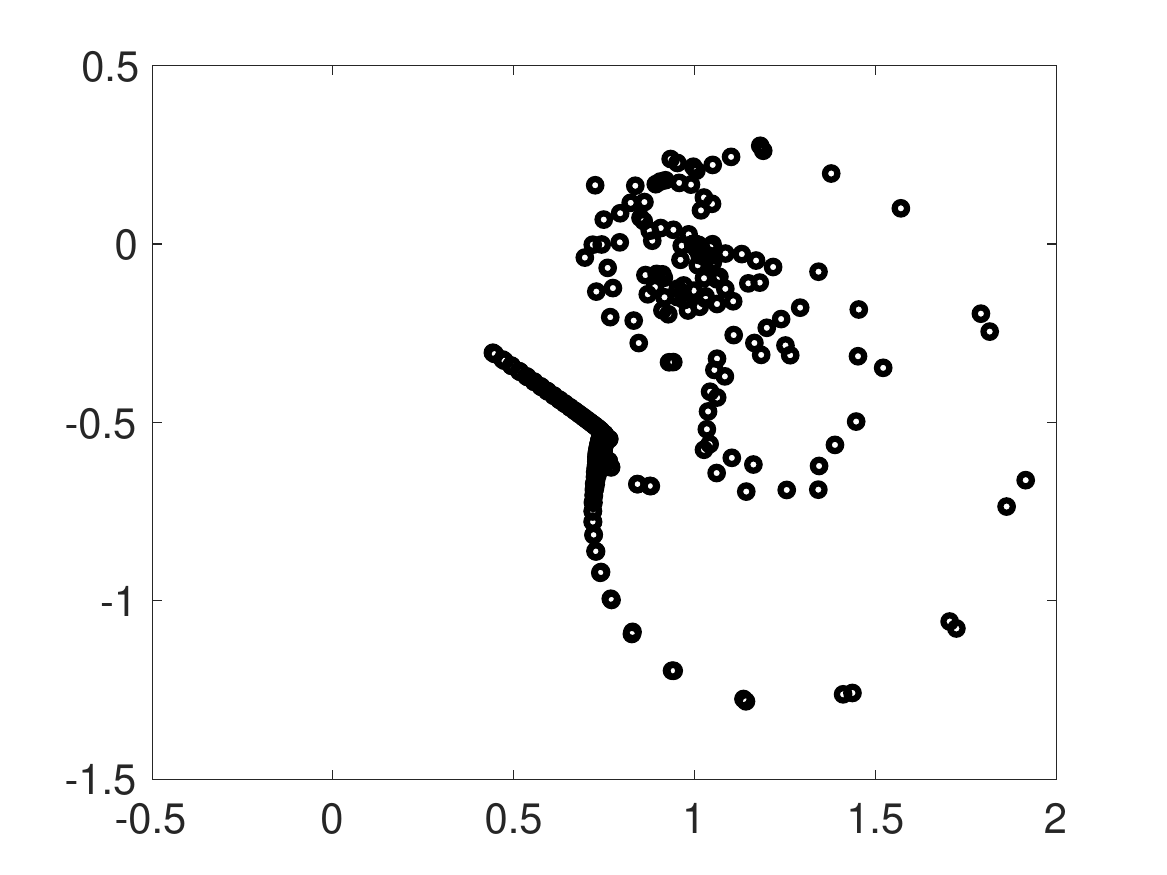}\includegraphics[width=0.48\textwidth]{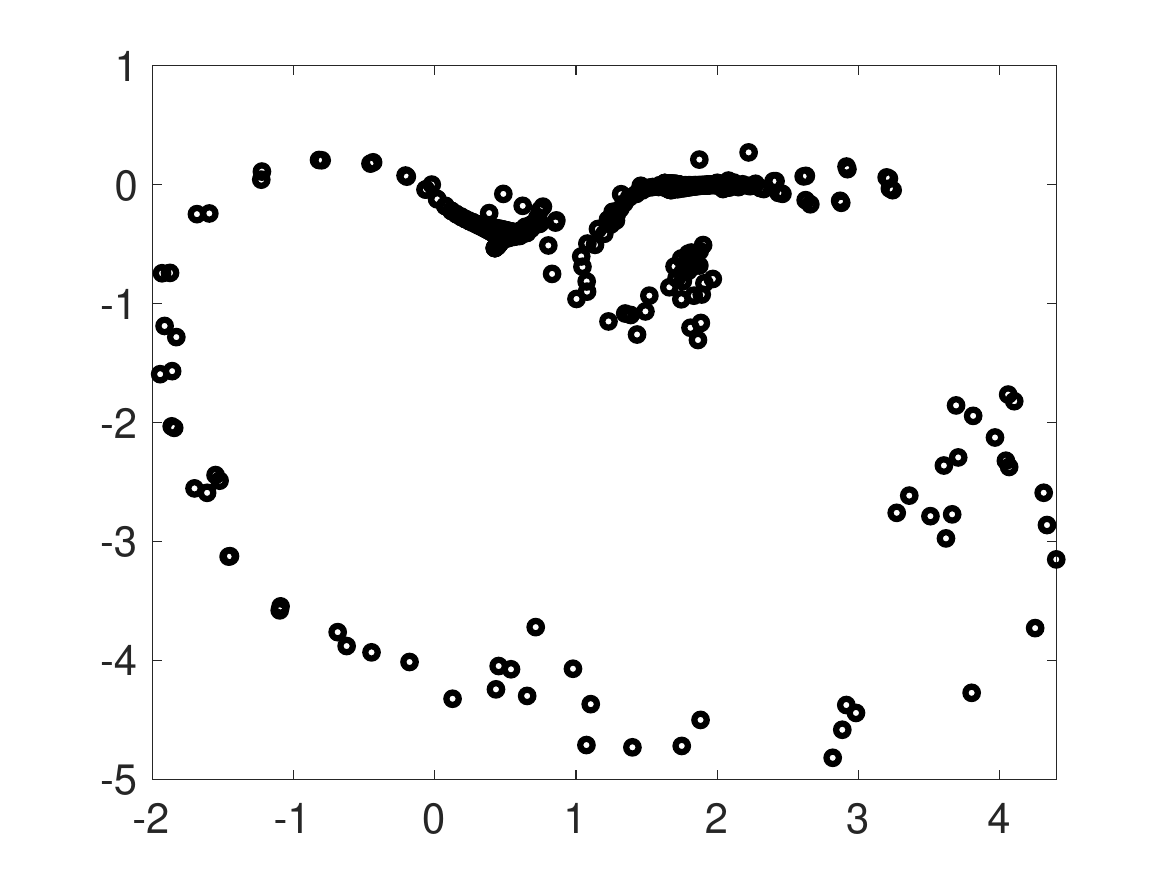}
 \caption{Eigenvalue distribution in the \added[id=catB]{complex plane} of the CFIER operators in the case of the kite geometry and $\omega=40$ for the Dirichlet (left) and Neumann (right) cases.}
 \label{fig:eig}
\end{figure}

\begin{table}[h]
\centering
{
\begin{tabular}{|c|c|c|c|c|c|c|}
 \cline{1-7}
Size & \multicolumn{2}{c|}{Navier CFIE} & \multicolumn{2}{c|}{Helmholtz CFIE}  & \multicolumn{2}{c|}{Helmholtz CFIER}\\
 \cline{1-7}
$n$& Dirichlet & Neumann & Dirichlet & Neumann & Dirichlet & Neumann\\
 \cline{1-7}
64 & 0.96 sec & 1.87  sec & 0.82 sec & 1.1 sec & 0.89 sec & 1.2 sec\\
128 & 2.58 sec & 6.02  sec & 2.50 sec & 4.6 sec & 2.69 sec & 5.2 sec\\
\hline
\end{tabular}}
\caption{Comparisons between the computational times required by our Nystr\"om solvers based on CFIE formulations based on Navier Green's function~\cite{dominguez2022nystrom} and the CFIE/CFIER Helmholtz decomposition formulations presented above for the kite geometry at the same level of discretizations. All the BIOs featured in these formulations were implemented using the Kussmaul-Martensen logarithmic splitting strategy and global trigonometric interpolation in the space of trigonometric polynomials of degree $n$. The computational times reported were produced by unoptimized MATLAB implementation on a single core of a 8 GB 2.7 GHz Dual-Core Intel Core i5 MacBook Pro.}
\label{tab:my_label}
\end{table}

\added[id=vD]{Hessians}{As we previously mentioned, the main appeal of using the BIE Helmholtz decomposition approach for the solution of the Navier scattering problems is the possibility to reuse discretizations of Helmholtz BIOs, thus avoiding dealing with BIOs associated with the Navier Green's function, whose implementation is more arduous. We illustrate in Table~\ref{tab:my_label} the commensurate computational times required in our unoptimized MATLAB implementation by solvers based on the two different approaches, that is the classical combined field based on Navier potentials and the regularized combined field approach based on the Helmholtz decomposition. These findings are not entirely surprising, given that one evaluation of the Navier Green's function requires an evaluation of a Helmholtz Green's function and two \added[id=vD]{Hessians} of Helmholtz Green's function for the pressure and shear wave-number. The kernels of the Navier single/double layer integral operators and especially those of the Navier hypersingular operators are significantly more complicated, and their numerical evaluations require the evaluation of the kernels of the four classical Helmholtz BIOs for both wave-numbers. }

\paragraph{Lipschitz case.}
\added[id=vD]{Finally, we illustrate in Figure~\ref{fig:mysecond3} the iterative behavior of the CFIE and CFIER formulations in the case of   a square and  an L-shaped scatterer of lengths $2\pi$
with arc length parametrizations, equipped with sigmoidal graded meshes that accumulate points polynomially (e.g. polynomials of degree three were used in our numerical experiments) towards the corner points}. We remark that the various well posedness proofs of the formulations considered in this paper relied heavily on the smoothness of the curve $\Gamma$. Indeed, a key ingredient in the analysis was the increased regularity of the double layer operators, that is $\K_k^\top\in \mathrm{OPS}(-3)$, which in the case of Lipschitz curves is only $\K_k^\top\in \mathrm{OPS}(0)$. As a consequence, the CFIER formulations are no longer of the second kind. Nevertheless, the CFIER formulations still outperform the CFIE formulations. For instance, we did not observe convergence when GMRES solvers were applied to CFIE formulations in the Neumann case.

\begin{figure}
 \centering
 \includegraphics[width=0.45\textwidth]{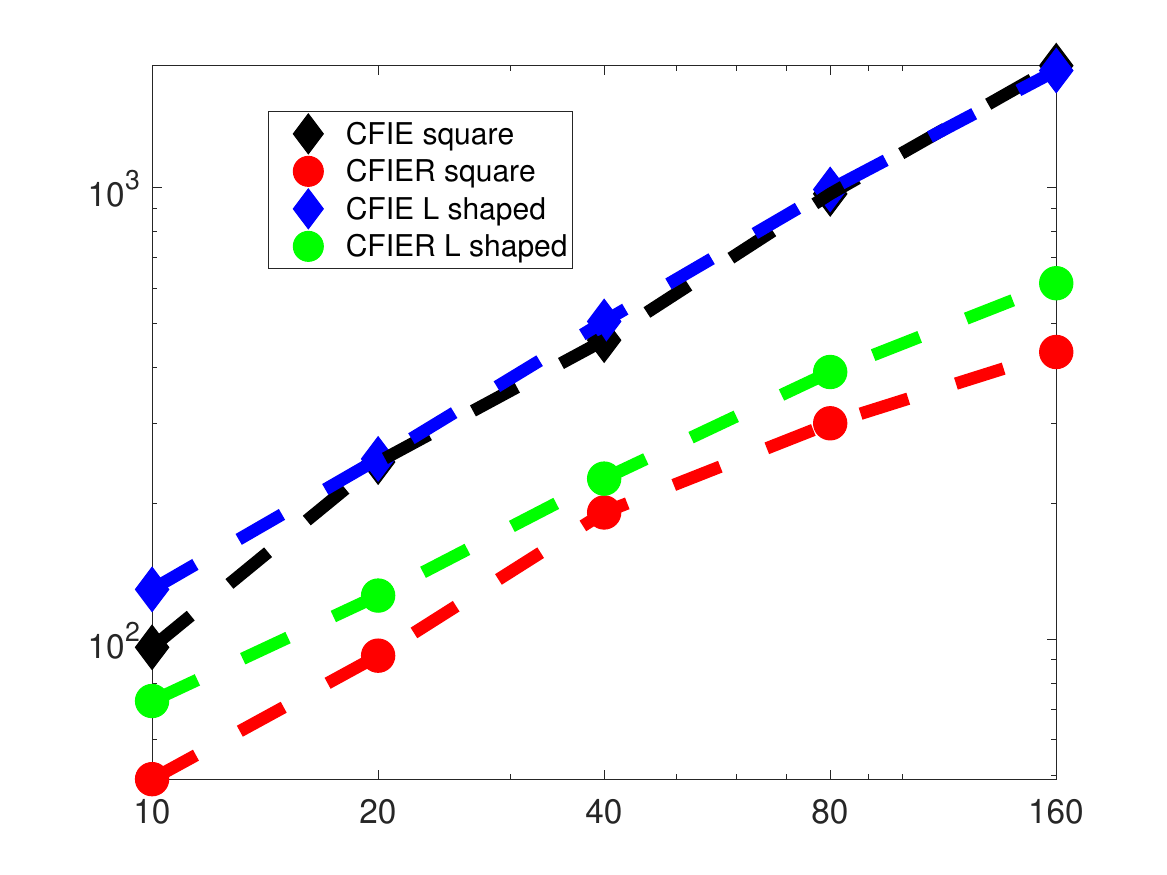}\includegraphics[width=0.45\textwidth]{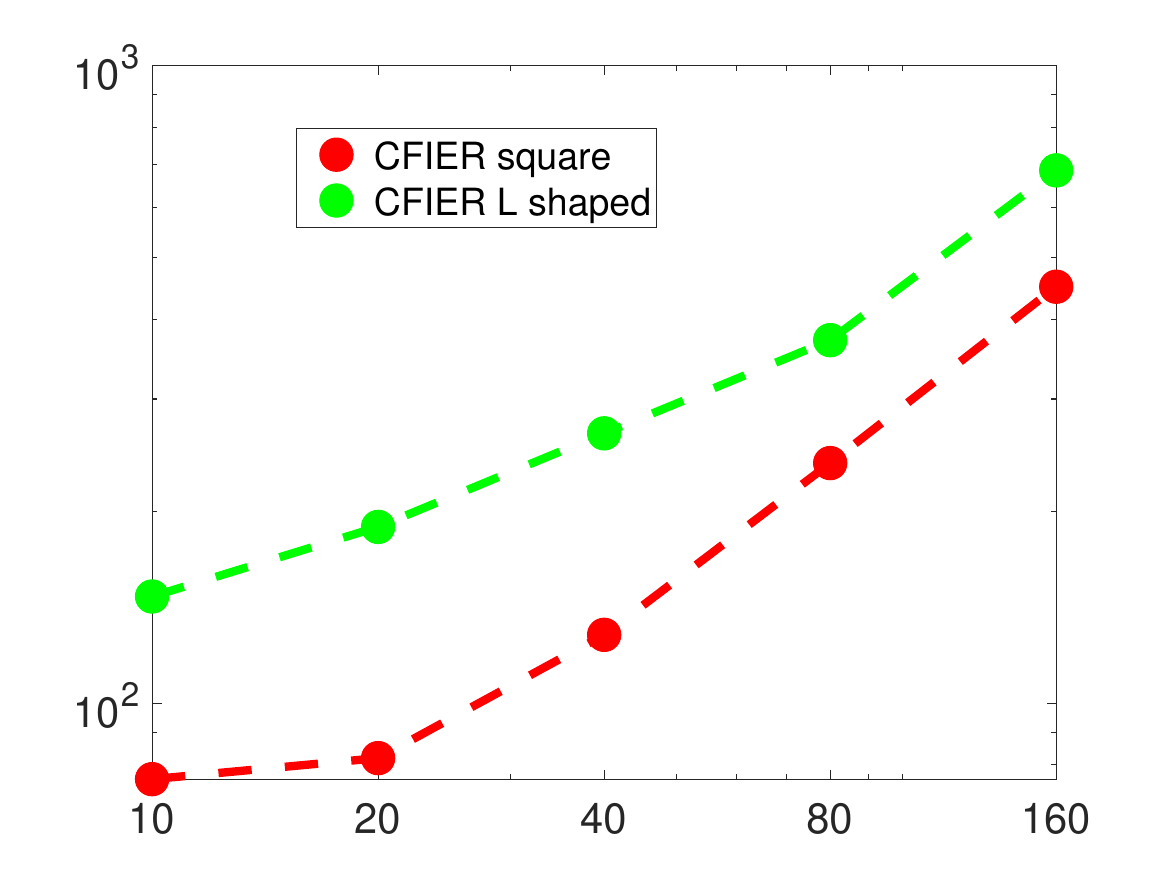}
 \caption{Numbers of GMRES iterations required to reach residuals of $10^{-4}$ for the CFIE and CFIER formulations for the square and the L-shaped scatterers in the high frequency regime for the Dirichlet (left) and Neumann (right) boundary conditions and the same material parameters as in the previous cases, In the case of Neumann boundary conditions, the solvers based on CFIE formulations did not converge. }
 \label{fig:mysecond3}
\end{figure}

\section{Conclusions}

We introduced and analyzed CFIER formulations for the solution of two dimensional elastic scattering problems via Helmholtz decompositions. Despite featuring non standard BIOs, we showed that these CFIER formulations are well posed in the case of smooth scatterers. The CFIER formulations, being of the second kind, possess superior spectral properties to the classical CFIE formulations for both Dirichlet and Neumann boundary conditions. The extension of our approach to the case of homogeneous penetrable scatterers is straightforward and is currently being pursued. The extension of the CFIER methodology to three dimensional elastic scattering problems via Helmholtz decompositions, on the other hand, is more challenging as it requires incorporation of both Helmholtz and Maxwell BIOs.

\section*{Acknowledgments}
 Catalin Turc gratefully acknowledges support from NSF through contract DMS-1908602. V\'{\i}ctor Dom\'{\i}nguez is partially supported by project ``Adquisici\'on de conocimiento y miner\'{\i}a de datos,
funciones especiales y m\'etodos num\'ericos avanzados'' from Universidad P\'ublica de Navarra  and ``T\'ecnicas innovadoras para la resoluci\'on de problemas evolutivos'', ref. PID2022-136441NB-I00 from Ministerio de Ciencia e Innovaci\'on, Gobierno de Espa\~na, Spain.

\added[id=vD]{The authors wish to thank the reviewers for their careful reading of the first version of this manuscript and for their valuable suggestions which undoubtedly helped us to improve this paper.}
\bibliography{elastodynamics_helmholtz_nov}

\appendix

\section{Helmholtz BIOs and pseudodifferential operator calculus}\label{ap:BIO}

For a given wave-number $k$ and a functional density $\varphi$ on the boundary $\Gamma$ denote in this section the Helmholtz single and double layer potentials in the form
\[
\SL_{k,\Gamma}[\varphi](\x):=\int_\Gamma H_0^{(1)}(k|\x-\y|)\varphi(\y) {{\rm d}\y},\quad {\DL_{k,\Gamma}}[\varphi](\x):=\int_\Gamma \frac{\partial H_0^{(1)}(k|\x-\y|)}{\partial \nnn(\y)}\varphi(\y) {{\rm d}\y},\ \x\in\mathbb{R}^2\setminus\Gamma.
\]
Here $\tfrac{i}4H_0^{(1)}(k|\cdot|)$, where $H_0^{(1)}$ is the Hankel function of first kind and order zero, is the fundamental solution of the Helmholtz equation.

Any exterior, radiating, solution of the Helmholtz equation can be written
\begin{equation}
\label{eq:repr:form}
 u(\x)=\DL_{k,\Gamma}[\gamma_\Gamma^+u](\x) - \SL_{k,\Gamma}[\partial_{\nnn}^+u](\x),\quad \x \in\Omega^+.
\end{equation}

The associated layer operators (read by rows in the matrix operator: double layer, single layer, hypersingular and adjoing double layer) are  defined as
\[
\begin{aligned}
 \begin{bmatrix}\gamma^+
\\
\partial_{\nnn}^+
 \end{bmatrix}
 \begin{bmatrix}\mathrm{DL}_{k,\Gamma} &
 - \mathrm{SL}_{k,\Gamma}
 \end{bmatrix}
 =\frac12 {\cal I} + {\cal C}_{k,\Gamma}
  \end{aligned}
  \]
  where
  \[
{\cal I} =\begin{bmatrix}
           \mathrm{I}& \\
                     &\mathrm{I}
          \end{bmatrix},\quad
{\cal C}_{k,\Gamma}:=
\begin{bmatrix}
   \K_{k,\Gamma}&-\V_{k,\Gamma}\\
   \W_{k,\Gamma}&-\K^\top_{k,\Gamma}
  \end{bmatrix} : H^s(\Gamma)\times H^{s-1}(\Gamma)\to H^s(\Gamma)\times H^{s-1}(\Gamma).
\]
The matrix operator ${\cal C}_{k,\Gamma}$ is then continuous  for any $s$ if $\Gamma$ is smooth.

Calder\'on identities
can be derived from the fact that (see for instance \cite[Ch.2 ]{Saranen}, \cite[Ch. 1]{hsiao2008boundary} or
\cite[Chapter 6-7]{mclean:2000}) that
 \[
 \left(\frac12\mathcal{I}+{\cal C}_{k,\Gamma}\right)^2=\frac12\mathcal{I}+{\cal C}_{k,\Gamma},
 \]
which are equivalent to
\begin{equation}\label{eq:calderonIdentitiesv2}
\begin{aligned}
\V_{k,\Gamma}   \W_{k,\Gamma} &= -\frac{1}4 \I+\K_{k,\Gamma}^2,&\quad
\W_{k,\Gamma}   \V_{k,\Gamma} &= -\frac{1}4 \I+(\K^\top_{k,\Gamma})^2\\
\V_{k,\Gamma} \K_{k,\Gamma}^\top & =  \K_{k,\Gamma}\V_{k,\Gamma},
\quad  &\W_{k,\Gamma} \K_{k,\Gamma} & =  \K^\top_{k,\Gamma} \W_{k,\Gamma}.
\end{aligned}
\end{equation}

It is convenient to present the analysis of robust BIE formulations of Helmholtz decompositions of Navier equations in the framework of periodic pseudodifferential operators. Consider then a regular positive oriented arc-length parameterization of $\Gamma$,  ${\bf x}:[0,L]\to  \Gamma $ where $L$ is the length of the curve. For any function, or distribution in the general case, $\varphi_\Gamma:\Gamma\to \mathbb{C}$ we will denote  by
\[
 \varphi(\tau) = \varphi_\Gamma({\bf x}(\tau))
\]
its parameterized counterpart.
We will extend this convention to the operators. For instance,
\[
(\V_k\varphi)(\tau) = \frac{i}4\int_0^L
H_0^{(1)}(k|{\bf x}(t)-{\bf x}(\tau)|)\varphi_\Gamma({\bf x}(\tau))\,
{\rm d}\tau.
\]
is the parameterized version of $\V_{k,\Gamma}$.

The unit tangent and normal parameterized vector to $\Gamma$ (at ${\bf x}(\tau)$) are then given by
\[
 \ttt(\tau) = {\bf x}'(\tau),\quad
 \nnn(\tau) = \mathrm{Q}{\bf x}'(\tau),\quad \mathrm{Q} =\begin{bmatrix}
                                                            & 1\\
                                                            -1&
                                                           \end{bmatrix}
\]
so that
\[
(\partial_{\ttt}\varphi_\Gamma)\circ{\bf x} = \varphi'=: \mathrm{D}\varphi.
\]
Besides, the (parameterized) signed quadrature can be then expressed as
\begin{equation}\label{eq:ttnn:appendix}
\kappa = \ttt\cdot \D\nnn = -\nnn\cdot \D\ttt.
\end{equation}

It is a well-established result, see for instance \cite[Ch. 8]{Kress}, that Sobolev spaces on $\Gamma$ $H^s(\Gamma)$ can be then identified with the $L-$periodic Sobolev spaces
\[
 H^s = \Big\{\varphi\in{\cal D}'(\mathbb{R}) \ :\  \varphi(\cdot+L) = \varphi, \quad \|\varphi\|_{s}<\infty \Big\}
\]
where, with
\[
 \quad \widehat{\varphi}(n) =\frac{1}{L}\int_{0}^{L} \varphi(\tau) e_{-n}(\tau)\,{\rm d}\tau,\qquad
 e_{n}(\tau)= \exp\Big( \frac{2\pi i}{L} n \tau \Big)
\]
the Fourier coefficients, the Sobolev norm is given by
\[
 \|\varphi\|_{s}^2 = |\widehat{\varphi}(0)|^2 + \sum_{n\ne 0} |n|^{2s}
 |\widehat{\varphi}(n)|^2.
\]
Clearly,
\[
 \varphi = \widehat{\varphi}(0)+\sum_{n\ne 0}\widehat{\varphi}(n)e_{n}
\]
with convergence in $H^s$. The set $\{H^s\}_{s\in\mathbb{R}}$ is a Hilbert scale, meaning that $H^s$ is actually a Hilbert space, that $H^{t}\subset  H^{s}$ for any $t> s$ with compact and dense injection and that
\[
 \bigcap_{s} H^s ={\cal D}:= \{\varphi\in{\cal C}^\infty(\mathbb{R})\ :\ \varphi= \varphi(\cdot+L)\},
 \quad
 \bigcup_{s} H^s ={\cal D}'.
\]
We will denote by $\mathrm{OPS}(m)$ the class of periodic pseudodifferential operators of order $m\in\mathbb{Z}$ on $\Gamma$. That is, $\mathrm{A}\in \mathrm{OPS}(m)$ if
$\mathrm{A}:H^{s}\to H^{s-m}$ is continuous for any $s$.   For convenience, we will write that
\[
 \mathrm{A} = \mathrm{B} +  \mathrm{OPS}(m-1),
\]
if $ \mathrm{A} - \mathrm{B} \in \mathrm{OPS}(m-1)$

 Trivially, $\mathrm{A}\in \mathrm{OPS}(m)$ implies that $\mathrm{A}\in \mathrm{OPS}(m+1)$ and as an operator in ${\mathrm{OPS}}(m+1)$ $\mathrm{A}$ is compact. We also set
\[
 {\mathrm{OPS}}(-\infty):=\bigcap_{m\in\mathbb{Z}} \mathrm{OPS}(m)
\]
the class of smoothing operators which in turn can be identified with  integral operators with $L-$periodic smooth kernel.


%
%
%
%

In connection to periodic pseudodifferential operators, Fourier multiplier operators will play a central role in what follows:
\begin{equation}\label{eq:FMult}
  \mathrm{A}\varphi =\sum_{m=-\infty}^\infty  \mathrm{A}(n) \widehat{\varphi}(n) e_n\qquad 
\end{equation}
where $\mathrm{A}(n)\in\mathbb{C}$ are referred to as the symbol of $\mathrm{A}$. Clearly, if $\mathrm{A}$ is a Fourier multiplier defined as in~\eqref{eq:FMult}, if there exists $c>0$ and $r$ such that $|\mathrm{A}(n)|\leq c|n|^r$ for all $n\in\mathbb{Z}$ (which for simplicity we shall denote in what follows as $ \mathrm{A}(n)  =\mathcal{O}(n^r)$) then $ \mathrm{A} \in \mathrm{OPS}(r)$.  Equivalently, if
\[
 \sum_{n\in\mathbb{Z}} |\mathrm{A}(n)|<\infty,
\]
then with the function
\[
 a(t)=\frac{1}{L}\sum_{n\in\mathbb{Z}} \mathrm{A}(n)e_n(t)\in L^\infty(\mathbb{R}), \quad\text{that is,\quad} \widehat{a}(n)=\frac{1}L \mathrm{A}(n),
\]
$\mathrm{A}$ is just a convolution operator:
\[
 \mathrm{A} f =\int_0^L a(\cdot-\tau)\varphi(\tau)\,{\rm d}\tau.
\]
Note that the tangential derivative becomes a Fourier multiplier:
\[
 \mathrm{D}\varphi = \sum_{n\ne 0} \left(\frac{2\pi i m}{L} \right) \widehat{\varphi}(m) e_m =\varphi'
\]
and that for any nonnegative integer $r$,
\[
  \mathrm{D}^r = \mathrm{D}_{r},\quad \mathrm{D}_r\varphi := \sum_{n\ne 0} \left(\frac{2\pi i m}{L} \right)^r \widehat{\varphi}(m) e_m.
\]
We will extend this definition to set $\mathrm{D}_r$ for negative integer values of $r$ too.

Three additional Fourier multiplier operators we will required in our analysis. First, the Bessel operator
\begin{eqnarray*}
 \Lambda\varphi &=& -\frac{1}{2\pi}\int_0^L \log\left(4e^{-1}\sin^2\left(\frac{\pi}L(\cdot-\tau)\right)\right)\varphi(\tau)\,{\rm d}\tau=\frac{L}{2\pi}\left[\widehat{\varphi}(0)+\sum_{n\ne 0}\frac{1}{|n|}\widehat{\varphi}(n)e_{n} \right],
\end{eqnarray*}
next, the Hilbert transform, or Hilbert singular operator,
\[
 \HH\varphi  :=  -\frac{1}{L}\mathrm{p.v.}\int_0^L \cot\left(\frac{\pi}L(\cdot-\tau)\right)\varphi(\tau)\,{\rm d}\tau+  i \mathrm{J}\varphi =  i  \bigg[ \widehat{\varphi}(0) + \sum_{n\ne 0}\mathop{\rm \rm sign}(n)\widehat{\varphi}(n)e_{n}\bigg],
\]
(p.v. stants for princial value integral, since the integrand is singular and does not exist in a classical sense)
with
\[
 \mathrm{J}\varphi := \widehat{\varphi}(0)= \frac{1}{L}\int_\Gamma \varphi_\Gamma.
\]
the mean operator. We notice then
\[
\begin{aligned}
 \D\Lambda & = \Lambda\D = i \HH+\mathrm{J} , & \quad \Lambda & =\HH\D_{-1}+\frac{L}{2\pi}\mathrm{J}, \\
 \D \D_{-1}& =\I-\mathrm{J} =-\HH^2-\mathrm{J}& \quad \Lambda^{-1} & =-\D\HH+\mathrm{J}=-\D_2\Lambda+\frac{2\pi}L\mathrm{J}.
 \end{aligned}
\]

By a direct analysis of the resulting kernel  we can easily check
\begin{equation}
  \label{eq:H3}
a  \HH - \HH a\in \mathrm{OPS}(-\infty)
\end{equation}
for any smooth function $a$. Similarly.
\begin{equation}
  \label{eq:H4}
a \D_r-\D_ra\in \mathrm{OPS}(r-1)
\end{equation}
which is trivial to show for positive $r$ and an easy consequence, from  negative values of $r$, of   the equality
\[
\D_{-1}(a\varphi)=
 a\D_{-1}\varphi -\D_{-1}(a'\D_{-1}\varphi)-(\D_{-1}a)\mathrm{J}\varphi.
\]
which implies, still for $r<0$
\begin{equation}\label{eq:Leibnitz:rule}
\D_{r}(a\varphi)=\sum_{m=-n}^r \binom{r}{-m}(\D_{r-m}a)(\D_{m}\varphi) +\mathrm{OPS}(-n-1).
\end{equation}

We point out that if $\{\K_0,\V_0,\W_0,\K^\top_0\}$ are the boundary layer operators associated to the Laplace equation ($k=0$),
\begin{eqnarray*}
 \V_0 &=& \frac{1}{2}\Lambda + \mathrm{OPS}(-\infty)= \frac{1}{2}\HH\D_{-1} + \mathrm{OPS}(-\infty), \\
 \W_0 &=& -\frac{1}2\mathrm{\Lambda}^{-1}_0+\mathrm{OPS}(-\infty) =\frac{1}{2}\D\HH  + \mathrm{OPS}(-\infty)\\
 \K_0 &=& \mathrm{OPS}(-\infty)\ = \K_0^\top.
\end{eqnarray*}
since the functions
\[
 \log\left(\frac{\left|\sin \left(\frac{\pi}L(t-\tau)\right)\right|}{|{\bf x}(t)-{\bf x}(\tau)|^2}\right),\quad
\partial_{\tau}\partial_t\log \left|\sin \left(\frac{\pi}L(t-\tau)\right)\right| -
  \frac{\nnn({\bf x}(t))\cdot\nnn({\bf x}(\tau))}{|{\bf x}(t)-{\bf x}(\tau)|^2}
\]
are smooth.

Our aim is to extend such expansions for the Calderon operators associated to Helmholtz equation. For such purposes, let us define for non-negative integer values of $r$
\[
 \alpha_{2r}(\tau)= -\frac{1}{2\pi}\left[2\sin^{2}\left(\frac{\pi}L\tau\right)\right]^r\log\left(4e^{-1}\sin^2\left(\frac{\pi}L\tau\right)\right)
\]
and denote the associated Fourier multiplier operator by $\Lambda_{2r}$. Note that  $\Lambda_0=\Lambda$ and that
\[
 \widehat{\alpha}_{2r}(n)=
 -\tfrac12\widehat{\alpha}_{2r-2}(n-1)
 +\widehat{\alpha}_{2r-2}(n)
 -\tfrac12\widehat{\alpha}_{2r-2}(n+1)
\]
which implies
\[
  \widehat{\alpha}_{2r}(n) = \frac{(-1)^r (2r)! L }{2^r (2\pi) }\frac{1}{|n|^{2r+1}} + {\cal O}(n^{-2r-3}) =
 \frac{(-1)^r (2r)! 2^r \pi^{2r}}{L^{2r}}  |\widehat{\alpha}_{0} (n)|^{2r+1}+ {\cal O}(n^{-2r-3})
\]
That is,  $\Lambda_{2r}\in \mathrm{OPS}(-2r-1)$ with
\[
 \Lambda_{2r}=  \frac{ (2r)! 2^r \pi^{2r} }{L^{2r}}(-1)^r\Lambda^{2r+1}+  \mathrm{OPS}(-2r-3).
\]
(Actually, $\Lambda_{2r}$ can be {\em expanded} in powers of $\Lambda^{2\ell+1}$ with $\ell\ge r$ is needed). In particular,
\[
 \Lambda_{2 }=  -  \frac{(2\pi)^2}{L^2} \Lambda^{3}+  \mathrm{OPS}(-5).
\]

 Similarly, we can set
\[
 \alpha_{2r+1}(\tau) = (e_{1}(\tau)-1) \alpha_{2r}(\tau),
\]
show next that
\[
  \widehat{\alpha}_{2r+1}(n) =   \frac{(-1)^r (2r+1)! L}{2^r(2\pi)}\frac{1}{n^{2r+1}} + {\cal O}(n^{-2r-4}).
\]
and conclude that the associated Fourier multiplier operator  satisfies
\[
 \Lambda_{2r+1}=  \frac{(2r+1)! 2^{r+1} \pi^{2r+1} }{L^{2r+1}}(-1)^r  \Lambda^{2r+2}+  \mathrm{OPS}(-2r-3).
\]
Finally, it is a well established result that periodic integral operators
\[
\mathrm{L}_{r}\varphi =  \int_0^{L} D(\cdot,\tau)\alpha_{r}(\cdot-\tau)\varphi(\tau)\,{\rm d}\tau,
\]
with $D$ a $L-$periodic smooth function,
belongs to $\mathrm{OPS}(-r-1)$ (see \cite{Saranen}).

The well posedness of the various BIE formulations considered in this paper relies heavily on the following result which provides decompositions of the Helmholtz operators $\V_k$ and $\W_k$ in sums of Fourier multiplier pseudodifferential operators of orders $-1$ and $-3$ for the first operator and $1$ and $-1$ for the second plus remainders that are smoother. Specifically, we establish
\begin{proposition}\label{prop:th:psido:exp}
 It holds
 \begin{eqnarray}
 \V_k&=& \frac12\Lambda +\frac{k^2}4\Lambda^3 +  \mathrm{OPS}(-5),\label{eq:01:prop:th:psido:exp}\\
 \W_k&=& -\frac12\Lambda^{-1} + \frac{k^2}{4}\Lambda  + \mathrm{OPS}(-3),\label{eq:02:prop:th:psido:exp}\\
 \K_k&=&  \frac{k^2}4\kappa \Lambda^{3} + \mathrm{OPS}(-4)= \K_k^\top ,\label{eq:03:prop:th:psido:exp}
 \end{eqnarray}
\end{proposition}
\begin{proof}
From the decomposition of the Bessel functions (see for instance \cite[Ch. 12]{Kress} or \cite[\S 10]{NIST:DLMF}) we get the decomposition:
\[
 \frac{i}{4} H_0^{(i)}(t) = -\frac{1}{4\pi}\log t^2+\frac{1}{16\pi}t^2\log t^2  +
 \underbrace{\frac{1}{4\pi} \frac{1 }{t^4}(1-\frac{1}4t^2-J_0(t))}_{B_0(t)} t^4\log t^2+
 C_0(t),
\]
with $B_0$ and $C_0$ smooth. Since
\[
 |{\bf x}(t)-{\bf x}(\tau)|^2 = \frac{L^2}{\pi^2}\sin^2\left(\frac{\pi}L(t-\tau)\right) + D_1(t,\tau)\sin^4\left(\frac{\pi}L(t-\tau)\right),
\]
$D_1$ being smooth, the following decomposition holds
\begin{eqnarray*}
  \frac{i}{4} H_0^{(i)}(k|{\bf x}(t)-{\bf x}(\tau)|)& =&  -\frac{1}{4\pi}\log\left(4 e^{-1}\sin^2\left(\frac{\pi}L(t-\tau)\right)\right) \\
  &&+\frac{(Lk)^2}{4} \frac{1}{(2\pi)^3}  \left(2\sin^2\frac{\pi}L(t-\tau)\right)\log\left(4 e^{-1} \sin^2 \left(\frac{\pi}L(t-\tau)\right)
  \right)
  \\
  &&+ D_2(t,\tau)\sin^4\left(\frac{\pi}L(t-\tau)\right)\log\sin^2(4 e^{-1}\pi(t-\tau))+ D_3(t,\tau),
\end{eqnarray*}
where $D_1,\ D_2$ smooth bi-periodic functions.
 Therefore,
\begin{eqnarray*}
 \V_k\varphi   &=&  \frac{1}2\Lambda  \varphi
-\frac{(Lk)^2}{4} \frac{1}{(2\pi)^2}\int_0^{L }\alpha_2( \,\cdot\,-\tau)\,\varphi(\tau)\,{\rm d}\tau\\
 &&
 +\int_0^{L }  D_2(t,\tau)\alpha_4 (\,\cdot\,-\tau)\,\varphi(\tau)\,{\rm d}\tau
 + \int_0^{L }B( \,\cdot\,,\tau)\,\varphi(\tau)\,{\rm d}\tau\\
  &=& \frac{1}2\Lambda  \varphi
 -\frac{k^2}{4}\frac{1}{(2\pi)^2}  \Lambda_2\varphi  + \mathrm{OPS}(-5)
 \\
 &=&
   \frac{1}2\Lambda  \varphi
+\frac{k^2}{4 } \Lambda  ^3 \varphi     + \mathrm{OPS}(-5).
\end{eqnarray*}
The analysis of $\mathrm{K}_k$ is very similar. Indeed, the kernel  is given by
\[
 \frac{i}4 H_1^{(1)}(k|\x-\y|)k|\x-\y|\frac{(\x-\y)\cdot\nnn(\y)}{|\x- \y|^2},
\]
which with the decomposition
\[
 \frac{i}4 H_1^{(1)}(t) t= -\frac{1}{8\pi} t^2\log t^2 + \underbrace{\frac{1}{4\pi t^2}\left(\frac{1}{2} -\frac{1}{t}J_1(t)\right)}_{E_1(t)}t^4\log t^2 +C_1(t).
\]
and the fact that
\[
 \frac{(\x-\y)\cdot\nnn(\y)}{|\x- \y|^2} = -\frac{1}2\kappa(\x) +E_2(\x,\y)|\x-\y|
\]
($E_1,\ E_2$ are again smooth functions)
allow us to   write
\begin{eqnarray*}
 \K_k\varphi   &=& -\frac{(Lk)^2}{4(2\pi)^2} \kappa \Lambda_2 \varphi
+\int_0^{L }  E_2(t,\tau)\alpha_3(\,\cdot\,-\tau)\,\varphi(\tau)\,{\rm d}\tau
 + \int_0^{L }E_3( \,\cdot\,,\tau)\,\varphi(\tau)\,{\rm d}\tau\\
 &=&   \frac{k^2}4\kappa \Lambda^3 \varphi+ \mathrm{OPS}(-4).
\end{eqnarray*}
The case $\K^\top$ is consequence of \eqref{eq:H3}-\eqref{eq:H4} since
\[
 \Lambda(\kappa \cdot)= \kappa
 \Lambda +\mathrm{OPS}(-2)
\]
which with the fact $\Lambda^\top = \Lambda$.

Finally, for the hypersingular operator, we note that with the identity $\W_k\V_k = -\frac{1}4\I+(\K_k^\top)^2$ cf. \eqref{eq:calderonIdentitiesv2},
\begin{eqnarray*}
 \W_k &=& \W_k\left(\I-\frac{k^4}{4} \Lambda^4\right) +
 \mathrm{OPS}(-3)\\
 &=&\W_k\left(\frac12\Lambda+\frac{k^2}4 \Lambda^3\right)\left(2\Lambda^{-1}-k^2 \Lambda\right)
 +
 \mathrm{OPS}(-3)\\
 &=& \W_k \V_k\left(2\Lambda^{-1}-k^2 \Lambda\right)
 +\underbrace{\W_k\left(\frac12\Lambda+\frac{k^2}4 \Lambda^3-\V_k\right)\left(2\Lambda^{-1}-k^2 \Lambda\right)}_{\in
 \mathrm{OPS}(-3)}
+
 \mathrm{OPS}(-3)\\
 &=&-\frac{1}4   \left(2\Lambda^{-1}-k^2 \Lambda\right)
 +\underbrace{(\K_k^\top)^2 \left(2\Lambda^{-1}-k^2 \Lambda\right)}_{\in \mathrm{OPS}(-5)}+ \mathrm{OPS}(-3)\\
 &=&   -\frac{1}2 \Lambda^{-1}+\frac{k^2}4 \Lambda
 +
 \mathrm{OPS}(-3)
\end{eqnarray*}
and the result is proven.
\end{proof}
\begin{remark}\label{remark:prop:th:psido:exp}
In the published version of this paper the constant in the expansion of $\K_k$ (see \eqref{eq:03:prop:th:psido:exp}, \eqref{eq:03:th:psido:exp}, \eqref{eq:psido:appendix:01g} and Proposition~\ref{prop:principalPart:Grad}) was misprinted as $k^2\kappa/2$ instead of $k^2\kappa/4$; none of the results of Section~\ref{BIE} depends on it, since these terms are absorbed in the remainders.
This result appears in a slightly different form~\cite{hsiao2008boundary} (equations (10.4.5)). Indeed, several terms in the asymptotic expansion of the
principal symbol of the periodic pseudodifferential operator $\V_k$ are provided in equations (10.4.5) in~\cite{hsiao2008boundary}, and they coincide with the symbols of the operators in the decomposition we provide in Proposition~\ref{prop:th:psido:exp}. Indeed,

 It is also evident that the expansions in the four operators for Helmholtz equation can be continued  in  powers of $\Lambda_0$ (or equivalently) $\D\HH$), all of them being negative except the first term  for $\W_k$. Such as expansions have appeared previously in the literature (see for instance \cite{Saranen} and, with applications to the study and design of numerical methods, in  \cite{MR2013967,MR1912906}).
\end{remark}

\begin{theorem}\label{th:psido:exp}
 It holds
 \begin{eqnarray}
 \V_k&=& \V_0+ 2k^2\V_0^3 +  \mathrm{OPS}(-5)=\frac{1}2\HH\D_{-1}-\frac{k^2}4\HH\D_{-3}  +  \mathrm{OPS}(-5) \label{eq:01:th:psido:exp},\\
 \W_k&=& \W_0  + \frac{k^2}{2}\V_0  + \mathrm{OPS}(-3)= \frac{1}2\HH\D_{ 1}+\frac{k^2}4\HH\D_{-1}+  \mathrm{OPS}(-3) \label{eq:02:th:psido:exp},\\
 \K_k&=&  \K_k^\top =  2 k^2 \kappa \V_0^{3} + \mathrm{OPS}(-4)= -\frac{1}4\kappa k^2 \HH\D_{-3} + \mathrm{OPS}(-4) \label{eq:03:th:psido:exp}
 \end{eqnarray}
Equivalently,
\begin{eqnarray}
 \V_k\varphi &=& \frac{L}{4\pi}\sum_{n\ne k}  \left(n^2-\left(\frac{k L}{2\pi}\right)^2\right)^{-1/2} \widehat{\varphi}(n)e_n + \mathrm{OPS}(-5)\label{eq:04:th:psido:exp},\\
 \W_k\varphi &=& -\frac{\pi}{L}\sum_{n} \left(n^2-\left(\frac{k L}{2\pi}\right)^2\right)^{1/2} \widehat{\varphi}(n)e_n + \mathrm{OPS}(-3)\label{eq:05:th:psido:exp}.
\end{eqnarray}
Finally, the Dirichlet-to-Neumann operator satisfies
\begin{equation}\label{eq:06:th:psido:exp}
 \begin{aligned}
 \mathrm{DtN}_k\varphi& =  \HH\D+\frac{k^2}2\HH\D_{-1}  -\kappa k^2 \mathrm{D}_{-2}+ \mathrm{OPS}(-4)\\
 &\ =\  \frac1{L} \sum_{n} \left(n^2-\left(\frac{k L}{2\pi}\right)^2\right)^{1/2}  \widehat{\varphi}(n)e_n +\mathrm{OPS}(-2) = 2\W_k +\mathrm{OPS}(-2).
\end{aligned}
\end{equation}
\end{theorem}
\begin{proof}
 Expansions \eqref{eq:01:th:psido:exp}-\eqref{eq:03:th:psido:exp} follows from Theorem \ref{prop:th:psido:exp} whereas  Properties \eqref{eq:04:th:psido:exp}-\eqref{eq:05:th:psido:exp} follows from
 \[
  (n^2-k^2)^{-1/2} = \frac{\sign(n) i }{i n} -\frac{\sign(n) i }{(in)^3}   +{\cal O}(n^{-5}),\quad
  -(n^2-k^2)^{1/2} = (\sign(n) i)(i n) +\frac{\sign(n) i }{ in }   +{\cal O}(n^{-3}).
 \]
Finally, if $\V_k$ is invertible, using \eqref{eq:calderonIdentitiesv2}
\begin{eqnarray*}
 \mathrm{DtN}_k &=&\V_k^{-1} (-\frac{1}{2}\I+\K_k) = -4 \W_k(-\frac12\I+\K_k) +\mathrm{OPS}(-5) \\
 &=&  2\W_k -4\W_k\K_k+\mathrm{OPS}(-5).
\end{eqnarray*}
Now \eqref{eq:06:th:psido:exp} is straightforward to derive. If $\V_k$ fails to be invertible we can use the alternative
expression for the Dirichlet-to-Neumann operator
\[
  \mathrm{DtN}_k = \left(\frac{1}{2}\I+\K_k^\top\right)^{-1}\W_k
\]
(notice that at least one of  $\V_k$ or $(\frac{1}{2}\I+\K_k^\top)$ must be invertible) and proceed similarly.
\end{proof}

\begin{lemma}\label{lemma:ntDnt}
 It holds
 \begin{subequations}
\begin{eqnarray}
 \nnn\cdot \W_k(\nnn\cdot)  =  \ttt\cdot \W_k(\ttt\cdot)   &=& -\frac{1}2 \Lambda^{-1} +\frac{k^2}4\Lambda +\mathrm{OPS}(-3), \\
 \ttt\cdot \W_k(\nnn\cdot)  = -\nnn\cdot \W_k(\ttt\cdot)   &=&  \frac{1}2\kappa\D\Lambda+\frac{k^2}4\kappa \D\Lambda^3+\mathrm{OPS}(-3),\\
 \nnn\cdot \V_k(\nnn\cdot)=\ttt\cdot \V_k(\ttt\cdot)       &=& \frac{1}2  \Lambda+\frac14(2\kappa+k^2)   \Lambda^3+\mathrm{OPS}(-4),\ \\
 \ttt\cdot \V_k(\nnn\cdot)  =  \nnn\cdot \V_k(\ttt\cdot)   &=& -\frac{1}2 \kappa   \D\Lambda^3+\mathrm{OPS}(-4), \  \\
 \nnn\cdot \D\V_k(\nnn\cdot)=\ttt\cdot \D\V_k(\ttt\cdot)   &=& \frac{1}2\D \Lambda   +\frac{k^2}4  \D\Lambda^3+\mathrm{OPS}(-4), \ \\
 \ttt\cdot \D\V_k(\nnn\cdot)=-\nnn\cdot \D\V_k(\ttt\cdot)  &=& -\frac{k^2}2\kappa  \Lambda ^3 +\mathrm{OPS}(-4).\\
\end{eqnarray}
\end{subequations}
\end{lemma}
\begin{proof} We start from the identities
  \begin{equation}
   \begin{aligned}
   {\nnn}\cdot\D_2 \nnn =    {\bm t}\cdot\D_2 {\bm t} \ &=\   \D_2- \kappa^2\I,\\
   {\nnn}\cdot\D_2 {\bm t} = - {\bm t}\cdot\D_2 \nnn  \ &=\   -2\kappa\D + \kappa'\, \I, \ \\
   {\nnn}\cdot\D \nnn = -   {\bm t}\cdot\D {\bm t}    \ &=\   \D, \ \\
   {\bm t}\cdot\D \nnn = -
   {\nnn}\cdot\D {\bm t} \ &=\   \kappa \I,
   \ \\
   {\nnn}\cdot\D_{-1} \nnn =
   {\bm t}\cdot\D_{-1} {\bm t} \ &=\   \D_{-1}-\kappa\D_{-3} +
   (1+\kappa)(\mathrm{D}\kappa)\D_{-4}+\mathrm{OPS}(-5),
   \ \\
   {\bm t}\cdot\D_{-1} \nnn = -
   {\nnn}\cdot\D_{-1} {\bm t} \ &=\   -\kappa \D_{-2}+\kappa'\, \D_{-3}+
    (\mathrm{D}^2\kappa-\kappa^2)\D_{-4}+\mathrm{OPS}(-5),
    \ \\
   {\nnn}\cdot\D_{-2} \nnn =
   {\bm t}\cdot\D_{-2} {\bm t} \ &=\   \D_{-2}+3\kappa^2\D_{-4} + \mathrm{OPS}(-5),
   \ \\
   {\bm t}\cdot\D_{-2} \nnn = -
   {\nnn}\cdot\D_{-2} {\bm t} \ &=\   -2\kappa \D_{-3}-3\kappa'\, \D_{-4}+ \mathrm{OPS}(-4)
   \end{aligned}
  \end{equation}
  which can be easily proven from \eqref{eq:ttnn:appendix}.  Hence the result follows from
  Proposition \ref{prop:th:psido:exp} and the commutation properties for $\HH$ cf. \eqref{eq:H3}, the Leibnitz rule for the derivatives and its extension cf. \eqref{eq:Leibnitz:rule} for the {\em negative order derivatives}  $\D_{-r}$.
\end{proof}

\begin{theorem}\label{theo:ntDnt:02}
 It holds
 \begin{subequations}\label{eq:psido}
\begin{eqnarray}
  \nnn\cdot \W_k(\nnn\cdot)  =  \ttt\cdot \W_k(\ttt\cdot)  &=& \W_k +\mathrm{OPS}(-3)\nonumber\\
  &=& \frac{1}2\D\HH +\frac{k^2}4\HH\D_{-1}+\mathrm{OPS}(-3) \label{eq:psido:appendix:01a}  \\
  \ttt\cdot \W_k(\nnn\cdot)  = -\nnn\cdot \W_k(\ttt\cdot)  &=&  \kappa \D \V_k +\mathrm{OPS}(-3)\nonumber\\
  &=& \frac{1}2\kappa \HH -\frac{k^2}4\kappa \D_{-2}\HH +\mathrm{OPS}(-4)\label{eq:psido:appendix:01b}\\
  \nnn\cdot \V_k(\nnn\cdot)  =  \ttt\cdot \V_k(\ttt\cdot)      &=&  \V_k+4\kappa \V_k^3+\mathrm{OPS}(-4) \nonumber\\
  &=&\frac12\HH\D_{-1}-\frac{1}4(k^2+2\kappa)\D_{-3}\HH + \mathrm{OPS}(-3) \label{eq:psido:appendix:01c} \\
  \ttt\cdot \V_k(\nnn\cdot)  = -\nnn\cdot \V_k(\ttt\cdot)  &=&  -4\kappa \V_k ^3+\mathrm{OPS}(-4) \nonumber\\
  &=& \frac{1}{2}\kappa \D_{-3}\HH+\mathrm{OPS}(-4) \ \label{eq:psido:appendix:01d} \\
  \nnn\cdot \D\V_k(\nnn\cdot)=  \ttt\cdot \D\V_k(\ttt\cdot)  &=&   \D\V_k +\mathrm{OPS}(-4) \nonumber\\
  &=&
   \frac{1}2  \HH -\frac{k^2}4  \D_{-2}\HH +\mathrm{OPS}(-4)\label{eq:psido:appendix:01e}\  \\
  \ttt\cdot \D\V_k(\nnn\cdot)=- \nnn\cdot \D\V_k(\ttt\cdot) &=&  -4 k^2  \kappa  \D\V_k^3 +\mathrm{OPS}(-4) \nonumber\\
  &=& \frac{1}{2}k^2\kappa \HH\D_{-2}+\mathrm{OPS}(-4)
  \label{eq:psido:appendix:01f}\\
  \nnn\cdot \K_k(\nnn\cdot)  =  \ttt\cdot \K_k(\ttt\cdot)     &=&    2 k^2  \kappa  \V_k^3 +\mathrm{OPS}(-4)
 \nonumber\\
  &=&
  - \frac{1}{4}k^2\kappa \HH\D_{-3}+\mathrm{OPS}(-4)
   \ \label{eq:psido:appendix:01g}\\
  \ttt\cdot \K_k(\nnn\cdot)  = -\nnn\cdot \K_k(\ttt\cdot) &=&  \mathrm{OPS}(-4) \label{eq:psido:appendix:01h}.
\end{eqnarray}
\end{subequations}
\end{theorem}
\begin{proof}
 It is consequence of Lemma \ref{lemma:ntDnt} and Proposition \ref{prop:th:psido:exp}.
\end{proof}

\end{document}